\documentclass[11pt]{amsart}

\usepackage{amsmath}
\usepackage{color}
\definecolor{skyblue}{RGB}{0,0,128}
\usepackage{amssymb}
\usepackage{bm}
\usepackage{slashbox}
\usepackage{booktabs,multirow}
\usepackage{dsfont}
\usepackage{pstricks}
\usepackage{amsfonts}
\usepackage{graphicx,subfigure,float,epstopdf}
\usepackage{algorithmic}
\ifpdf
  \DeclareGraphicsExtensions{.eps,.pdf,.png,.jpg}
\else
  \DeclareGraphicsExtensions{.eps}
\fi

\newtheorem{theorem}{Theorem}[section]
\newtheorem{cor}[theorem]{Corollary}

\newtheorem{prop}[theorem]{Proposition}

\newtheorem{assum}[theorem]{Assumption}
\theoremstyle{definition}

\newtheorem{rem}[theorem]{Remark}

\numberwithin{equation}{section}

\usepackage{amsopn}

\graphicspath{{fig/}}
\allowdisplaybreaks

\begin{document}

% Title. If the supplement option is on, then "Supplementary Material"
% is automatically inserted before the title.
\title[Inertial Proximal ADMM and Compressive Affine Phase Retrieval]{Inertial Proximal ADMM for Separable Multi-Block Convex Optimizations and  Compressive Affine Phase Retrieval} %{\TheTitle}}

% Authors: full names plus addresses.
\author{
Peng Li}
 \address{Peng Li:\ School of Mathematics and Statistics, Lanzhou University,  Lanzhou , Gansu, China;
   Graduate School, China Academy of Engineering Physics, Beijing 100088, China}
   %  \address{ Graduate School, China Academy of Engineering Physics, Beijing 100088, China}
  %; 2. Department of Mathematics, City University of Hong Kong, Hong Kong, China (
  \email{lipeng16@gscaep.ac.cn}

\author{  Wengu Chen}
   \address{Wengu Chen: \ Institute of Applied Physics and Computational Mathematics, Beijing 100088, China}
      \email{chenwg@iapcm.ac.cn}
  \author{  Qiyu Sun}
  \address{%Corresponding author.
  Qiyu Sun: \ Department of Mathematics, University of Central Florida, Orlando, FL 32816, USA}
  \email{qiyu.sun@ucf.edu}
\thanks{The project is partially supported by the Natural Science Foundation of China (No. 11871109), the NSAF (Grant
No. U1830107), the Science  Challenge Project (TZ2018001) and the National Science Foundation (DMS 1816313).}

\maketitle

%%% Local Variables:
%%% mode:latex
%%% TeX-master: "ex_article"
%%% End:

% REQUIRED
\begin{abstract}
Separable multi-block convex optimization problem
appears in  many  mathematical and engineering fields.
In the first part of this paper, we propose an inertial proximal ADMM  to  solve a  linearly constrained separable multi-block convex optimization problem,
and we show that the proposed inertial proximal ADMM has global convergence under mild assumptions on the regularization matrices.
Affine phase retrieval  arises in
holography, data separation and phaseless sampling,
and it is also considered as a nonhomogeneous version
of  phase retrieval that has received considerable attention in recent
years.
Inspired by  convex relaxation of
 vector sparsity and  matrix rank  in compressive sensing and
by phase lifting  in phase retrieval,
in the second part of this paper, we introduce a compressive affine phase retrieval via lifting  approach to connect affine phase retrieval with
multi-block convex optimization, and then based on the proposed
inertial proximal ADMM for multi-block convex optimization, we propose an  algorithm to recover sparse real signals from their (noisy) %corrupted/uncorrupted
affine quadratic measurements.
Our numerical simulations  show that the proposed algorithm  has satisfactory performance for affine phase retrieval of sparse  real signals.

\end{abstract}

% REQUIRED
%\begin{keywords}
%inertial proximal ADMM, separable multi-block convex optimization,   affine phase retrieval
%\end{keywords}

%% REQUIRED
%\begin{AMS}
%65K10, 94A12, 42C15, 52A41.
%\end{AMS}

\section{ Introduction}\label{s1}

In the first part of  this paper, we consider the following linearly constrained  separable multi-block convex optimization, %with linear constraint,
\begin{equation}\label{Multiseparatedoperator}
\min_{{\bf x}_j\in \mathcal{X}^j, 1\le j\le l}\ \  \sum_{j=1}^{l}f_j({\bf x}_j)  \ \ {\rm subject \ to }\ \
\sum_{j=1}^{l}{\bf A}_j{\bf x}_j={\bf c},
\end{equation}
 where  ${\bf A}_j\in\mathbb{R}^{m\times n_j}$,  ${\bf c}\in\mathbb{R}^m$,
  $\mathcal{X}^j $ are closed convex sets in $\mathbb{R}^{n_j}$ and
  $f_j:\mathbb{R}^{n_j}\rightarrow (-\infty,\infty)$
are closed convex functions on  $\mathbb{R}^{n_j}, 1\le j\le l$.
 The above minimization  problem  appears in  machine learning, statistics,
 signal and image processing,
%and many more fields \cite{CNZ2015, MXAP2015, PGWXM2012, YZ2011}.
and many more fields \cite{MXAP2015, PGWXM2012, YZ2011}.
Denote  the standard
inner product and norm on the  Euclidean space by
$\langle \cdot, \cdot\rangle$ and $\|\cdot\|_2$   respectively.
A  conventional approach to the  convex optimization problem  \eqref{Multiseparatedoperator}
is the alternating direction method of multipliers (ADMM)
with initial
$({\bf x}_1^0,\ldots,{\bf x}_l^0; {\bf z}^0)\in {\mathcal W}:=\mathcal{X}^1\times\cdots\times\mathcal{X}^l\times\mathbb{R}^m$
chosen appropriately or randomly, and with update in each iteration by
\begin{subequations}\label{DirectiveExtendADMM}
\begin{equation}\label{DirectiveExtendADMM-1}
{\bf x}_1^{k+1}\in \arg\min_{{\bf x}_1\in \mathcal{X}^1} \mathcal{L}_{\beta}({\bf x}_1,{\bf x}_2^{k},\ldots,{\bf x}_l^{k};{\bf z}^k),
%+\|{\bf x}^1-{\bf x}_1^k\|_{{\bf H}_1}^2/2:
\end{equation}
\vskip-0.18in
\begin{equation}\label{DirectiveExtendADMM-3}
{\bf x}_i^{k+1}\in \arg\min_{{\bf x}_i\in \mathcal{X}^i} \mathcal{L}_{\beta}({\bf x}_1^{k+1}, %{\bf x}_2^{k+1},
\ldots,{\bf x}_{i-1}^{k+1},{\bf x}_i,
{\bf x}_{i+1}^k,\ldots,{\bf x}_l^k;{\bf z}^{k}), \ i=2, \ldots, l,
%+\|{\bf x}_i-{\bf x}_i^k\|_{{\bf H}_i}^2/2:
\end{equation}
\vskip-0.18in
\begin{equation}\label{DirectiveExtendADMM-4}
{\bf z}^{k+1}={\bf z}^k-\beta\Big(\sum_{j=1}^l{\bf A}_j{\bf x}_j^{k+1}-{\bf c}\Big),
\end{equation}
\end{subequations}
where
\begin{equation}\label{AugmentedLagrangefunction1}
\mathcal{L}_{\beta}({\bf x}_1,\ldots,{\bf x}_l;{\bf z})
:=\sum_{j=1}^lf_j({\bf x}_j)-\Big\langle {\bf z},\sum_{j=1}^l{\bf A}_j{\bf x}_j-{\bf c} \Big\rangle
+\frac{\beta}{2}\Big\|\sum_{j=1}^l{\bf A}_j{\bf x}_j-{\bf c}\Big\|_2^2
\end{equation}
is the  augmented Lagrange function with Lagrange multiplier
 ${\bf z}\in\mathbb{R}^m$  and  penalty parameter $\beta>0$.

The   ADMM algorithm with $l=2$  was introduced in the 1970s  %\cite{GM1976}
and its convergence
%has been well studied \cite{CST2017, GM1976,G1984}.
has been well studied \cite{GM1976,G1984}.
For  $l\ge 3$, the  multi-block ADMM \eqref{DirectiveExtendADMM} works very well for many concrete applications
\cite{BLXZ2018, HYZ2014, PGWXM2012,TY2011}, however it
may  not  converge   without additional information on the objective functions $f_j$ and  constraint matrices ${\bf A}_j, 1\le j\le l$
\cite{CHYY2016}. For instance,  Han and Yuan \cite{HY2012}
showed
that the scheme \eqref{DirectiveExtendADMM} is convergent if all the objective functions $f_j, 1\le j\le l$, are strongly convex and the penalty parameter $\beta$ is chosen in a certain range. The above  strongly convex condition on the objective functions is relaxed  in
\cite{LMZ2015-JORSC}
that not all functions in the objective are required
to be strongly convex.
 For general multi-block convex problems,
%a convex linear programming problem,
 many convergent proximal variants
of the multi-block ADMM \eqref{DirectiveExtendADMM} have been proposed to overcome the divergence issue,
including  the proximal parallel splitting method %(PPSM)
\cite{HHX2014},
the Jacobi-Proximal ADMM  %(Prox-JADMM)
\cite{DLPY2017} and the twisted version of the proximal ADMM  %(TPADMM)
\cite{WS2017}.
The reader may refer  to
%the survey papers \cite{BPCBJ2011,EY2015, G2014}
the survey paper \cite{G2014}
for additional historical remarks and recent advances on the ADMM and its   variations.

In this paper, we introduce an inertial proximal ADMM to solve  the multi-block convex optimization problem  \eqref{Multiseparatedoperator},  Prox-IADMM for abbreviation, with
initial value
 $(\textbf{x}_1^0,\ldots,{\bf x}_l^0;{\bf z}^0)$  chosen appropriately or randomly
in ${\mathcal W}$,
and with
 update in each iteration given by the following:
\begin{subequations}\label{MultiproximalADMM}
\begin{equation}\label{MultiproximalADMM-0}
(\bar{{\bf x}}_1^k,\ldots,\bar{{\bf x}}_l^k;\bar{{\bf z}}^k)
=({\bf x}_1^k,\ldots,{\bf x}_l^k;{\bf z}^k)
+\alpha_k({\bf x}_1^k-{\bf x}_1^{k-1},\ldots,{\bf x}_l^k-{\bf x}_l^{k-1};{\bf z}^k-{\bf z}^{k-1}),
\end{equation}
\vskip-0.18in
\begin{equation}\label{MultiproximalADMM-1}
{\bf x}_1^{k+1}\in \arg\min_{{\bf x}_1\in \mathcal{X}^1}\mathcal{L}_{\beta}({\bf x}_1,\bar{{\bf x}}_2^k,\ldots,\bar{{\bf x}}_l^k;\bar{{\bf z}}^k)
+\frac{1}{2}({\bf x}_1-\bar{{\bf x}}_1^k)^T{\bf H}_1 ({\bf x}_1-\bar{{\bf x}}_1^k), %\nonumber\\
\end{equation}
\vskip-0.18in
\begin{equation}\label{MultiproximalADMM-2}
{\bf z}^{k+1}=\bar{{\bf z}}^k
-\beta\big({\bf A}_1{\bf x}_1^{k+1}+\sum_{j=2}^l{\bf A}_j\bar{{\bf x}}_j^k-{\bf c}\big),
\end{equation}
\vskip-0.18in
\begin{eqnarray}\label{MultiproximalADMM-4}
\qquad {\bf x}_i^{k+1} &\hskip-0.08in  \in  &
\hskip-0.08in \arg\min_{{\bf x}_i\in \mathcal{X}^i}\big(\mathcal{L}_{\beta}({\bf x}_1^{k+1},\bar{{\bf x}}_2^k, \ldots,\bar{{\bf x}}_{i-1}^k, {\bf x}_{i}, \bar{{\bf x}}_{i+1}^k
\ldots,\bar{{\bf x}}_l^k;{\bf z}^{k+1})\\
\qquad & & \qquad\qquad\quad
+\frac{1}{2}({\bf x}_i-\bar{{\bf x}}_i^k)^T{\bf H}_i({\bf x}_i-\bar{{\bf x}}_i^k)\big),\  \ 2\le i\le l, \nonumber
\end{eqnarray}
\end{subequations}
where $({\bf x}_1^{-1},\ldots,{\bf x}_l^{-1};{\bf z}^{-1})=({\bf x}_1^0,\ldots,{\bf x}_l^0;{\bf z}^0)$, $\alpha_k,~k\geq 0$ are step sizes, ${\mathcal L}_\beta$ is the  augmented Lagrange function  in  \eqref{AugmentedLagrangefunction1},
and ${\bf H}_j, 1\le j\le l$, are regularization matrices.
%for  $2\le i\le l$.
Our illustrative examples of regularization matrices % ${\bf H}_1, \ldots, {\bf H}_l$
are  prox-linear matrices
\begin{equation}\label{prox-linear.def}
{\bf H}_j=\beta{\bf I}/{\eta_j}-\beta{\bf A}_j^{T}{\bf A}_j,\  1\le j\le l, \end{equation}
and standard proximal matrices
 \begin{equation}\label{standardproximal.def}
 {\bf H}_j= \beta {\bf I}/{\eta_j}, % \|{\bf A}_j\|_{2\to 2}^2/\eta_j
\  1\le j\le l,\end{equation}
where $\eta_j, 1\le j\le l$, are positive numbers, see
 \cite{DY2016} for additional regularization matrices  of interest.
The  Prox-IADMM  \eqref{MultiproximalADMM} extends the inertial proximal ADMM for two-block convex optimization in \cite{CCMY2015} nontrivially,
it mixes  the Jacobi method  in \cite{DLPY2017, HHY2015} and Gauss-Seidel method in \cite{CHYY2016, HTY2012},
and it  also yields  a new inertial variant of the ADMM
for multi-block convex optimization when all regularization matrices are set to be  zero.
In the first part of this paper, we establish the convergence of its inertial version
under the assumption that
${\bf H}_1$ and ${\bf H}_j-\beta (l-2) {\bf A}_j^T{\bf A}_j,  2\le j\le l$, are positive definite,
see  Theorem \ref{IADMMConvergence3.mainthm}. This is  a nontrivial extension of
 the convergence result in \cite{CCMY2015} where  $l=2$ and matrices
  ${\bf H}_1$ and ${\bf H}_2$  are under a strong assumption that
             ${\bf H}_1$ and ${\bf H}_2-\beta{\bf A}_2^T{\bf A}_2$ are positive definite.

\medskip

In the second part of this paper, we  consider
recovering {\em sparse} real vectors ${\bf x}\in \mathbb{R}^n$ from their affine quadratic measurements
\begin{equation}\label{APRModel1-}
\bar{{\bf b}}: =|{\bf A} {\bf x}+{\bf b}|^2=
\big(|{\bf a}_1^T{\bf x}+b_1|^2,\ldots,  |{\bf a}_m^T{\bf x}+b_m|^2\big)^T,
\end{equation}
where ${\bf A}=[{\bf a}_1,\ldots,{\bf a}_m]^T$ is a measurement matrix and
${\bf b}=(b_1, \ldots, b_m)^T$ is a reference vector.
%The above affine phase retrieval problem arises in holography, data separation and phaseless sampling
%\cite{CMVG2010,  CCSW2016, DK2013, LBCMDU2003,  LLS2013}.
 The above affine phase retrieval problem arises in holography \cite{LBCMDU2003}, data separation \cite{DK2013,LLS2013}, phaseless sampling \cite{CCSW2020}, phase retrieval with background information \cite{ELB2018, YW2019}, %(called ``molecular replacement" in \cite{ELB2018}),
  and phase retrieval with reference signal \cite{AA2020, BSLLC2019, BP2015, HCA2020,  HHA2019}.
% In phaseless sampling and reconstruction of temporal signals, one wants to recover a signal ${\bf x}$ at a time period for its phaseless measurements $|{\bf A}{\bf x}+{\bf B}{\bf y}|$, where ${\bf y}$ is a known temporal signal on the prior time period  \cite{CCSW2016}.
A sufficient and necessary condition on  the pair $({\bf A}, {\bf b})$ of measurement matrix and reference vector is introduced in
\cite{CCS19, GSWX2018} so that any (sparse) real  vector ${\bf x}$ is  uniquely determined by its affine quadratic measurements  $|{\bf A}{\bf x}+{\bf b}|^2$ in \eqref{APRModel1-}. However
the  reconstruction  of the sparse real vector ${\bf x} \in \mathbb{R}^n$ from its affine quadratic measurements
is highly nonlinear and notoriously difficult to solve numerically and stably.
Observe that  affine quadratic measurements  in \eqref{APRModel1-}
 is the same as the quadratic measurements of the vector $\tilde {\bf x}\in \mathbb{R}^{n+1}$ via
 the measurement matrix $\tilde {\bf A}=[\tilde {\bf a}_1,\ldots,\tilde {\bf a}_m]^T$,
\begin{equation}\label{affinePRvsPR} |{\bf A}{\bf x}+{\bf b}|^2= |\tilde   {\bf A}\tilde {\bf x}|^2,
\end{equation}
where
$ \tilde {\bf x}=\left [ \begin{array}{c} {\bf x}\\
1\end{array}\right] \ {\rm  and} \  \tilde{{\bf a}}_i= \left [ \begin{array}{c} {\bf a}_i\\
b_i\end{array}\right], \ 1\le i\le m.
$
Then a conventional approach for (sparse) affine phase retrieval is to recover the sparse real vector $\tilde {\bf x}$ from its  quadratic measurements
 in \eqref{affinePRvsPR} by applying available iterative reconstruction algorithms in phase retrieval,
% such as alternating minimization  \cite{NJS2015},
%semidefinite programming \cite{CSV2013,CCG2015,  LV2013, OYDS2012, YX2015},
%     Wirtinger flow approach \cite{ CLM2016, CLS2015} and Gauss-Newton algorithm \cite{GX2017}
 such as alternating minimization  \cite{NJS2015},  semidefinite programming \cite{CSV2013, LV2013, OYDS2012}
     and Wirtinger flow approach \cite{ CLM2016, CLS2015}
with additional normalization to the last component of the reconstructed vector %to $1$
      in each iteration.
%{\color{blue}
We observe
that there are some space for the improvement
on the performance of those conventional approaches
to  sparse affine phase retrieval,  see Subsections \ref{noiseless.simulation}--\ref{boundednoise.simulation}.
 % \ref{Numerical.Section}.
In the second part of this paper,   we apply the  inertial Prox-ADMM scheme
 and propose the  CAPReaL algorithm     to reconstruct sparse
real signals from their (noisy) affine quadratic measurements.

Define the $\ell_0$ norm  $\|{\bf x}\|_0$ (resp. $\|{\bf X}\|_0$) of  a vector ${\bf x}$  (resp.  a matrix ${\bf X}$)
 by the number of its nonzero entries.
Set ${\bf X}={\bf x} {\bf x}^T$ for  a real $s$-sparse vector  ${\bf x}$, i.e., $\|{\bf x}\|_0\le s$.
 Then
${\bf X}$ is a positive semi-definite matrix with rank at most one and its $\ell^0$ norm $\|{\bf X}\|_0$ is  no larger than $s^2$. Moreover
the affine quadratic measurements of ${\bf x}$  in \eqref{APRModel1-}
are affine measurements of ${\bf x}$ and ${\bf X}$,
\begin{equation}\label{APRModel1}
\bar{{\bf b}} =
\mathcal{A}({\bf X})
+{\bf B}{\bf x}+ |{\bf b}|^2,
\end{equation}
where $\mathcal{A}:\mathbb{R}^{n\times n}\ni {\bf X}\longmapsto
 (\langle {\bf a}_1{\bf a}_1^{T},{\bf X}\rangle, \ldots,  \langle {\bf a}_m{\bf a}_m^{T},{\bf X}\rangle)^T\in
\mathbb{R}^m$ is a linear map, ${\bf B}=2[{b}_1{\bf a}_1, \ldots,  {b}_m {\bf a}_m]^{T}$   and $|{\bf b}|^2=(|{ b}_1|^2, \ldots, |{b}_m|^2)^T$.
Therefore our recovery problem
reduces to finding  a real signal ${\bf x}$ with minimal  $\ell^0$ norm
 and a positive semi-definite matrix ${\bf X}$ with minimal rank and $\ell^0$ norm,
 \begin{equation}\label{AffinePhaseLiftModel}
\text{min}_{{\bf x}, {\bf X}\succeq {\bf 0}}\  \|{\bf x}\|_0,  \  {\rm rank}({\bf X}) \ {\rm and} \ \|{\bf X}\|_0 \ {\rm subject\ to} \
\mathcal{A}({\bf X})+{\bf B}{\bf x}={\bf c} \  {\rm and}\ {\bf X}={\bf x}{\bf x}^{T},
\end{equation}
where ${\bf c}=\bar{{\bf b}}-|{\bf b}|^2$.
Inspired by the lifting technique \cite{CESV2013-2015} for phase retrieval and
 the convex relaxation for rank of matrices  and sparsity of matrices/vectors \cite{CLMW2010, CRT2006, RFP2010},
we consider heuristically nuclear norm convex relaxation of matrix rank and  $\ell^1$-norm convex relaxation of
vector/matrix sparsity in \eqref{AffinePhaseLiftModel}. This leads to the following multi-convex relaxation to solve the compressive affine phase retrieval problem \eqref{AffinePhaseLiftModel}:
\begin{subequations}\label{CompressedAffinePhaseLift-Biconvex-Equivalent}
\begin{equation}\label{CompressedAffinePhaseLift-Biconvex-Equivalent.a}
\min_{{\bf X}\succeq {\bf O},{\bf Y}\in\mathbb{R}^{n\times n},
{\bf x}\in\mathbb{R}^n}~\text{tr}({\bf X})+\tau\|{\bf Y}\|_1+\lambda\|{\bf x}\|_1
\end{equation}
\begin{equation} \label{CompressedAffinePhaseLift-Biconvex-Equivalent.b}
\text{subject  \ to}\quad \frac{1}{2}\mathcal{A}({\bf X})+\frac{1}{2}\mathcal{A}({\bf Y})
+{\bf B}{\bf x}={\bf c},\
%\end{equation}
%\begin{equation} \label{CompressedAffinePhaseLift-Biconvex-Equivalent.d}
{\bf X}-{\bf Y}={\bf O} \ \ {\rm and} \
\end{equation}
\begin{equation}\label{CompressedAffinePhaseLift-Biconvex-Equivalent.c}
{\bf Y}={\bf x}{\bf x}^{T},
\end{equation}
\end{subequations}
where $\tau>0$ and $\lambda>0$ are balance parameters.
We call the above model \eqref{CompressedAffinePhaseLift-Biconvex-Equivalent} as Compressive Affine Phase Retrieval via Lifting (CAPReaL).

Denote by  $\mathcal{I}_n:\mathbb{R}^{n\times n} \rightarrow\mathbb{R}^{n\times n}$ the identity operator on $\mathbb{R}^{n\times n}$. Without imposing
  the constraint  ${\bf Y}={\bf x}{\bf y}^{T}$ in
  \eqref{CompressedAffinePhaseLift-Biconvex-Equivalent.c}, the proposed CAPReaL model becomes
  \begin{subequations} \label{CompressedAffinePhaseLift-Equivalent}
  \begin{equation}\label{CompressedAffinePhaseLift-Equivalent.a}
\min_{{\bf X}\succeq {\bf O},{\bf Y}\in\mathbb{R}^{n\times n},{\bf x}\in\mathbb{R}^n}\text{tr}({\bf X})+\tau\|{\bf Y}\|_1+\lambda\|{\bf x}\|_1
\end{equation}
\vskip-0.1in
\begin{equation}\label{CompressedAffinePhaseLift-Equivalent.b}
\text{subject \ to}\quad \frac{1}{2}\mathcal{A}({\bf X})+\frac{1}{2}\mathcal{A}({\bf Y})
+{\bf B}{\bf x}={\bf c}
 \  \  {\rm and}\ \ {\bf X}-{\bf Y}={\bf O},
\end{equation}
\end{subequations}
which is
a linearly constrained  separable  {\bf $3$-block} convex optimization problem %with linear constraint,
\eqref{Multiseparatedoperator} with  ${\bf x}_1={\bf x}, {\bf x}_2={\bf X}, {\bf x}_3={\bf Y}$, ${\bf A}_1=[{\bf B}; {\bf O}]$, ${\bf A}_2=[\mathcal{A}/2;  ~\mathcal{I}_n]$ and ${\bf A}_3=[\mathcal{A}/2;  ~-\mathcal{I}_n]$.
In Section \ref{s4.section},
 we apply the inertial proximal ADMM to solve  \eqref{CompressedAffinePhaseLift-Equivalent} and then take few more steps to
   compensate the  relaxation of the constraints \eqref{CompressedAffinePhaseLift-Biconvex-Equivalent.c}.
   Numerical simulations in Section
   \ref{Numerical.Section} show that the proposed algorithm has a satisfactory performance
to recover
sparse real vectors  from their (un)corrupted affine quadratic measurements.

\subsection{Contributions}\label{s1.2}
The inertial proximal ADMM  for solving a  {\bf two-block} convex optimization has been proposed and well studied
 \cite{CCMY2015}.
The first contribution of this paper is to extend the inertial proximal ADMM nontrivially
for solving the   {\bf multi-block}
convex optimization problem  \eqref{Multiseparatedoperator}.
The proposed inertial proximal ADMM unifies and greatly extends the existing twisted version of the proximal ADMM \cite{WS2017}  and the proximal parallel splitting method %(PPSM)
\cite[Algorithm 3.1]{HHX2014}, with  additional simpler iteration scheme.
The second contribution  is the  global convergence of   the proposed inertial proximal ADMM
for a multi-block convex optimization with separable objective functions, see Theorem  \ref{IADMMConvergence3.mainthm}.
The third contribution is to apply the  inertial proximal ADMM
to recover  sparse
real vectors from their (un)corrupted affine quadratic measurements.
The numerical simulations show that  in most cases,
the proposed  ADMM-based algorithm has better performance
in retrieving sparse signals from their
(un)corrupted affine quadratic measurements  than conventional ADMM-based and  phase-retrieval-based approaches  do.

\subsection{Organization}\label{s1.3}
In Section \ref{s2.section},  we first introduce a proximal ADMM  and establish a mixed variational inequality.
Then based on the proposed proximal ADMM, we introduce
an inertial proximal ADMM to approximate KKT points of the multi-block convex optimization  \eqref{Multiseparatedoperator}.
 In Section \ref{s3.section},
 we establish  the convergence of  the  proposed inertial proximal ADMM
for the multi-block convex optimization  \eqref{Multiseparatedoperator}, which extends the corresponding conclusion in  \cite{CCMY2015,WS2017}
where two-block convex optimizations are considered.
In Section \ref{s4.section}, based on the proposed inertial proximal ADMM,  we introduce
 a compressive affine phase retrieval via lifting (CAPReaL)  algorithm
 to recover a sparse real vector from its noiseless  affine quadratic measurements,
 and  also a compressive affine phase retrieval via lifting with $\ell^p$-constraints ($p$-CAPReaL)
to  reconstruct a real signal approximately  from its  affine quadratic measurements corrupted by Gaussian/Cauchy/bounded noises.
 The performance of the CAPReaL and $p$-CAPReaL algorithms and the comparison with some conventional affine phase retrieval algorithms
 are presented in  Section \ref{Numerical.Section}.

\subsection{Notation}\label{s1.4}
In this paper, we use boldfaced  capital and small letters to denote a matrix and a  vector, and
denote the zero matrix and the zero vector by
${\bf O}$ and  ${\bf 0}$ respectively.
For a real number $t$, we denote its sign and positive part by ${\rm sgn} (t)$ and $t_+$ respectively.
For a matrix ${\bf X}$ (resp. a vector ${\bf x}$), we use  ${\bf X}^T$ (resp. ${\bf x}^T$) to denote its transpose, and
   $\|{\bf X}\|_p, 0<p\le \infty$ (resp.  $\|{\bf x}\|_p$) to denote its standard $\ell_p$ (quasi-)norm.
   The matrix norm $\|{\bf X}\|_p$ with $p=2$ is the same as the Frobenius norm, denoted  by $\|{\bf X}\|_F$, of the matrix ${\bf X}$.
 We use the notion ${\bf A}\succeq {\bf O}$ (resp., ${\bf A}\succ {\bf O}$) to represent that
the  matrix ${\bf A}$ is  positive semidefinite (resp. positive definite) and denote the set of all positive semidefinite (resp. positive definite) matrices of size $n$ by ${\bf S}^{n}_+$ (resp. ${\bf S}_{++}^n$). Given ${\bf A}\succeq {\bf O}$ of size $n$, we
define  $\langle {\bf u}, {\bf v}\rangle_{\bf A}:={\bf u}^{T}{\bf A}{\bf v}$ and $\|{\bf u}\|_{{\bf A}}:=\sqrt{\langle {\bf u}, {\bf u}\rangle_{\bf A}}$ for vectors ${\bf u},{\bf v}\in\mathbb{R}^n$. For a positive definite matrix ${\bf A}$,
$\langle \cdot, \cdot\rangle_{\bf A}$ and $\|\cdot\|_{\bf A}$ define an inner product and norm on $\mathbb {R}^n$ respectively, which
become the standard inner product $\langle \cdot,  \cdot\rangle$ and Euclidean norm $\|\cdot\|_2$  respectively
when ${\bf A}$ is the identity matrix ${\bf I}$.
A matrix ${\bf A}$ of size $m\times n$ is also considered as a linear map from
$\mathbb {R}^n$ to $\mathbb {R}^m$, and its operator norm is denoted by
$\|{\bf A}\|_{p\rightarrow p}=\sup_{{\bf x}\neq {\bf 0}}\|{\bf A}{\bf x}\|_p/{\|{\bf x}\|_p}, 0<p\le \infty$.
Similarly for a linear map $\mathcal{B}:\mathbb{R}^{m_1\times n_1}\rightarrow \mathbb{R}^{m_2\times n_2}$, we denote $\|\mathcal{B}\|_{F\rightarrow F}=\sup_{{\bf X}\neq {\bf O}}\|\mathcal{B}({\bf X})\|_F/{\|{\bf X}\|_F}$ as the induced norm of $\mathcal{B}$.

\section{Inertial Proximal ADMM}\label{s2.section}

Let $n=\sum_{i=1}^l n_i$.
We define an  affine function $F$ on ${\mathcal W}\subset  {\mathbb R}^{m+n}$
 by
  \begin{align}\label{Variate}
F({\bf w})  :=
\begin{bmatrix}
{\bf O}&                  {\bf O}&        \cdots&      {\bf O}&          -{\bf A}_1^{T}\\
{\bf O}&                   {\bf O}&       \cdots&     {\bf O}&          -{\bf A}_2^{T}\\
\vdots&                         \vdots&                \ddots&            \vdots&             \vdots\\
{\bf O}&                   {\bf O}&        \cdots&     {\bf O}&          -{\bf A}_l^{T}\\
{\bf A}_1&                  {\bf A}_2&   \cdots&      {\bf A}_l&    {\bf O}\\
\end{bmatrix}
\begin{pmatrix} {\bf x}_1\\  {\bf x}_2\\  \vdots\\  {\bf x}_l\\  {\bf z}\end{pmatrix}
-\begin{pmatrix} {\bf 0}\\    {\bf 0}\\   \vdots\\   {\bf 0}\\  {\bf c}\end{pmatrix},
% \in {\mathbb R}^{m+n},
\ ~{\bf w}   :=
\begin{pmatrix}
{\bf x}_1\\  {\bf x}_2\\  \vdots\\  {\bf x}_l\\  {\bf z}
\end{pmatrix}
\in {\mathcal W},
\end{align}
 and we say that  ${\bf w}\in {\mathcal W}$ is  a  Karush-Kuhn-Tucker (KKT) point
  %\cite[Section 5.5]{BV2004}
  of the convex optimization problem \eqref{Multiseparatedoperator} if
\begin{equation}
{\bf A}_j^{T}{\bf z}\in\partial f_j({\bf x}_j), \ 1\le j\le l,\ \ {\rm and} \ \
\sum_{j=1}^{l}{\bf A}_j{\bf x}_j={\bf c}
\end{equation}
\cite{BV2004}.
Then the convex  optimization problem \eqref{Multiseparatedoperator} reduces to finding  KKT  points ${\bf w}^*$
with %to satisfy
the mixed variational property,
\begin{align}\label{MixedVI1}
\theta({\bf w})-\theta({\bf w}^*)+\langle {\bf w}-{\bf w}^*,F({\bf w}^*)\rangle\geq 0 \  \ {\rm for\ all}\  \ {\bf w}\in\mathcal{W},
\end{align}
%where
%\begin{equation}
%\label{theta.def} \theta({\bf w}):=\sum_{j=1}^lf_j({\bf x}_j),  \ {\bf w}\in {\mathcal W}.
%\end{equation}
where
$\label{theta.def} \theta({\bf w}):=\sum_{j=1}^lf_j({\bf x}_j),  \ {\bf w}\in {\mathcal W}.$
So in this paper we always assume the existence of  KKT points for the convex optimization problem
\eqref{Multiseparatedoperator}.

\begin{assum}\label{wstar.assumption}
 The set of KKT points of the convex optimization problem
(\ref{Multiseparatedoperator}), denoted by $\mathcal{W}^*$, is nonempty.
\end{assum}

In this section, we introduce an inertial proximal ADMM  to approximate KKT points of the convex optimization problem \eqref{Multiseparatedoperator}.

\subsection{Proximal ADMM and mixed variational inequality}\label{s2.1}

For the proximal ADMM  \eqref{MultiproximalADMM}, we observe that in each iteration after updating the  first variable ${\bf x}_1$ and  the multiplier ${\bf z}$,
 variables ${\bf x}_2,\ldots,{\bf x}_l$  can be updated separately, and hence
subproblems for  ${\bf x}_2,\ldots,{\bf x}_l$  can  be implemented in a  parallel manner. In fact,
 we can minimize local versions of the augmented Lagrange function ${\mathcal L}_\beta$ to update   ${\bf x}_i^{k+1}, 1\le i\le l$, and ${\bf z}^{k+1}$ in each iteration:
\begin{subequations}\label{MultiproximalADMM-Equivalent}
\begin{equation}\label{MultiproximalADMM-Equivalent-1}
{\bf x}_1^{k+1}\in
\arg\min_{{\bf x}_1\in \mathcal{X}^1}
f_1({\bf x}_1)-\langle {\bf z}^k,{\bf A}_1{\bf x}_1 \rangle
+\frac{\beta}{2}\Big\| {\bf A}_1 {\bf x}_1+\sum_{j=2}^l{\bf A}_j{\bf x}_j^k-{\bf c}\Big\|_2^2
+\frac{1}{2}\|{\bf x}_1-{\bf x}_1^k\|_{{\bf H}_1}^2,
\end{equation}
\vskip-0.18in
\begin{equation}\label{MultiproximalADMM-Equivalent-2}
{\bf z}^{k+1}={\bf z}^k
-\beta\big({\bf A}_1{\bf x}_1^{k+1}+\sum_{j=2}^l{\bf A}_j{\bf x}_j^{k}-{\bf c}\big),  % \ {\rm and} \
\end{equation}
\vskip-0.18in
\begin{eqnarray}
\label{MultiproximalADMM-Equivalent-4}
\qquad   &  & {\bf x}_i^{k+1}   \in \arg\min_{{\bf x}_i\in\mathcal{X}^i}\bigg\{f_i({\bf x}_i)-\langle {\bf z}^{k+1},{\bf A}_i{\bf x}_i \rangle
+\frac{1}{2}\|{\bf x}_i-{\bf x}_i^k\|_{{\bf H}_i}^2\\
 \qquad &    & \quad +\frac{\beta}{2}\Big\|{\bf A}_1{\bf x}_1^{k+1}+{\bf A}_i{\bf x}_i
 +\Big(\sum_{j=2}^{i-1}+\sum_{j=i+1}^l\Big){\bf A}_j{\bf x}_j^k-{\bf c}\Big\|_2^2\bigg\}, \ \ 2\le i\le l.  \nonumber
\end{eqnarray}
\end{subequations}

Define a %customized
 proximal regularization matrix   ${\bf G}$ by
\begin{gather}\label{Proximalregularizationmatrix}
{\bf G}=
\begin{bmatrix}
{\bf H}_1&                  {\bf O}&    {\bf O}&           \cdots&       {\bf O}&       {\bf O}\\
{\bf O}&   \beta{\bf A}_2^{T}{\bf A}_2+{\bf H}_2&  {\bf O}&   \cdots&    {\bf O}&     -{\bf A}_2^{T}\\
{\bf O}&   {\bf O}&  \beta{\bf A}_3^{T}{\bf A}_3+{\bf H}_3&   \cdots&  {\bf O}&   -{\bf A}_3^{T}\\
\vdots&           \vdots&               \vdots&                    \ddots&    \vdots&      \vdots\\
{\bf O}&  {\bf O}& {\bf O}& \cdots&     \beta{\bf A}_l^{T}{\bf A}_l+{\bf H}_l&     -{\bf A}_l^{T}\\
{\bf O}&   -{\bf A}_2&  -{\bf A}_3&   \cdots&      -{\bf A}_l&    \frac{1}{\beta}{\bf I}_m\\
\end{bmatrix},
\end{gather}
which is introduced in \cite{CCMY2015} for $l=2$.
Following the argument used in \cite{CGHY2013, CCMY2015}, %\cIn the following theorem,
  we can show
  that the proximal ADMM algorithm  is a proximal-like
scheme  satisfying a mixed variational
inequality, which is similar to the mixed variational
inequality \eqref{MixedVI1} to be satisfied for a  KKT point ${\bf x}^*\in {\mathcal W}^*$.

\begin{theorem}\label{MixedVITheorem}  Let  $F$,  ${\bf G}$
and  ${\bf w}^{k}=\big(({\bf x}_1^{k})^T, \ldots,({\bf x}_l^{k})^T, ({\bf z}^{k})^T\big)^T, \ k\ge 0$,
 be  as in \eqref{Variate}, \eqref{Proximalregularizationmatrix} %the proximal ADMM
and \eqref{MultiproximalADMM-Equivalent} respectively.
Then  %${\bf w}^{k+1}\in\mathcal{W}$ and
\begin{equation}\label{MixedVI2}
 \theta({\bf w})-\theta({\bf w}^{k+1})+\langle {\bf w}-{\bf w}^{k+1},F({\bf w}^{k+1})+{\bf G}({\bf w}^{k+1}-{\bf w}^k)\rangle\geq 0, \ \ {\bf w}\in {\mathcal W},
\end{equation}
hold for all   $k\ge 0$.
\end{theorem}

\subsection{Inertial proximal ADMM}\label{s2.2}

To solve
 the separable multi-block convex  optimization problem \eqref{Multiseparatedoperator}, we introduce
an inertial proximal ADMM, Prox-IADMM for abbreviation, whose convergence analysis will be discussed in the next section.

%\begin{algorithm}[t]
\noindent\rule[0.25\baselineskip]{\textwidth}{1pt}
 \label{PIADMMa.def}\centerline {\bf   Prox-IADMM Algorithm}\\
 {\bf Input}: \ Given ${\bf H}_1,\ldots,{\bf H}_l\succeq {\bf O}$, penalty
parameter $\beta>0$ and  step sizes $\alpha_k, k\ge 0$.
\\
{\bf Initials}:\  Initial step $k=0$, and initial vectors  $({\bf x}_1^0,\ldots,{\bf x}_l^0;{\bf z}^0)\in\mathcal{W}$ with\\ $({\bf x}_1^{-1},\ldots,{\bf x}_l^{-1};{\bf z}^{-1})=({\bf x}_1^0,\ldots,{\bf x}_l^0;{\bf z}^0)$. \\
{\bf Circulate}~Step 1--Step 3 until ``a stopping criterion is satisfied": %\\

{\bf Step 1 (Inertial Step)}
\begin{equation}\label{InertialProximalADMM-Multi-1}
(\bar{{\bf x}}_1^k,\ldots,\bar{{\bf x}}_l^k;\bar{{\bf z}}^k)
=({\bf x}_1^k,\ldots,{\bf x}_l^k;{\bf z}^k)
+\alpha_k({\bf x}_1^k-{\bf x}_1^{k-1},\ldots,{\bf x}_l^k-{\bf x}_l^{k-1};{\bf z}^k-{\bf z}^{k-1}).
\end{equation}

{\bf Step 2 (Prox-ADMM)}
\begin{subequations}\label{InertialProximalADMM-Multi}
\begin{equation}\label{InertialProximalADMM-Multi-2}
{\bf x}_1^{k+1}\in \arg\min_{{\bf x}_1\in\mathcal{X}^1}\mathcal{L}_{\beta}({\bf x}_1,\bar{{\bf x}}_2^{k},\ldots,\bar{{\bf x}}_l^{k};\bar{{\bf z}}^k)
+\frac{1}{2}\|{\bf x}_1-\bar{{\bf x}}_1^k\|_{{\bf H}_1}^2,
\end{equation}
\vskip-0.28in
\begin{equation}\label{InertialProximalADMM-Multi-3}
{\bf z}^{k+1}=\bar{{\bf z}}^k
-\beta\Big({\bf A}_1{\bf x}_1^{k+1}+\sum_{j=2}^l{\bf A}_j\bar{{\bf x}}_j^{k}-{\bf c}\Big),
\end{equation}
\vskip-0.28in
\begin{equation}\label{InertialProximalADMM-Multi-5}
{\bf x}_i^{k+1}\in
\arg\min_{{\bf x}_i\in\mathcal{X}^i}\mathcal{L}_{\beta}({\bf x}_1^{k+1},\bar{{\bf x}}_2^{k},\ldots, \bar{{\bf x}}_{i-1}^{k}, {\bf x}_i,  \bar{{\bf x}}_{i+1}^{k}\ldots,\bar{{\bf x}}_l^k;
{\bf z}^{k+1})
+\frac{1}{2}\|{\bf x}_i-\bar{{\bf x}}_i^k\|_{{\bf H}_i},
\end{equation}
\end{subequations}
where $i=2,\ldots, l$.
%for  $2\le i\le l$.
%\\

{\bf Step 3}\  Update $k$ to $k+1$.\\
{\bf Output}: $(\hat{{\bf x}}_1,\ldots,\hat{{\bf x}}_l;\hat{{\bf z}})$.\\
\noindent\rule[0.25\baselineskip]{\textwidth}{1pt}
%\end{algorithm}

The above inertial proximal ADMM is introduced in  \cite{CCMY2015} for $l=2$. It  extrapolates at the current point in the direction
of last movement, and then applies the proximal ADMM to the extrapolated point at each iteration.
 Taking ${\bf H}_j={\bf O}, 1\le j\le l$, in \eqref{InertialProximalADMM-Multi},     we obtain
 % an inertial variant of  ADMM, % IADMM for abbreviation,
  a twisted version of the proximal ADMM in  \cite{WS2017}
where step sizes $\alpha_k, k\ge 0$, are also selected to be the same in each iteration.

Let
${\bf w}^{k}=\big(({\bf x}_1^{k})^T, \ldots,({\bf x}_l^{k})^T, ({\bf z}^{k})^T\big)^T\in {\mathcal W}, \ k\ge 0,$
be as in the above Prox-IADMM, and
set
\begin{equation}\label{barwk.def}
\bar{{\bf w}}^k:={\bf w}^k+\alpha_k({\bf w}^k-{\bf w}^{k-1}), \ k\ge 0.
\end{equation}
Similar to  the conclusion in %Following the argument used in the proof of
Theorem \ref{MixedVITheorem}, we have the following
mixed variational inequality  for ${\bf w}^k$ in the Prox-IADMM algorithm  \eqref{InertialProximalADMM-Multi-1}  and \eqref{InertialProximalADMM-Multi}.

\begin{theorem}\label{MixedVITheorem-Inertial}  Let  $F$,  ${\bf G}$
and  ${\bf w}^{k}=\big(({\bf x}_1^{k})^T, \ldots,({\bf x}_l^{k})^T, ({\bf z}^{k})^T\big)^T, \ k\ge 0$,
 be  as in \eqref{Variate}, \eqref{Proximalregularizationmatrix} %the proximal ADMM
and \eqref{MultiproximalADMM} respectively.
Then  %${\bf w}^{k+1}\in\mathcal{W}$ and
\begin{equation}\label{MixedVI3}
 \theta({\bf w})-\theta({\bf w}^{k+1})+\langle {\bf w}-{\bf w}^{k+1},F({\bf w}^{k+1})+{\bf G}({\bf w}^{k+1}-\bar{{\bf w}}^k)\rangle\geq 0,  \   \ {\bf w}\in\mathcal{W}
\end{equation}
hold for all   $k\ge 0$.
\end{theorem}

\begin{rem}\label{stop.remark} {\rm
The  stopping criterion in the   Prox-IADMM algorithm  \eqref{InertialProximalADMM-Multi-1}  and \eqref{InertialProximalADMM-Multi} should be appropriately chosen.
In this paper, we will use the following stopping criterion for any given  accuracy $\epsilon$,
\begin{equation}\label{StoppingCreterion.e2}
\|{\bf x}_1^{k+1}-\bar{{\bf x}}_1^{k}\|_{{\bf H}_1}^2+
2\sum_{j=2}^{l}\|{\bf x}_j^{k+1}-\bar{{\bf x}}_j^{k}\|_{\beta{\bf A}_j^{T}{\bf A}_j+{\bf H}_j}^2
+\frac{l}{\beta}\|{\bf z}^{k+1}-\bar{{\bf z}}^{k}\|_{2}^2\le \epsilon,
\end{equation}
see Subsection \ref{capreal.simulation0} for numerical demonstrations.
Under the above stopping criterion,  one may verify that  %we have
%$\begin{eqnarray*}\label{StoppingCreterion.e1}
$\|{\bf w}^{k+1}-\bar{{\bf w}}^k\|_{\bm{G}}^2
%&  = &  \|{\bf x}_1^{k+1}-\bar{{\bf x}}_1^{k}\|_{{\bf H}_1}^2
%+\sum_{j=2}^{l}\|{\bf x}_j^{k+1}-\bar{{\bf x}}_j^{k}\|_{\beta{\bf A}_j^{T}{\bf A}_j+{\bf H}_j}^2\nonumber\\
%& & +\frac{1}{\beta}\|{\bf z}^{k+1}-\bar{{\bf z}}^{k}\|_{2}^2
%-2\sum_{j=2}^l\langle {\bf A}_j({\bf x}_j^{k+1}-\bar{{\bf x}}_j^{k}), {\bf z}^{k+1}-\bar{{\bf z}}^{k}\rangle %\\
%&  \leq  &  \|{\bf x}_1^{k+1}-\bar{{\bf x}}_1^{k}\|_{{\bf H}_1}^2
%+\frac{l}{\beta}\|{\bf z}^{k+1}-\bar{{\bf z}}^{k}\|_2^2 %\nonumber\\
%%& &
%+ 2\sum_{j=2}^{l}\|{\bf x}_j^{k+1}-\bar{{\bf x}}_j^{k}\|_{\beta{\bf A}_j^{T}{\bf A}_j+{\bf H}_j}^2
\le \epsilon$.
%\end{eqnarray*}
}\end{rem}

\section{Convergence Analysis of the Inertial Proximal ADMM}\label{s3.section}

In this section, we establish convergence of  the Prox-IADMM \eqref{InertialProximalADMM-Multi-1}
and \eqref{InertialProximalADMM-Multi}
for multi-block convex optimizations, and we
 discuss (non)asymptotic rates of convergence for the best
primal function value and feasibility residues. This
extends the corresponding conclusions in  \cite{CCMY2015,WS2017}
where two-block convex optimizations are considered.

 The convergence of the Prox-IADMM \eqref{InertialProximalADMM-Multi-1}  and \eqref{InertialProximalADMM-Multi} depends on  adaptive  selection of step sizes $\alpha_k, k\ge 0$, in \eqref{InertialProximalADMM-Multi-1},  see  \cite[Proposition 2.1]{AA2001},
\cite[Proposition 2.5]{A2004} and \cite[Proposition 4.5]{CCMY2015}.
In this paper, we always assume the following:
\begin{assum}\label{assumption1}
Step sizes  $\alpha_k, k\ge 0$, in \eqref{InertialProximalADMM-Multi-1}
are nonnegative and bounded by some $\alpha\in (0, 1)$,
\begin{equation}\label{assumption1.eq1}
0\leq\alpha_k\leq\alpha<1, \ \  k\ge 0,
\end{equation}
and satisfy
\begin{equation}
\label{assumption1.eq2}
\sum_{k=1}^{\infty}\alpha_k\|{\bf w}^k-{\bf w}^{k-1}\|_{\bf G}^2<\infty.
\end{equation}
%where $\alpha\in[0, 1)$ and  ${\bf w}^k, k\ge 0$, are given in \eqref{wk+1.def}.
%{\color{blue}where $\alpha\in[0, 1)$.}
\end{assum}

Assumption \ref{assumption1}  % on   step sizes $\alpha_k, k\ge 0$,
 has been used in \cite[Assumption 1]{CCMY2015} for the
convergence of   the Prox-IADMM  for the two-block convex optimization problem.
In practice, we may  select step sizes dynamically based on historical
iterative information, for instance,
$
\alpha_k=\min\left\{{1}/{3}, (k\|{\bf w}^k-{\bf w}^{k-1}\|_{{\bf G}})^{-2}\right\}, k\ge 0,
$ see Section \ref{caprealalgorithm.simulation} for  numerical demonstrations.
Inspired by  \cite[Theorem 2]{CMY2015}, we can show
 that
 the  monotonic  family of step sizes in the following proposition  satisfies Assumption \ref{assumption1}, see
 Section \ref{supplement} for the proof.
%{\color{blue}with its detailed proof given in the supplementary material \cite{LCS2019}.}

\begin{prop}\label{Prop-Guarantee-Assumption1} Let ${\bf w}^*\in\mathcal{W}^*$.
If step sizes  $\alpha_k, k\ge 0$, in \eqref{InertialProximalADMM-Multi-1}
satisfy
\begin{equation}\label{Prop-Guarantee-Assumption1.eq1}
0\leq \alpha_{k}\leq \alpha_{k+1}\leq \alpha,\ \  k\ge 0,
\end{equation}
for some $\alpha<1/3$, then
\begin{equation}\label{Prop-Guarantee-Assumption1.eq2}
\sum_{k=1}^\infty \alpha_k \|{\bf w}^k-{\bf w}^{k-1}\|_{{\bf G}}^2\le
\alpha \sum_{k=1}^{\infty}\|{\bf w}^k-{\bf w}^{k-1}\|_{{\bf G}}^2
\leq \frac{\alpha}{(1-3\alpha)(1-\alpha)}\|{\bf w}^0-{\bf w}^{*}\|_{{\bf G}}^2.
\end{equation}
\end{prop}

The main theoretical conclusion of this paper is the following theorem about feasibility and  convergence of the Prox-IADMM scheme \eqref{InertialProximalADMM-Multi-1}
and \eqref{InertialProximalADMM-Multi}.

\begin{theorem}\label{IADMMConvergence3.mainthm}
Let ${\bf x}_1^{k}, \cdots, {\bf x}_l^{k},
 {\bf w}^k, k\ge 0$,  be as in   the Prox-IADMM  \eqref{InertialProximalADMM-Multi-1}
and \eqref{InertialProximalADMM-Multi},
  the family $\alpha_k, k\ge 0$, of step sizes satisfy Assumption \ref{assumption1},
  and the regularization matrices ${\bf H}_j, 1\le j\le l$, satisfy
  \begin{equation}\label{IADMMConvergence3.cor.eq0}
{\bf H}_1\succ {\bf 0} \ \ {\rm and}\  \ {\bf H}_j \succ \beta (l-2) {\bf A}_j^T{\bf A}_j, \ 2\le j\le l.
\end{equation}
 Then the following statements hold.

 \begin{itemize}

\item[{(i)}]   The
Prox-IADMM algorithm \eqref{InertialProximalADMM-Multi-1}  and \eqref{InertialProximalADMM-Multi} is feasible,
\begin{equation}\label{ConvergenceTheorem1.eq2}
\lim_{k\to\infty}\sum_{j=1}^l{\bf A}_j {\bf x}_j^k={\bf c}.
\end{equation}

\item[{(ii)}] The objective function in the
Prox-IADMM algorithm \eqref{InertialProximalADMM-Multi-1}  and \eqref{InertialProximalADMM-Multi}
converges to the optimal value,
 \begin{equation}\label{objectivefunctionconvergence}
 \lim_{k\to \infty}
  \sum_{j=1}^lf_j({\bf x}_j^k)=\min_{ {\bf x}_j\in \mathcal{X}^j, 1\le j\le l\ {\rm  and} \  \sum_{j=1}^{l}{\bf A}_j{\bf x}_j={\bf c}}  \sum_{j=1}^lf_j({\bf x}_j).
 \end{equation}

 \item[{(iii)}] The sequence ${\bf w}^k, k\ge 1$, in the
Prox-IADMM algorithm \eqref{InertialProximalADMM-Multi-1}  and \eqref{InertialProximalADMM-Multi}
converges to a KKT point  ${\bf w}^*\in \mathcal{W}^*$ of the convex optimization
problem \eqref{Multiseparatedoperator},
\begin{equation}\label{mostimportant} \lim_{k\to \infty} {\bf w}^k= {\bf w}^*.
\end{equation}

\end{itemize}

\end{theorem}

We remark that  the  requirement \eqref{IADMMConvergence3.cor.eq0} on  regularization matrices  ${\bf H}_j, 1\le j\le l$,
are met for prox-linear matrices
 in \eqref{prox-linear.def}
 when
%\begin{equation}
$$\eta_j< (l-1)^{-1} \|{\bf A}_j\|_{2\to 2}^{-2}, \ 1\le j\le l,$$
%\end{equation}
and similarly for standard proximal matrices
in
 \eqref{standardproximal.def}
 when
%\begin{equation}
$$\eta_j<(l-2)^{-1} \|{\bf A}_j\|_{2\to 2}^{-2}, 1\le j\le l.$$
%\end{equation}

 The  proximal regularization matrix   ${\bf G}$ in \eqref{Proximalregularizationmatrix}
 plays an important role in our study of the convergence of the Prox-IADMM \eqref{InertialProximalADMM-Multi-1}
and \eqref{InertialProximalADMM-Multi}.  In Section \ref{assumption.subsection},  we discuss its positive semi-definite assumption
and the convergence of $\|{\bf w}^{k}-{\bf w}^*\|_{\bf G}^2, k\ge 1$, for all ${\bf w}^*\in {\mathcal W}^*$.
We divide the proof of Theorem \ref{IADMMConvergence3.mainthm} into several steps. In Sections \ref{feasibility.subsection},
 we consider feasibility of   the Prox-IADMM  \eqref{InertialProximalADMM-Multi-1}
and \eqref{InertialProximalADMM-Multi}, and we prove the first conclusion of Theorem \ref{IADMMConvergence3.mainthm}. In Section
 \ref{s3.1b},  we discuss the convergence of objective functions  and provide the proof of
 the second conclusion of Theorem \ref{IADMMConvergence3.mainthm}.
 To prove the third conclusion  of Theorem \ref{IADMMConvergence3.mainthm}, we establish
 the
 boundedness of
${\bf w}^k, k\ge 0$, in  Section \ref{boundedness.subsection} first,
and then in Sections \ref{convergence.subsection},  we give the proof of the third conclusion of
Theorem \ref{IADMMConvergence3.mainthm}.

\subsection{Weak convergence}\label{assumption.subsection}
Let ${\bf G}$  be the
 proximal regularization matrix  in \eqref{Proximalregularizationmatrix}.
 In this paper, we always assume that  proximal regularization matrix   ${\bf G}$ in \eqref{Proximalregularizationmatrix} is positive semi-definite.

\begin{assum}\label{Gge0.assumption} The   matrix   ${\bf G}$ in \eqref{Proximalregularizationmatrix}
is  positive semi-definite,
%\begin{equation}\label{Gge0}
i.e., ${\bf G}\succeq{\bf O}.$
\end{assum}

\begin{rem}\label{Gpositive.rem} {\rm
The  Assumption \ref{Gge0.assumption}
is satisfied if
${\bf H}_j, 1\le j\le l$, are  positive semi-definite matrices chosen appropriately.
Assume that
\begin{equation} {\bf H}_j+\beta{\bf A}_j^{T}{\bf A}_j\succ{\bf O},\ 2\le j\le l.\end{equation}
By \eqref{Proximalregularizationmatrix} and the Schur complement \cite[Section A.5.5]{BV2004},
% the   positive semi-definite property \eqref{Gge0} % for the matrix ${\bf G}$
 the   positive semi-definite property ${\bf G}\succeq {\bf O}$
reduces to
\begin{align}\label{e2.1}
{\bf H}_1\succeq {\bf O}
\end{align}
and
\begin{equation}\label {e2.1+}
\beta^{-1} {\bf I}- \sum_{j=2}^l{\bf A}_j({\bf H}_j+\beta{\bf A}_j^{T}{\bf A}_j)^{-1}{\bf A}_j^{T}\succeq {\bf O}.
\end{equation}
Clearly, the requirement \eqref{e2.1+}  % for ${\bf H}_j, 2\le j\le l$,
 is met if
\begin{equation}\label {e2.1++}
\sum_{j=2}^l \|{\bf A}_j\|_{2\to 2}^2 \|({\bf H}_j+\beta{\bf A}_j^{T}{\bf A}_j)^{-1}\|_{2\to 2}\le \beta^{-1}.
\end{equation}
From the above argument, we conclude that  Assumption \ref{Gge0.assumption}
is satisfied  %  the  proximal regularization matrix   ${\bf G}$  is positive semi-definite
if  \eqref{IADMMConvergence3.cor.eq0}  holds for  regularization matrices  ${\bf H}_j, 1\le j\le l$. %is satisfied.
}\end{rem}

   For the case that ${\bf H}_j, 1\le j\le l$, are prox-linear matrices in \eqref{prox-linear.def},
we obtain from \eqref{e2.1} and \eqref{e2.1++} that
Assumption \ref{Gge0.assumption} is satisfied when
$$0<\eta_1 \|{\bf A}_1\|_{2\to 2}^2\le 1 \ \ {\rm and }\ \  \sum_{j=2}^l \eta_j \|{\bf A}_j\|_{2\to 2}^2< 1.$$
%\end{equation}
We remark that the strictly positive definite property for the matrix
${\mathbf G}$ is established in \cite[Theorem 2.1]{DLPY2017} under a stronger assumption that
$$0<\eta_j \|{\bf A}_j\|_{2\to 2}^2\le  l^{-1}, \  1\le j\le l.$$
%$0<\eta_j \|{\bf A}_j\|_{2\to 2}^2\le  (l-1)^{-1},   1\le j\le l.$

 For standard proximal matrices ${\bf H}_j, 1\le j\le l$,   in
 \eqref{standardproximal.def},
 we obtain from \eqref{e2.1} and \eqref{e2.1+} that
 Assumption \ref{Gge0.assumption} is satisfied
 when
$$\sum_{j=2}^l {\eta_j \|{\bf A}_j\|_{2\to 2}^2}({1+ \eta_j \|{\bf A}_j\|_{2\to 2}^2})^{-1}\le 1.$$
We remark that the strictly positive definite property for the matrix
${\mathbf G}$ is established in \cite[Theorem 2.1]{DLPY2017} under a stronger assumption that
$$0<\eta_j \|{\bf A}_j\|_{2\to 2}^2< 1/(l-1), \ 1\le j\le l.$$
%$0<\eta_j \|{\bf A}_j\|_{2\to 2}^2< 1/(l-2), 1\le j\le l$.

 Motivated by  \cite[Theorem 1]{CMY2015},
 %Following the argument used in \cite[Theorem 1]{CMY2015},
 we obtain that $\|{\bf w}^{k}-{\bf w}^*\|_{\bf G}, k\ge 0$, converges for all ${\bf w}^*\in {\mathcal W}^*$, in the following theorem.

\begin{theorem}\label{WeakConvergenceThm}
%Let ${\bf w}^*\in {\mathcal W}^*$ and $ {\bf w}_k$ and $\bar{\bf w}_k, k\ge 0$,
% be as in \eqref{wk+1.def} and \eqref{barwk.def}.
Let ${\bf w}^*\in {\mathcal W}^*$ and $\bar{\bf w}_k, k\ge 0$ be as in  \eqref{barwk.def}.
If  Assumption \ref{Gge0.assumption} is satisfied for the proximal regularization
matrix ${\bf G}$ in \eqref{Proximalregularizationmatrix},  then
\begin{equation}\label{ConvergenceLemma1.eq1}
\sum_{k=0}^{\infty}\|{\bf w}^{k+1}-\bar{\bf w}^k\|_{\bf G}^2\le  \|{\bf w}^0-{\bf w}^*\|_{\bf G}^2+ \frac{1+\alpha}{1-\alpha}\sum_{j=0}^{\infty} \alpha_j
\|{\bf w}^j-{\bf w}^{j-1}\|_{\bf G}^2 <\infty,
\end{equation}
\begin{equation}\label{ConvergenceLemma1.eq2-}
\lim_{k\to \infty} \|{\bf w}^{k}-{\bf w}^*\|_{\bf G}^2 \ \ {\rm exists},
\end{equation}
and
\begin{equation}\label{ConvergenceLemma1.eq2}
\sup_{k\ge 0}\|{\bf w}^{k}-{\bf w}^*\|_{\bf G}^2 \le \|{\bf w}^0-{\bf w}^*\|_{\bf G}^2+ \frac{1+\alpha}{1-\alpha}\sum_{j=0}^{\infty} \alpha_j
\|{\bf w}^j-{\bf w}^{j-1}\|_{\bf G}^2.
\end{equation}
\end{theorem}

\begin{proof}
By \eqref{Variate}, the function $F$ on ${\mathcal W}$  satisfies
 %\begin{equation}\label{F.property}
$$\langle {\bf w}_1-{\bf w}_2, F({\bf w}_1)-F({\bf w}_2)\rangle= 0$$
for all ${\bf w}_1, {\bf w}_2\in {\mathcal W}$.
%\end{equation}
This together with  the
mixed variational inequalities \eqref{MixedVI1} and \eqref{MixedVI3} implies that
\begin{eqnarray}\label{e3.7}
 \langle {\bf w}^{k+1}-{\bf w}^*, {\bf w}^{k+1}-\bar{{\bf w}}^k \rangle_{\bf G}
  & \hskip-0.08in  \leq & \hskip-0.08in   \theta({\bf w}^*)-\theta({\bf w}^{k+1})-\langle {\bf w}^{k+1}-{\bf w}^*, F({\bf w}^{k+1}) \rangle\nonumber\\
 & \hskip-0.08in   \leq  &  \hskip-0.08in  \theta({\bf w}^*)-\theta({\bf w}^{k+1})-\langle {\bf w}^{k+1}-{\bf w}^*, F({\bf w}^{*}) \rangle\leq 0.
\end{eqnarray}
By direct calculation, we have
\begin{equation} \label{e3.8}
2\langle {\bf w}^{k+1}-{\bf w}^*, {\bf w}^{k+1}-{\bf w}^k \rangle_{\bf G}=
\|{\bf w}^{k+1}-{\bf w}^k\|_{\bf G}^2+ \|{\bf w}^{k+1}-{\bf w}^*\|_{\bf G}^2- \|{\bf w}^{k}-{\bf w}^*\|_{\bf G}^2,
\end{equation}
\vskip-0.28in
\begin{eqnarray} \label{e3.9}
& &
2\langle {\bf w}^{k+1}-{\bf w}^*, {\bf w}^{k}-{\bf w}^{k-1} \rangle_{\bf G}  =    2\langle {\bf w}^{k+1}-{\bf w}^k, {\bf w}^{k}-{\bf w}^{k-1} \rangle_{\bf G}\\
& & \qquad \qquad\quad    +
\|{\bf w}^{k}-{\bf w}^{k-1}\|_{\bf G}^2+ \|{\bf w}^{k}-{\bf w}^*\|_{\bf G}^2- \|{\bf w}^{k-1}-{\bf w}^*\|_{\bf G}^2, \nonumber
\end{eqnarray}
and
\begin{eqnarray}\label{e3.11}
\|{\bf w}^{k+1}-\bar{{\bf w}}^k\|_{\bf G}^2 & \hskip-0.08in = & \hskip-0.08in \|{\bf w}^{k+1}-{\bf w}^k\|_{\bf G}^2+\alpha_k^2\|{\bf w}^{k}-{\bf w}^{k-1}\|_{\bf G}^2\\
& & \qquad -2\alpha_k\langle {\bf w}^{k+1}-{\bf w}^k, {\bf w}^{k}-{\bf w}^{k-1}\rangle_{\bf G}.  \nonumber
\end{eqnarray}
Set
$\nu_k=\|{\bf w}^{k}-{\bf w}^*\|_{\bf G}^2- \|{\bf w}^{k-1}-{\bf w}^*\|_{\bf G}^2, \  \ k\ge 0$.
 Then it follows from \eqref{e3.7}, \eqref{e3.8}, \eqref{e3.9} and \eqref{e3.11}
 that
\begin{equation}\label{e3.11+1}
\nu_{k+1}  \le  \alpha_k \nu_k +(\alpha_k+\alpha_k^2)
\|{\bf w}^{k}-{\bf w}^{k-1}\|_{\bf G}^2-\|{\bf w}^{k+1}-\bar{{\bf w}}^k\|_{\bf G}^2, \ \ k\ge 0.
\end{equation}
This together with Assumption \ref{assumption1} implies that
\begin{eqnarray}\label{e3.11+}
\max(\nu_{k+1}, 0)&   \le &  \alpha \max(\nu_k, 0)+ (1+\alpha) \alpha_k
\|{\bf w}^{k}-{\bf w}^{k-1}\|_{\bf G}^2\nonumber\\
&   \le &   \cdots \le (1+\alpha) \sum_{j=0}^k \alpha^{k-j}
\alpha_j
\|{\bf w}^{j}-{\bf w}^{j-1}\|_{\bf G}^2, \ \ k\ge 0.
\end{eqnarray}
Therefore
\begin{eqnarray}\label{e3.11+2}
\sum_{k=1}^\infty
\max(\nu_{k}, 0) &  \hskip-0.08in\le &  \hskip-0.08in \frac{1+\alpha}{1-\alpha}  \sum_{j=0}^\infty\alpha_j
\|{\bf w}^{j}-{\bf w}^{j-1}\|_{\bf G}^2\\
&  \hskip-0.08in= &  \hskip-0.08in \frac{1+\alpha}{1-\alpha}  \sum_{j=1}^\infty\alpha_j
\|{\bf w}^{j}-{\bf w}^{j-1}\|_{\bf G}^2<\infty,\nonumber
\end{eqnarray}
where the last inequality holds by Assumption \ref{assumption1}.

By \eqref{e3.11+1}, we obtain
\begin{eqnarray}
\|{\bf w}^{k+1}-\bar{{\bf w}}^k\|_{\bf G}^2 & \hskip-0.08in   \le &   \hskip-0.08in -\nu_{k+1} + \alpha_k \nu_k +(\alpha_k+\alpha_k^2)
\|{\bf w}^{k}-{\bf w}^{k-1}\|_{\bf G}^2\nonumber\\
&    \hskip-0.08in\le &   \hskip-0.08in  -\nu_{k+1} + \alpha \max(\nu_k, 0) + (1+\alpha) \alpha_k
\|{\bf w}^{k}-{\bf w}^{k-1}\|_{\bf G}^2
, \ \ k\ge 0.
\end{eqnarray}
Summing over all nonnegative $k\ge 0$ in the above inequality and applying \eqref{e3.11+2} proves \eqref{ConvergenceLemma1.eq1}.

Set
\begin{equation} \label{e3.12+0}
\gamma_k:=\|{\bf w}^{k}-{\bf w}^*\|_{\bf G}^2-\sum_{j=0}^k\max(\nu_j, 0), \ k\ge 0.
\end{equation}
Then  the sequence $\{\gamma_k\}_{k=0}^{\infty}$ is bounded below by \eqref{e3.11+2},
and  it is  nonincreasing
as
$\gamma_{k+1}-\gamma_k= \nu_{k+1}-\max(\nu_{k+1}, 0)\le 0, \ k\ge 0$.
 Therefore
the sequence $\{\gamma_k\}_{k=0}^{\infty}$ converges.
Hence  the convergence  in \eqref{ConvergenceLemma1.eq2-}
follows from \eqref{e3.11+2}
and
\eqref{e3.12+0}.

By \eqref{e3.11+2} and the monotonicity of $\gamma_k, k\ge 0$, we have
\begin{eqnarray*}
\|{\bf w}^{k}-{\bf w}^*\|_{\bf G}^2 & \hskip-0.08in = & \hskip-0.08in
\gamma_k+\sum_{j=0}^k\max(\nu_j, 0)\\
& \hskip-0.08in \le & \hskip-.08in \|{\bf w}^{0}-{\bf w}^*\|_{\bf G}^2+ \frac{1+\alpha}{1-\alpha}  \sum_{j=1}^\infty\alpha_j
\|{\bf w}^{j}-{\bf w}^{j-1}\|_{\bf G}^2<\infty.
\end{eqnarray*}
 This proves \eqref{ConvergenceLemma1.eq2}.
\end{proof}

\begin{rem} {\rm
By Theorem  \ref{WeakConvergenceThm},
%we  have
%\begin{equation}\label{ConvergenceLemma1.cor2.eq0}
% \min_{1\leq i\leq k}\|{\bf w}^{i+1}-\bar{\bf w}^i\|_{{\bf G}}^2=o(k^{-1/2}).
%\end{equation}
we  have that
$\min_{1\leq i\leq k}\|{\bf w}^{i+1}-\bar{\bf w}^i\|_{{\bf G}}^2=o(k^{-1/2}).$
For the case that the step sizes $\alpha_k,~k\geq 0$ are chosen in \eqref{Prop-Guarantee-Assumption1.eq1},
we can apply  the argument used in the proof of Theorem  \ref{WeakConvergenceThm} to show that
 \begin{equation}\label{ConvergenceLemma1.cor2.eq1}
\sum_{k=0}^{\infty}\|{\bf w}^{k+1}-\bar{\bf w}^k\|_{{\bf G}}^2
\le\left(1+ \frac{\alpha(1+\alpha)}{(1-\alpha)^2(1-3\alpha)}\right)\|{\bf w}^0-{\bf w}^{*}\|_{{\bf G}}^2,
\end{equation}
which implies that
\begin{equation}\label{ConvergenceLemma1.cor2.eq2}
 \min_{1\leq i\leq k}\|{\bf w}^{i+1}-\bar{\bf w}^i\|_{{\bf G}}
 \le \sqrt{\left(1+ \frac{\alpha(1+\alpha)}{(1-\alpha)^2(1-3\alpha)}\right)}\|{\bf w}^0-{\bf w}^{*}\|_{{\bf G}} k^{-1/2}.
\end{equation}
%We remark that
The  above asymptotic/nonsymptotic convergence rates
%in \eqref{ConvergenceLemma1.cor2.eq0}
%and
%\eqref{ConvergenceLemma1.cor2.eq2}
 for $\|{\bf w}^{k+1}-\bar{\bf w}^k\|_{{\bf G}}^2$, $k\ge 1$,
are also given in
 \cite[Theorems 4.4 and 4.6]{CCMY2015} and  \cite[Theorems 2, 4 and 7]{CMY2015}.
 }\end{rem}

\subsection{Feasibility of the Prox-IADMM}\label{feasibility.subsection}

Set
\begin{align}\label{G2.def}
{\bf G}_2=
\begin{bmatrix}
{\bf H}_1&                  {\bf O}&    {\bf O}&           \cdots&       {\bf O}&       {\bf O}\\
{\bf O}&   {\bf H}_2&  -\beta{\bf A}_2^{T}{\bf A}_3&   \cdots&    -\beta{\bf A}_2^{T}{\bf A}_l&     {\bf O}\\
\vdots&           \vdots&               \vdots&                    \ddots&    \vdots&      \vdots\\
{\bf O}&   -\beta({\bf A}_2^{T}{\bf A}_l)^{T}&  -\beta({\bf A}_3^{T}{\bf A}_l)^{T}&   \cdots&    {\bf H}_l&     {\bf O}\\
{\bf O}&   {\bf O}&  {\bf O}&   \cdots&      {\bf O}&    {\bf O}&\\
\end{bmatrix}.
\end{align}
 In this section,  we prove the first conclusion \eqref{ConvergenceTheorem1.eq2} of Theorem \ref{IADMMConvergence3.mainthm}
 under a weaker assumption that ${\bf G}$ and ${\bf G}_2$ are positive semi-definite.

\begin{theorem}\label{ConvergenceTheorem1}  Let   ${\bf G}$,
$ {\bf G}_2$ and $\{{\bf x}_j^{k}\}_{k=0}^{\infty}, 1\le j\le l$,
 be as in  \eqref{Proximalregularizationmatrix},  \eqref{G2.def}
and the Prox-IADMM  \eqref{InertialProximalADMM-Multi-1}
and \eqref{InertialProximalADMM-Multi} respectively, and ${\bf w}^*\in {\mathcal W}^*$.
If  Assumption \ref{Gge0.assumption} is satisfied, and  matrices $ {\bf G}_2$ and ${\bf H}_j, 1\le j\le l$, are positive semi-definite, then
%\begin{eqnarray}\label{ConvergenceTheorem1.pf.eq1}
% \Big\|\sum_{j=1}^l{\bf A}_j{\bf x}_j^{k+1}-{\bf c}\Big\|_2^2\leq \frac{1}{\beta}\|{\bf w}^{k+1}-\bar{{\bf w}}^k\|_{{\bf G}}^2,
%\end{eqnarray}
%and
\begin{equation}\label{ConvergenceTheorem1.eq1}
\sum_{k=1}^{\infty}\Big\|\sum_{j=1}^l{\bf A}_j {\bf x}_j^k-{\bf c}\Big\|_2^2
\le  \frac{1}{\beta} \|{\bf w}^0-{\bf w}^*\|_{\bf G}^2
+ \frac{1+\alpha}{\beta(1-\alpha)}\sum_{k=0}^{\infty} \alpha_k
\|{\bf w}^k-{\bf w}^{k-1}\|_{\bf G}^2<\infty.
\end{equation}
\end{theorem}

\begin{proof} %[Proof of Theorem \ref{ConvergenceTheorem1}]
By \eqref{InertialProximalADMM-Multi-3}, we have
$$
\sum_{j=1}^l{\bf A}_j{\bf x}_j^{k+1}-{\bf c}=
\sum_{j=2}^n{\bf A}_j({\bf x}_j^{k+1}-\bar{{\bf x}}_j^k)
-\frac{1}{\beta}({\bf z}^{k+1}-\bar{{\bf z}}^k).
$$
Therefore
\begin{eqnarray}\label{ConvergenceTheorem1.pf.eq1}
 \Big\|\sum_{j=1}^l{\bf A}_j{\bf x}_j^{k+1}-{\bf c}\Big\|_2^2
& \hskip-0.08in \le & \hskip-0.08in
%\Big\|\sum_{j=2}^n{\bf A}_j({\bf x}_j^{k+1}-\bar{{\bf x}}_j^k)\Big\|_2^2 \nonumber\\
%& & -2 {\beta}^{-1}  \sum_{j=2}^n
%({\bf z}^{k+1}-\bar{{\bf z}}^k)^T  {\bf A}_j({\bf x}_j^{k+1}-\bar{{\bf x}}_j^k)
%+ \beta^{-2}\|{\bf z}^{k+1}-\bar{{\bf z}}^k\|_2^2\nonumber \\
 \beta^{-2}\|{\bf z}^{k+1}-\bar{{\bf z}}^k\|_2^2  \nonumber\\
& \hskip-0.08in & \hskip-0.08in
  -2 {\beta}^{-1}  \sum_{j=2}^n
({\bf z}^{k+1}-\bar{{\bf z}}^k)^T  {\bf A}_j({\bf x}_j^{k+1}-\bar{{\bf x}}_j^k)
\nonumber\\
\hskip-0.08in& \hskip-0.08in & %\quad
+ \frac{1}{\beta}\sum_{j=2}^l ({\bf x}_j^{k+1}-\bar{{\bf x}}_j^k)^T ({\bf H}_j+ \beta{\bf A}_j^{T}{\bf A}_j)
({\bf x}_j^{k+1}-\bar{{\bf x}}_j^k)  %\nonumber\\
%& = &  \frac{1}{\beta}\|{\bf w}^{k+1}-\bar{{\bf w}}^k\|_{{\bf G}}^2
%- \frac{1}{\beta}({\bf x}_1^{k+1}-\bar{{\bf x}}_1^k)^T {\bf H}_1
%({\bf x}_1^{k+1}-\bar{{\bf x}}_1^k)
\nonumber\\
& \hskip-0.08in  \le & \hskip-0.08in   \frac{1}{\beta}\|{\bf w}^{k+1}-\bar{{\bf w}}^k\|_{{\bf G}}^2,
\end{eqnarray}
where the first inequality follows from \eqref{hi.condition2} and positive semidefiniteness of the matrix ${\bf G}_2$,
the second equality holds by \eqref{Proximalregularizationmatrix}, and
and the last inequality is true as ${\bf H}_1\succeq {\bf 0}$.
The above estimate together with \eqref{ConvergenceLemma1.eq1} in  Theorem \ref{WeakConvergenceThm} proves
 \eqref{ConvergenceTheorem1.eq1}.
\end{proof}

\begin{rem}\label{G2positive.rem}
{\rm
The positive semi-definite  requirement for the matrix ${\mathbf G}_2$ is met if
 ${\bf H}_j\succeq {\bf O}, 1\le j\le l$, are chosen appropriately.
Clearly,   ${\mathbf G}_2\succeq {\bf 0}$ %if and only if
%\begin{align*}
%\begin{bmatrix}
%{\bf H}_2&         -\beta{\bf A}_2^{T}{\bf A}_3&         \cdots&    -\beta{\bf A}_2^{T}{\bf A}_l\\ -\beta({\bf A}_2^{T}{\bf A}_3)^{T}&  {\bf H}_3&          \cdots&    -\beta{\bf A}_3^{T}{\bf A}_l\\
%\vdots&                  \vdots&                                  \ddots&         \vdots\\
%-\beta({\bf A}_2^{T}{\bf A}_l)^{T}&  -\beta({\bf A}_3^{T}{\bf A}_l)^{T}&   \cdots&    {\bf H}_l
%\end{bmatrix}\succeq {\bf 0}
%\end{align*}
if and only if
\begin{equation}\label{hi.condition2}
\sum_{j=2}^l {\bf u}_j^T ({\bf H}_j+ \beta{\bf A}_j^{T}{\bf A}_j) {\bf u}_j-\beta\Big\|\sum_{j=2}^l {\bf A}_j {\bf u}_j\Big\|_2^2\ge 0
\ \ {\rm for \ all} \ \ {\bf u}_j\in {\bf R}^{n_j}, 2\le j\le l.
\end{equation}
From the above argument, we see that  ${\mathbf G}_2\succeq {\bf 0}$
is satisfied  %  the  proximal regularization matrix   ${\bf G}$  is positive semi-definite
if  \eqref{IADMMConvergence3.cor.eq0}  holds for   ${\bf H}_j, 1\le j\le l$. %is satisfied.
One may also verify from  \eqref{hi.condition2} that ${\mathbf G}_2$  is  positive semidefinite
if  the  prox-linear ${\bf H}_j, 1\le j\le l$, in \eqref{prox-linear.def} satisfies
\begin{equation*}\label{assumption3-strong2}
0<\eta_1\leq \|{\bf A}_1\|_{2\rightarrow 2}^{-2}\ \ {\rm  and} \ \  0<\eta_j\leq (l-1)^{-1}\|{\bf A}_j\|_{2\rightarrow 2}^{-2},\  2\le j\le l,
\end{equation*}
and
if the standard proximal  ${\bf H}_j, 1\le j\le l$ in \eqref{standardproximal.def} satisfies
\begin{equation*}\label{assumption3-strong3}
0<\eta_1\leq \|{\bf A}_1\|_{2\rightarrow 2}^{-2}\  \ {\rm and} \ \ 0<\eta_j\leq (l-2)^{-1}\|{\bf A}_j\|_{2\rightarrow 2}^{-2},\ 2\le j\le l.
\end{equation*}
}
\end{rem}

For the case that the step size $\alpha_k, k\ge 0$ are chosen in \eqref{Prop-Guarantee-Assumption1.eq1},
we obtain from \eqref{InertialProximalADMM-Multi-1} and \eqref{Prop-Guarantee-Assumption1.eq2} that
\begin{eqnarray}\label{Prop-Guarantee-Assumption1.eq3+}
\sum_{k=1}^\infty \|{\bf w}^{k+1}-\bar{{\bf w}}^{k}\|_{{\bf G}}^2
 &  \hskip-0.08in \le &
\hskip-0.08in \sum_{k=1}^\infty
2\|{\bf w}^{k+1}-{{\bf w}}^{k}\|_{{\bf G}}^2+ 2\alpha_k^2\|{\bf w}^{k}-{{\bf w}}^{k-1}\|_{{\bf G}}^2
\nonumber\\
&  \hskip-0.08in \le & \hskip-0.08in
\frac{2(1+\alpha^2)}{(1-3\alpha)(1-\alpha)}\|{\bf w}^0-{\bf w}^{*}\|_{{\bf G}}^2
\end{eqnarray}
and
\begin{equation}
\min_{1\leq i\leq k}\|{\bf w}^i-\bar{{\bf w}}^{i-1}\|_{{\bf G}} %&   \leq &
\le \sqrt{\frac{2(1+\alpha^2)}{(1-3\alpha)(1-\alpha)}} \|{\bf w}^0-{\bf w}^{*}\|_{{\bf G}} k^{-1/2}, \ k\ge 1.
\end{equation}
By  \eqref{ConvergenceTheorem1.pf.eq1} and  \eqref{Prop-Guarantee-Assumption1.eq3+}
we obtain a strong estimate about feasibility of the Prox-IADMM.

\begin{cor} \label{ConvergenceTheorem1.cor2}
 Let
$ {\bf G}, {\bf G}_2, {\bf H}_j$ and $\{{\bf x}_j^{k}\}_{k=0}^{\infty}, 1\le j\le l$,
 be as in  Theorem \ref{ConvergenceTheorem1}, and let $\alpha_k, k\ge 0$ be as in \eqref{Prop-Guarantee-Assumption1.eq1}.
 Then
 \begin{equation}\label{ConvergenceTheorem1.cor2.eq1}
\sum_{k=1}^{\infty}\Big\|\sum_{j=1}^l{\bf A}_j {\bf x}_j^k-{\bf c}\Big\|_2^2\le
 \frac{1}{\beta}\left(\frac{\alpha(1+\alpha)}{(1-3\alpha)(1-\alpha)^2}
 +1\right)
 \|{\bf w}^0-{\bf w}^{*}\|_{{\bf G}}^2.
\end{equation}
\end{cor}

By Corollary \ref{ConvergenceTheorem1.cor2}, we have the following
nonasymptotic  convergence rate for the residual of constraint $\|\sum_{j=1}^{l}{\bf A}_j{\bf x}_j-{\bf c}\|_{2}$,
\begin{equation}
 \min_{1\leq i\leq k}\Big\|\sum_{j=1}^{l}{\bf A}_j{\bf x}_j^i-{\bf c}\Big\|_{2}\le \sqrt{\frac{1}{\beta}\left(\frac{\alpha(1+\alpha)}{(1-3\alpha)(1-\alpha)^2}
 +1\right)}\|{\bf w}^0-{\bf w}^{*}\|_{{\bf G}} k^{-1/2},
\end{equation}
which is given in \cite[Theorem  4.6]{CCMY2015}.

\medskip

We finish this section with the proof of
 the conclusion (i)  in  Theorem \ref{IADMMConvergence3.mainthm}.

 \begin{proof} [Proof of the first conclusion in Theorem \ref{IADMMConvergence3.mainthm}]
By Remarks \ref{Gpositive.rem} and \ref{G2positive.rem},
the  positive semi-definite requirements for ${\bf G}$ and  ${\mathbf G}_2$
in Theorem \ref{ConvergenceTheorem1}
are met. Therefore the desired limit  \eqref{ConvergenceTheorem1.eq2} follows from \eqref{ConvergenceTheorem1.eq1}.\end{proof}

\subsection{Convergence of objective functions}\label{s3.1b}

In this section, we prove the following version of the second conclusion  of Theorem \ref{IADMMConvergence3.mainthm}
under a weak version that  ${\bf G}$ and ${\bf G}_2$ are positive semi-definite.
%{\color{blue}We give the proof detail in the in the supplementary material \cite{LCS2019}.}

\begin{theorem}\label{ConvergenceTheorem1.part2}
Let  ${\bf x}_j^k, 1\le j\le l, k\ge 0$ be
as the inertial proximal ADMM  \eqref{InertialProximalADMM-Multi-1}
and \eqref{InertialProximalADMM-Multi}.
If matrices
${\bf G}, {\bf G}_2, {\bf H}_j, 1\le j\le l$, and  $\alpha_k, k\ge 0$
be as in Theorem \ref{ConvergenceTheorem1},
 then
% \begin{equation}\label{ConvergenceTheorem1.part2.eq00}
% \sum_{k=1}^\infty \Big|
%  \sum_{j=1}^lf_j({\bf x}_j^k)- \min_{ {\bf x}_j\in \mathcal{X}^j, 1\le j\le l\ {\rm  and} \  \sum_{j=1}^{l}{\bf A}_j{\bf x}_j={\bf c}}  \sum_{j=1}^lf_j({\bf x}_j)\Big|^2<\infty.
%\end{equation}
%then
 \begin{equation}\label{bound}
|\theta({\bf w}^*)-\theta({\bf w}^{k+1})|\le (\| {\bf w}^*-{\bf w}^{k+1} \|_{\bf G}+\|{\bf z}^*\|) \|{\bf w}^{k+1}-\bar{{\bf w}}^k\|_{\bf G}
\end{equation}
and
 \begin{equation}\label{ConvergenceTheorem1.part2.eq00}
 \sum_{k=1}^\infty \Big|
  \sum_{j=1}^lf_j({\bf x}_j^k)- \min_{ {\bf x}_j\in \mathcal{X}^j, 1\le j\le l\ {\rm  and} \  \sum_{j=1}^{l}{\bf A}_j{\bf x}_j={\bf c}}  \sum_{j=1}^lf_j({\bf x}_j)\Big|^2<\infty.
\end{equation}
\end{theorem}

\begin{proof} %[Proof of  Theorem \ref{ConvergenceTheorem1.part2}]
 We follow the argument used in \cite[Theorem 4.3]{CCMY2015} where $l=2$.
  Take ${\bf w}^*=\big(({\bf x}_1^*)^T,\ldots,({\bf x}_l^*)^T;({\bf z}^*)^T\big)^T\in {\mathcal W}^*$.   Then
  \begin{equation}\label{ConvergenceTheorem1.part2.pf.eq1}
  \sum_{j=1}^l{\bf A}{\bf x}_j^*={\bf c}
\end{equation}
and
\begin{equation}\label{ConvergenceTheorem1.part2.pf.eq1+}
\sum_{j=1}^lf_j({\bf x}_j^*)=\min_{{\bf x}_j\in \mathcal{X}^j, 1\le j\le l \ {\rm  and} \  \sum_{j=1}^{l}{\bf A}_j{\bf x}_j={\bf c} } \ \sum_{j=1}^lf_j({\bf x}_j).
\end{equation}
Applying  \eqref{MixedVI1} with ${\bf w}=\big(({\bf x}_1^{k+1})^T,\ldots,({\bf x}_l^{k+1})^T;({\bf z}^*)^T\big)^T$
and using   \eqref{ConvergenceTheorem1.pf.eq1} and \eqref{ConvergenceTheorem1.part2.pf.eq1}, we obtain
\begin{eqnarray}\label{Lowerbound}
\theta({\bf w}^{k+1})- \theta({\bf w}^*)
 &  \hskip-0.08in \ge &   \hskip-0.08in \Big\langle\sum_{j=1}^l{\bf A}{\bf x}_j^{k+1}-{\bf c},{\bf z}^*\Big\rangle\nonumber\\
 &   \hskip-0.08in \ge &  \hskip-0.08in
 - \Big\|\sum_{j=1}^l{\bf A}{\bf x}_j^{k+1}-{\bf c}\Big\| \|{\bf z}^*\|
   \ge - \|{\bf w}^{k+1}-\bar{{\bf w}}^k\|_{{\bf G}} \|{\bf z}^*\|.
\end{eqnarray}

Applying \eqref{MixedVI2} in Theorem \ref{MixedVITheorem}   with ${\bf w}$ replaced by ${\bf w}^*$, and then using
\eqref{ConvergenceTheorem1.pf.eq1},
we obtain
\begin{eqnarray}\label{Preprocessing-Upperbound-1}
%\sum_{j=1}^lf_j({\bf x}_j^*)-\sum_{j=1}^lf_j({\bf x}_j^{k+1})
%%=
\theta({\bf w}^*)-\theta({\bf w}^{k+1}) %\nonumber\\
&  \hskip-0.08in \geq &   \hskip-0.08in -\big\langle {\bf w}^*-{\bf w}^{k+1},{\bf G}({\bf w}^{k+1}-\bar{{\bf w}}^k)\big\rangle
-\big\langle {\bf w}^*-{\bf w}^{k+1},F({\bf w}^{k+1})\big\rangle\nonumber\\
% &   = &   -\big\langle {\bf w}^*-{\bf w}^{k+1},{\bf G}({\bf w}^{k+1}-\bar{{\bf w}}^k)\big\rangle
%-\Big\langle\sum_{j=1}^l{\bf A}_j{\bf x}_j^{k+1}-{\bf c},{\bf z}^*\Big\rangle\nonumber\\
&   \hskip-0.08in \ge &  \hskip-0.08in -\| {\bf w}^*-{\bf w}^{k+1} \|_{\bf G}\|{\bf w}^{k+1}-\bar{{\bf w}}^k\|_{\bf G}-
 \Big\|\sum_{j=1}^l{\bf A}_j{\bf x}_j^{k+1}-{\bf c}\Big\| \|{\bf z}^*\|\nonumber\\
&   \hskip-0.08in \ge &  \hskip-0.08in - (\| {\bf w}^*-{\bf w}^{k+1} \|_{\bf G}+\|{\bf z}^*\|) \|{\bf w}^{k+1}-\bar{{\bf w}}^k\|_{\bf G}.
\end{eqnarray}
Combining \eqref{Lowerbound} and \eqref{Preprocessing-Upperbound-1} gives
\begin{equation}\label{bound**}
|\theta({\bf w}^*)-\theta({\bf w}^{k+1})|\le (\| {\bf w}^*-{\bf w}^{k+1} \|_{\bf G}+\|{\bf z}^*\|) \|{\bf w}^{k+1}-\bar{{\bf w}}^k\|_{\bf G}.
\end{equation}
This together with  Theorem  \ref{WeakConvergenceThm} completes the proof.
\end{proof}

If the step sizes $\alpha_k, k\ge 0$, are as chosen in \eqref{Prop-Guarantee-Assumption1.eq1},
we obtain the following  corollary about
convergence of objective functions in the Prox-IADMM.

\begin{cor} \label{objectivefunction.cor2+}
Let matrices
${\bf G}, {\bf G}_2, {\bf H}_j, 1\le j\le l$, and  $\alpha_k, k\ge 0$
be as in Theorem \ref{ConvergenceTheorem1},
 ${\bf x}_j^k, 1\le j\le l, k\ge 0$ be
as the inertial proximal ADMM  \eqref{InertialProximalADMM-Multi-1}
and \eqref{InertialProximalADMM-Multi}, and let $\alpha_k, k\ge 0$ be chosen in \eqref{Prop-Guarantee-Assumption1.eq1}.
Then
 \begin{eqnarray}\label{objectivefunction.cor2+.eq1}
& &  \sum_{k=1}^\infty \Big|
  \sum_{j=1}^lf_j({\bf x}_j^k)- \min_{ {\bf x}_j\in \mathcal{X}^j, 1\le j\le l\ {\rm  and} \  \sum_{j=1}^{l}{\bf A}_j{\bf x}_j={\bf c}}  \sum_{j=1}^lf_j({\bf x}_j)\Big|^2\nonumber\\
  & \hskip-0.08in \le & \hskip-0.08in
  \frac{2\alpha}{(1-3\alpha)^2 (1-\alpha)^3} \big( \|{\bf w}^{0}-{\bf w}^*\|_{{\bf G}}^2+
\|{\bf z}^*\|^2 \big)\|{\bf w}^0-{\bf w}^{*}\|_{{\bf G}}^2.
\end{eqnarray}
\end{cor}

%\begin{proof}  %[Proof of Corollary \ref{objectivefunction.cor2+} ]
%By  \eqref{ConvergenceLemma1.eq2} and \eqref{Prop-Guarantee-Assumption1.eq3+}, we have
%\begin{align}\label{e3.6}
%\|{\bf w}^{k}-{\bf w}^*\|_{G}^2
%&\leq\|{\bf w}^{0}-{\bf w}^*\|_{{\bf G}}^2
%+\frac{1+\alpha}{1-\alpha}\sum_{j=1}^{\infty}\alpha_j\|{\bf w}^{j}-{\bar{{\bf w}}}^{j-1}\|_{{\bf G}}^2\nonumber\\
%&\leq\bigg(1+\frac{\alpha(1+\alpha)}{(1-\alpha)^2(1-3\alpha)}\bigg) \|{\bf w}^{0}-{\bf w}^*\|_{{\bf G}}^2\nonumber\\
%&=
%\frac{1-4\alpha+8\alpha^2-3\alpha^2}{(1-\alpha)^2(1-3\alpha)} \|{\bf w}^{0}-{\bf w}^*\|_{{\bf G}}^2.
%\end{align}
%Combining the above estimate with \eqref{bound} in the proof of Theorem \ref{ConvergenceTheorem1.part2}, we obtain
%\begin{eqnarray*}
%& &\sum_{k=1}^\infty\bigg|\sum_{j=1}^lf_j({\bf x}_j^k)-\sum_{j=1}^lf_j({\bf x}_j^*)\bigg|^2\\
%&\leq & \sum_{k=1}^\infty (\|{\bf w}^{k} -{\bf w}^*\|_{\bf G}+\|{\bf z}^*\|)^2 \|{\bf w}^{k}-\bar{{\bf w}}^{k-1}\|_{\bf G}^2\\
%&\leq & \Big(\frac{2-8\alpha+16\alpha^2-6\alpha^2}{(1-\alpha)^2(1-3\alpha)} \|{\bf w}^{0}-{\bf w}^*\|_{{\bf G}}^2
%+2\|{\bf z}^*\|^2\Big) \frac{\alpha}{(1-3\alpha)(1-\alpha)}\|{\bf w}^0-{\bf w}^{*}\|_{{\bf G}}^2
%\\
%&\leq & \frac{2\alpha}{(1-3\alpha)^2 (1-\alpha)^3} \big( \|{\bf w}^{0}-{\bf w}^*\|_{{\bf G}}^2+
%\|{\bf z}^*\|^2 \big)\|{\bf w}^0-{\bf w}^{*}\|_{{\bf G}}^2.
%\end{eqnarray*}
%This proves \eqref{objectivefunction.cor2+.eq1}.
%\end{proof}

\begin{rem}\label{ConvergencerateRemark}{\rm
For the case that  $\alpha_k, k\ge 0$ are chosen in \eqref{Prop-Guarantee-Assumption1.eq1}, we obtain
 from \eqref{objectivefunction.cor2+.eq1} that
\begin{eqnarray} \label{objectivefunction.cor2+.eq2}
 \quad & \hskip-0.08in   &  \hskip-0.08in  \min_{1\le i\le k} \Big|
  \sum_{j=1}^lf_j({\bf x}_j^i)- \min_{ {\bf x}_j\in \mathcal{X}^j, 1\le j\le l\ {\rm  and} \  \sum_{j=1}^{l}{\bf A}_j{\bf x}_j={\bf c}}  \sum_{j=1}^lf_j({\bf x}_j)\Big|\\
 \qquad  & \hskip-0.08in  \le & \hskip-0.08in   %\nonumber
   \sqrt{\frac{2\alpha}{(1-3\alpha)^2 (1-\alpha)^{3}}} \Big( \|{\bf w}^{0}-{\bf w}^*\|_{\bf G}+
\|{\bf z}^*\| \Big)\|{\bf w}^0-{\bf w}^{*}\|_{{\bf G}} k^{-1/2} %, \ \ k\ge 1.
 \nonumber
\end{eqnarray}
hold for all $k\ge 1$. We remark that
the above conclusion about  nonasymptotic convergence rate
 about  objective functions has been given in \cite[Theorem 4.6]{CCMY2015}.
}\end{rem}

\begin{proof}[Proof of the second  conclusion in Theorem \ref{IADMMConvergence3.mainthm}]
By Remarks \ref{Gpositive.rem} and \ref{G2positive.rem},
the  positive semi-definite requirements for ${\bf G}$ and  ${\mathbf G}_2$
in Theorem \ref{ConvergenceTheorem1}
are met.
 Therefore the desired limit  \eqref{objectivefunctionconvergence} follows from \eqref{ConvergenceTheorem1.part2.eq00}.
\end{proof}

\subsection{Boundedness of the Prox-IADMM}\label{boundedness.subsection}

In this section, we consider the boundedness of
${\bf w}^k, k\ge 0$, in the Prox-IADMM \eqref{InertialProximalADMM-Multi-1}
and \eqref{InertialProximalADMM-Multi}.

\begin{theorem}\label{ConvergenceTheorem2}
 Let matrices
${\bf G}, {\bf G}_2$ and ${\bf H}_j, 1\le j\le l$, and the family
$\alpha_k, k\ge 0$ of step sizes be as in Theorem \ref{ConvergenceTheorem1},
and let  ${\bf x}_j^k, 1\le j\le l$ and ${\bf z}^k, k\ge 0$, be
as the Prox-IADMM  \eqref{InertialProximalADMM-Multi-1}
and \eqref{InertialProximalADMM-Multi}.
 Then
 \begin{eqnarray}\label{ConvergenceTheorem2.eq1}
& &    \sum_{j=1}^l \|{\bf x}_j^k\|_{{\bf H}_j+\beta {\bf A}_j^{T}{\bf A}_j}^2 + \frac{1}{\beta}\|{\bf z}^k\|_2^2 \\
 &  \hskip-0.08in \le  &\hskip-.08in   30\beta\|{\bf A}_1{\bf x}_1^*\|_2^2+12\beta\| {\bf x}_1^*\|_{{\bf H}_1}^2+36\beta \|{\bf c}\|^2  +30\beta^{-1}\|{\bf z}^*\|_2^2 +100 \|{\bf w}^*\|_{{\bf G}}^2 \nonumber\\
& \hskip-0.08in  &  \hskip-0.08in  +142 \|{\bf w}^0-{\bf w}^*\|_{\bf G}^2   +   \frac{142(1+\alpha)}{1-\alpha}\sum_{j=0}^{\infty} \alpha_j
\|{\bf w}^j-{\bf w}^{j-1}\|_{{\bf G}}^2<\infty.\nonumber
\end{eqnarray}
\end{theorem}

\begin{proof}  By the definition  \eqref{Proximalregularizationmatrix} of the matrix ${\bf G}$, we have
\begin{align}\label{ConvergenceTheorem2.pf.eq9}
\|{\bf w}^k\|_{{\bf G}}^2&=\sum_{j=2}^l \|{\bf x}_j^k\|_{{\bf H}_j+\beta {\bf A}_j^{T}{\bf A}_j}^2+\|{\bf x}_1^k\|_{{\bf H}_1}^2
+\beta\Big\|\sum_{j=2}^l{\bf A}_j{\bf x}_j^k-\frac{{\bf z}^k}{\beta}\Big\|_2^2
-\beta\Big\|\sum_{j=2}^l{\bf A}_j{\bf x}_j^k\Big\|_2^2.
\end{align}
Therefore
\begin{eqnarray} \label{ConvergenceTheorem2.pf.eq9+1-1}
& &\sum_{j=1}^l \|{\bf x}_j^k\|_{{\bf H}_j+\beta {\bf A}_j^{T}{\bf A}_j}^2+\frac{1}{\beta}\|{\bf z}^k\|_2^2 \nonumber\\
& \hskip-0.08in  \le  & \hskip-0.08in
 \|{\bf w}^k\|_{{\bf G}}^2+ \beta\|{\bf A}_1{\bf x}_1^k\|_2^2+ \beta\Big\|\sum_{j=2}^l{\bf A}_j{\bf x}_j^k\Big\|_2^2 +\frac{1}{\beta}\|{\bf z}^k\|_2^2
-\beta\Big\|\sum_{j=2}^l{\bf A}_j{\bf x}_j^k-\frac{{\bf z}^k}{\beta}\Big\|_2^2
 \nonumber\\
%&   \le  &
%\|{\bf w}^k\|_{{\bf G}}^2+ \beta\|{\bf A}_1{\bf x}_1^k\|_2^2+\beta\Big\|\sum_{j=2}^l{\bf A}_j{\bf x}_j^k-\frac{{\bf z}^k}{\beta}\Big\|_2^2+\frac{3}{\beta}\|{\bf z}^k\|_2^2\nonumber\\
&  \hskip-0.08in \le  & \hskip-0.08in \beta\Big\|\sum_{j=2}^l{\bf A}_j{\bf x}_j^k-\frac{{\bf z}^k}{\beta}\Big\|_2^2+ 6 \big(\beta\|{\bf A}_1({\bf x}_1^k-{\bf x}_1^*)\|_2^2+\beta^{-1}\|{\bf z}^k-{\bf z}^*\|_2^2\big)
+ \|{\bf w}^k\|_{{\bf G}}^2 \nonumber\\
& &  \quad  +6\beta\|{\bf A}_1{\bf x}_1^*\|_2^2+6\beta^{-1}\|{\bf z}^*\|_2^2,
\end{eqnarray}
where the first inequality  follows from  \eqref{ConvergenceTheorem2.pf.eq9},  and the second  inequality is obtained by applying the elementary inequality $(a+b)^2\leq2(a^2+b^2)$.

Next we estimate $\beta\|\sum_{j=2}^l{\bf A}_j{\bf x}_j^k-{\bf z}^k/{\beta}\|_2^2$.
%Observe from the equivalent condition  \eqref{hi.condition2} for the positive semi-definiteness of ${\bf G}_2$ that
%\begin{equation}\label{ConvergenceTheorem2.pf.eq2}
%0\le \sum_{i=2}^l \|{\bf x}_j^k\|_{{\bf H}_j+\beta {\bf A}_j^{T}{\bf A}_j}^2 -\beta \Big\|\sum_{i=2}^l {\bf A}_i {\bf x}_i^k \Big\|_2^2.
%\end{equation}
By \eqref{ConvergenceTheorem2.pf.eq9} and the equivalent condition  \eqref{hi.condition2} for the positive semi-definiteness of ${\bf G}_2$, we have
\begin{equation}\label{ConvergenceTheorem2.pf.eq1}
 \beta\Big\|\sum_{j=2}^{l}{\bf A}_j{\bf x}_j^k-\frac{{\bf z}^k}{\beta}\Big\|_2^2\le \|{\bf w}^k\|_{{\bf G}}^2.
\end{equation}
Hence by \eqref{ConvergenceTheorem2.pf.eq9+1-1} and \eqref{ConvergenceTheorem2.pf.eq1},
\begin{eqnarray} \label{ConvergenceTheorem2.pf.eq9+1}
\sum_{j=1}^l \|{\bf x}_j^k\|_{{\bf H}_j+\beta {\bf A}_j^{T}{\bf A}_j}^2+\frac{1}{\beta}\|{\bf z}^k\|_2^2
&\hskip-0.08in \le &  \hskip-0.08in  6 \big(\beta\|{\bf A}_1({\bf x}_1^k-{\bf x}_1^*)\|_2^2 +\beta^{-1}\|{\bf z}^k-{\bf z}^*\|_2^2\big)\nonumber\\
 &    &
+ 2\|{\bf w}^k\|_{{\bf G}}^2+6\beta\|{\bf A}_1{\bf x}_1^*\|_2^2+6\beta^{-1}\|{\bf z}^*\|_2^2.
\end{eqnarray}

Let's turn our attention to estimate $\beta\|{\bf A}_1({\bf x}_1^k-{\bf x}_1^*)\|_2^2+\beta^{-1}\|{\bf z}^k-{\bf z}^*\|_2^2$.
%Following the argument used to establish  \eqref{e3.2}
Applying \eqref{MultiproximalADMM-2} and
\eqref{InertialProximalADMM-Multi-2},
we obtain
\begin{equation}  \label{ConvergenceTheorem2.pf.eq14}
f_1({\bf x}_1^*)-f_1({\bf x}_1^{k+1})+
({\bf x}_1^*-{\bf x}_1^{k+1})^{T}\big(-{\bf A}_1^{T}{\bf z}^{k+1}
+{\bf H}_1({\bf x}_1^{k+1}-\bar{{\bf x}}_1^k)\big)\geq 0.
\end{equation}
Applying \eqref{MixedVI1} with $\bf w$ replacing by $( {\bf x}_1^{k+1}, {\bf x}_2^*, \ldots, {\bf x}_l^*; {\bf z}^*)$, we have
\begin{equation}  \label{ConvergenceTheorem2.pf.eq15}
f_1({\bf x}_1^{k+1})-f_1({\bf x}_1^*)-({\bf x}_1^{k+1}- {\bf x}_1^*)^{T} {\bf A}_1^T {\bf z}^*
\geq 0.
\end{equation}
Summing up the estimates in \eqref{ConvergenceTheorem2.pf.eq14} and \eqref{ConvergenceTheorem2.pf.eq15} gives
\begin{equation}  \label{ConvergenceTheorem2.pf.eq5}
({\bf x}_1^{k+1}- {\bf x}_1^*)^{T} {\bf A}_1^T ( {\bf z}^{k+1}-{\bf z}^*)\ge
({\bf x}_1^{k+1}- {\bf x}_1^*)^{T} {\bf H}_1
({\bf x}_1^{k+1}-{\bar{{\bf x}}_1}^k).
\end{equation}
Therefore
\begin{eqnarray}\label{ConvergenceTheorem2.pf.eq16}
&  & \beta\|{\bf A}_1({\bf x}_1^k-{\bf x}_1^*)\|_2^2+\beta^{-1}\|{\bf z}^k-{\bf z}^*\|_2^2\nonumber\\
%&  = &   \beta\|({\bf A}_1{\bf x}_1^k+\beta^{-1}{\bf z}^k)
%-({\bf A}_1{\bf x}_1^*+\beta^{-1}{\bf z}^*)\|_2^2
%-2  ({\bf x}_1^k-{\bf x}_1^*)^T {\bf A}_1^T ({\bf z}^k-{\bf z}^*)\nonumber\\
& \hskip-0.08in \le  & \hskip-0.08in \|({\bf A}_1{\bf x}_1^k+\beta^{-1}{\bf z}^k)
-({\bf A}_1{\bf x}_1^*+\beta^{-1}{\bf z}^*)\|_2^2
-2({\bf x}_1^{k}- {\bf x}_1^*)^{T} {\bf H}_1
({\bf x}_1^{k}-{\bar{{\bf x}}_1}^{k-1})\nonumber\\
& \hskip-0.08in \leq  &   \hskip-0.08in 2\beta \|{\bf A}_1{\bf x}_1^k+\beta^{-1}{\bf z}^k\|_2^2+ 2\beta \|{\bf A}_1{\bf x}_1^*+\beta^{-1}{\bf z}^*\|_2^2\nonumber\\
& &
+  \|{\bf x}_1^{k}- {\bf x}_1^*\|_{{\bf H}_1}^2+ \|{\bf x}_1^{k}-{\bar{{\bf x}}_1}^{k-1}\|_{{\bf H}_1}^2\nonumber\\
& \hskip-0.08in \leq  &  \hskip-0.08in 2\beta\|{\bf A}_1{\bf x}_1^k+\beta^{-1}{\bf z}^k\|_2^2+2 \|{\bf w}^{k}\|_{{\bf G}}^2
+\|{\bf w}^{k}-{\bar{{\bf w}}}^{k-1}\|_{{\bf G}}^2 \nonumber\\
& &
 +4\beta\|{\bf A}_1{\bf x}_1^*\|_2^2
+4\beta^{-1}\|{\bf z}^*\|_2^2+2\beta\| {\bf x}_1^*\|_{{\bf H}_1}^2,
\end{eqnarray}
where the first inequality holds by \eqref{ConvergenceTheorem2.pf.eq5}, and the third inequality
follows as  $\|{\bf x}_1^k\|_{{\bf H}_1}\le \|{\bf w}^k\|_{\bf G}$ and
 $\|{\bf x}_1^{k}-{\bar{{\bf x}}_1}^{k-1}\|_{{\bf H}_1}\leq \|{\bf w}^{k}-{\bar{{\bf w}}}^{k-1}\|_{{\bf G}}$
 by the definition  \eqref{Proximalregularizationmatrix} of the matrix ${\bf G}$ and Assumption \ref{Gge0.assumption}.
Combining
 \eqref{ConvergenceTheorem2.pf.eq9+1} and \eqref{ConvergenceTheorem2.pf.eq16}, we obtain
 \begin{eqnarray} \label{ConvergenceTheorem2.pf.eq9+11}
 & &
   \sum_{j=1}^l \|{\bf x}_j^k\|_{{\bf H}_j+\beta {\bf A}_j^{T}{\bf A}_j}^2 + \frac{1}{\beta}\|{\bf z}^k\|_2^2\nonumber\\
& \hskip-0.08in   \le  & \hskip-0.08in
12\beta\|{\bf A}_1{\bf x}_1^k+\beta^{-1}{\bf z}^k\|_2^2+14 \|{\bf w}^{k}\|_{{\bf G}}^2
+6\|{\bf w}^{k}-{\bar{{\bf w}}}^{k-1}\|_{{\bf G}}^2\nonumber\\
& &   +30 \beta\|{\bf A}_1{\bf x}_1^*\|_2^2+12\beta\| {\bf x}_1^*\|_{{\bf H}_1}^2+30\beta^{-1}\|{\bf z}^*\|_2^2.
\end{eqnarray}

Now we estimate $\|{\bf A}_1{\bf x}_1^k+\beta^{-1}{\bf z}^k\|_2^2$.
By \eqref{ConvergenceTheorem1.pf.eq1} and\eqref{ConvergenceTheorem2.pf.eq1}, we have
\begin{eqnarray}\label{ConvergenceTheorem2.pf.eq17}
\beta\|{\bf A}_1{\bf x}_1^k+\beta^{-1}{\bf z}^k\|_2^2
&  \hskip-0.08in  \leq &  \hskip-0.08in  3\beta\Big\|\sum_{j=1}^l{\bf A}_j{\bf x}_j^k-{\bf c}\Big\|_2^2+ 3\beta\Big\|\frac{{\bf z}^k}{\beta}-\sum_{j=2}^l{\bf A}_j{\bf x}_j^k\Big\|_2^2+ 3\beta \|{\bf c}\|_2^2
\nonumber\\
& \hskip-0.08in  \leq &  \hskip-0.08in  3\|{\bf w}^{k}-{\bar{{\bf w}}}^{k-1}\|_{{\bf G}}^2+3\|{\bf w}^k\|_{{\bf G}}^2 + 3\beta \|{\bf c}\|_2^2.
\end{eqnarray}
This together with \eqref{ConvergenceTheorem2.pf.eq9+11} implies that
 \begin{eqnarray} \label{ConvergenceTheorem2.pf.eq9+12}
   \sum_{j=1}^l \|{\bf x}_j^k\|_{{\bf H}_j+\beta {\bf A}_j^{T}{\bf A}_j}^2 + \frac{1}{\beta}\|{\bf z}^k\|_2^2
&  \hskip-0.08in  \le  &  \hskip-0.08in
50 \|{\bf w}^{k}\|_{{\bf G}}^2
+42\|{\bf w}^{k}-{\bar{{\bf w}}}^{k-1}\|_{{\bf G}}^2 +30\beta\|{\bf A}_1{\bf x}_1^*\|_2^2\nonumber\\
& &+12\beta\| {\bf x}_1^*\|_{{\bf H}_1}^2+30\beta^{-1}\|{\bf z}^*\|_2^2+36\beta \|{\bf c}\|^2.
\end{eqnarray}

Finally we estimate $\|{\bf w}^{k}\|_{{\bf G}}$ and $\|{\bf w}^{k}-{\bar{{\bf w}}}^{k-1}\|_{{\bf G}}$.
By Theorem  \ref{WeakConvergenceThm}, we obtain
\begin{eqnarray}\label{ConvergenceLemma1.pf.eq10}
 \|{\bf w}^{k}\|_{{\bf G}}^2
& \hskip-0.08in  \leq & \hskip-0.08in  2\|{\bf w}^*\|_{{\bf G}}^2+ 2\|{\bf w}^{k}-{\bf w}^*\|_{{\bf G}}^2
\nonumber\\
&  \hskip-0.08in  \leq & \hskip-0.08in   2\|{\bf w}^*\|_{{\bf G}}^2+ 2\|{\bf w}^0-{\bf w}^*\|_{\bf G}^2+  \frac{2+2\alpha}{1-\alpha}\sum_{j=0}^{\infty} \alpha_j
\|{\bf w}^j-{\bf w}^{j-1}\|_{{\bf G}}^2,
\end{eqnarray}
and
\begin{equation}\label{ConvergenceTheorem2.pf.eq13}
  \|{\bf w}^k-{\bar {\bf w}}^{k-1}\|_{\bf G}^2\le
 \Big(\|{\bf w}^0-{\bf w}^*\|_{\bf G}^2
+ \frac{1+\alpha}{1-\alpha}\sum_{k=0}^{\infty} \alpha_k
\|{\bf w}^k-{\bf w}^{k-1}\|_{\bf G}^2\Big).
\end{equation}
Then the desired conclusion \eqref{ConvergenceTheorem2.eq1} follows from \eqref{ConvergenceTheorem2.pf.eq9+12},   \eqref{ConvergenceLemma1.pf.eq10} and  \eqref{ConvergenceTheorem2.pf.eq13}.
\end{proof}

\begin{rem} {\rm  For the case that step sizes  $\alpha_k, k\ge 0$, are chosen to
satisfy \eqref{Prop-Guarantee-Assumption1.eq1}, then
\begin{equation}\label{Prop-Guarantee-Assumption1.eq2++}
\sum_{k=1}^\infty \alpha_k \|{\bf w}^k-{\bf w}^{k-1}\|_{{\bf G}}^2\le
\alpha \sum_{k=1}^{\infty}\|{\bf w}^k-{\bf w}^{k-1}\|_{{\bf G}}^2
\leq \frac{\alpha}{(1-3\alpha)(1-\alpha)}\|{\bf w}^0-{\bf w}^{*}\|_{{\bf G}}^2
\end{equation}
by  Proposition \ref{Prop-Guarantee-Assumption1}. This together with \eqref{ConvergenceTheorem2.eq1}
leads to  the following estimate
\begin{eqnarray*}\label{ConvergenceTheorem2.eq1special}
\hskip-0.08in & \hskip-0.08in&\hskip-0.08in \sum_{j=1}^l \|{\bf x}_j^k\|_{{\bf H}_j  \beta {\bf A}_j^{T}{\bf A}_j}^2 + \frac{1}{\beta}\|{\bf z}^k\|_2^2
 \le  30\beta\|{\bf A}_1{\bf x}_1^*\|_2^2+12\beta\| {\bf x}_1^*\|_{{\bf H}_1}^2+36\beta \|{\bf c}\|^2\nonumber\\
\hskip-0.08in&\hskip-0.08in  &\hskip-0.08in +30\beta^{-1}\|{\bf z}^*\|_2^2 +100 \|{\bf w}^*\|_{{\bf G}}^2 +142 \Big(1+
\frac{\alpha (1+\alpha)}{(1-3\alpha)(1-\alpha)^2}\Big)  \|{\bf w}^0-{\bf w}^*\|_{\bf G}^2<\infty.
\end{eqnarray*}
}\end{rem}

\subsection{Convergence of the Prox-IADMM} \label{convergence.subsection}

 Observe that Theorem \ref{ConvergenceTheorem1} does not ensure the convergence of ${\bf w}^{k}, k\ge 0$.
 In this section, we show the convergence conclusion of ${\bf w}^k, k\ge 0$, in Theorem \ref{IADMMConvergence3.mainthm} under the weak assumption that
 ${\bf G}, {\bf G}_2$ are positive semi-definite and
\begin{align}\label{PD-Condition1}
{\bf H}_j+\beta{\bf A}_j^{T}{\bf A}_j\succ {\bf O},~j=1,\ldots,l.
\end{align}

\begin{theorem}\label{ConvergenceTheorem3}
 Let matrices
${\bf G}, {\bf G}_2$, and  $\alpha_k, k\ge 0$
be as in Theorem \ref{ConvergenceTheorem1},
and let  ${\bf w}^k, k\ge 0$ be
as the inertial proximal ADMM  \eqref{InertialProximalADMM-Multi-1}
and \eqref{InertialProximalADMM-Multi}.
If ${\bf H}_j, 1\le j\le l$, are positive semi-definite and  satisfy
\eqref {PD-Condition1}, then there exists a unique ${\bf w}^* \in {\mathcal W}^*$ such that
%\begin{equation}
%\lim_{k\to \infty} {\bf w}^k= {\bf w}^*.
%\end{equation}
$\lim_{k\to \infty} {\bf w}^k= {\bf w}^*$.
\end{theorem}

The third conclusion in Theorem \ref{IADMMConvergence3.mainthm} follows easily from
Theorem \ref{ConvergenceTheorem3}, and  Remarks \ref{Gpositive.rem} and \ref{G2positive.rem}.
Then it remains to prove Theorem \ref{ConvergenceTheorem3}.

\begin{proof}  [Proof of Theorem \ref{ConvergenceTheorem3}]
By \eqref{PD-Condition1} and  Theorem \ref{ConvergenceTheorem2}, the sequence ${\bf w}^k, k\ge 0$, is  bounded and hence it has limit points.

Take a limit point  ${\bf w}^*$ of the sequence ${\bf w}^k, k\ge 0$. As the sequence   is contained in
${\mathcal W}$ and the set ${\mathcal W}$ is closed, we have that
${\bf w}^*\in {\mathcal W}.$
Let ${\bf w}_{k_j}, j\ge 1$ be a  convergent subsequence which has limit ${\bf w}^*$.
Taking the limit over $k=k_j$ in \eqref{MixedVI3} and applying the observation that
$\lim_{k\to \infty} \|{\bf w}^{k}-{\bar{\bf w}}^{k-1}\|_{\bf G}=0$ by \eqref{ConvergenceLemma1.eq1},
 we obtain
$$
\theta({\bf w})-\theta({\bf w}^*)+\langle {\bf w}-{\bf w}^*,F({\bf w}^*) \rangle\geq 0, \ {\bf w}\in {\mathcal W}.
$$
 This implies that
 %\begin{equation}\label{ConvergenceTheorem3.pf.eq1}
 ${\bf w}^*\in\mathcal{W}^*$ and hence any limit point of  the sequence ${\bf w}^k, k\ge 0$ lie in $\mathcal{W}^*$.

Now we prove the uniqueness of the limit points.
Let  ${\bf w}_{1}^*$ and ${\bf w}_2^*$ be two limits points of the sequence ${\bf w}^k, k\ge 0$.
This together with the observation that
\begin{eqnarray*}
 &  \hskip-0.08in& \hskip-0.08in  \|{\bf w}^k-{\bf w}_1^*\|_{{\bf G}}^2-\|{\bf w}^k-{\bf w}_2^*\|_{{\bf G}}^2
=  \|{\bf w}_1^*-{\bf w}_2^*\|_{{\bf G}}^2+2\langle {\bf w}_2^*-{\bf w}_1^*,{\bf w}^k-{\bf w}_2^* \rangle_{{\bf G}}\nonumber\\
 &\hskip-0.08in  = &\hskip-0.08in -\|{\bf w}_1^*-{\bf w}_2^*\|_{{\bf G}}^2+2\langle {\bf w}_2^*-{\bf w}_1^*,{\bf w}^k-{\bf w}_1^* \rangle_{{\bf G}},\ \  k\ge 0,
\end{eqnarray*}
implies that the  sequence $\|{\bf w}^k-{\bf w}_1^*\|_{{\bf G}}^2-\|{\bf w}^k-{\bf w}_2^*\|_{{\bf G}}^2, k\ge 0$
has two limit points $\pm \|{\bf w}_1^*-{\bf w}_2^*\|_{{\bf G}}^2$.
On the other hand, it follows from Theorem  \ref{WeakConvergenceThm} that
the  sequence $\|{\bf w}^k-{\bf w}_1^*\|_{{\bf G}}^2-\|{\bf w}^k-{\bf w}_2^*\|_{{\bf G}}^2, k\ge 0$ is convergent.
Therefore two limit points  ${\bf w}_{1}^*$ and ${\bf w}_2^*$ of the sequence ${\bf w}^k, k\ge 0$ satisfy
$$\|{\bf w}_1^*-{\bf w}_2^*\|_{{\bf G}}^2=0.$$
This together with  Assumption \ref{Gge0.assumption} on ${\bf G}$ implies that
\begin{equation} \label{ConvergenceTheorem3.pf.eq2}
{\bf G} ({\bf w}_1^*-{\bf w}_2^*)={\bf 0}.
\end{equation}
Write
${\bf w}^*_t=\big(({\bf x}_{1,t}^*)^T,\ldots,({\bf x}_{l,t}^*)^T; ({\bf z}_{t}^*)^T)^T, t=1,2$.
Then it follows from \eqref{ConvergenceTheorem3.pf.eq2} that
\begin{eqnarray} \label{ConvergenceTheorem3.pf.eq3}
\qquad \left\{\begin{array}{l}
{\bf H}_1({\bf x}_{1,1}^*-{\bf x}_{1,2}^*)={\bf 0},\\ \big({\bf H}_j+\beta{\bf A}_j^{T}{\bf A}_j\big)({\bf x}_{j,1}^*-{\bf x}_{j,2}^*)- {\bf A}_j^{T}({\bf z}_{1}^*-{\bf z}_{2}^*)={\bf 0},\  j=2,\ldots,l\\
-\sum_{j=2}^{l}{\bf A}_j({\bf x}_{j,1}^*-{\bf x}_{j,2}^*)
+({\bf z}_{1}^*-{\bf z}_{2}^*)/{\beta}={\bf 0}.\end{array}
\right.
 \end{eqnarray}
 By ${\bf w}^*\in\mathcal{W}^*$, we have that $\sum_{j=1}^{l}{\bf A}_j({\bf x}_{j,1}^*-{\bf x}_{j,2}^*)={\bf c}-{\bf c}={\bf 0}$. This together with
 the third equality in \eqref{ConvergenceTheorem3.pf.eq3}
 implies that
 \begin{equation} \label{ConvergenceTheorem3.pf.eq4}
 {\bf A}_1({\bf x}_{1,1}^*-{\bf x}_{1,2}^*)+({\bf z}_{1}^*-{\bf z}_{2}^*)/{\beta}={\bf 0}.
 \end{equation}
On the other hand, applying the mixed variational property, %inequality, %(VI) characterization of the
%primal-dual optimality conditions of (\ref{Multiseparatedoperator})
\eqref{MixedVI1} with ${\bf w}^*$ replaced by ${\bf w}_1^*$ and ${\bf w}^*_2$ respectively, we obtain that
$
f_1({\bf x}_{1,1}^*)-f_1({\bf x}_{1,2}^*)+\langle {\bf x}_{1,1}^*-{\bf x}_{1,2}^*,-{\bf A}_1^{T}{\bf z}_2^* \rangle\geq 0$
and $f_1({\bf x}_{1,2}^*)-f_1({\bf x}_{1,1}^*)+\langle {\bf x}_{1,2}^*-{\bf x}_{1,1}^*,-{\bf A}_1^{T}{\bf z}_1^* \rangle\geq 0$.
Taking the sum of the above two inequalities gives
\begin{equation}  \label{ConvergenceTheorem3.pf.eq5}
\langle {\bf A}_1({\bf x}_{1,1}^*-{\bf x}_{1,2}^*),{\bf z}_1^*-{\bf z}_2^* \rangle \geq 0.\end{equation}
Combining  \eqref{ConvergenceTheorem3.pf.eq4} and \eqref{ConvergenceTheorem3.pf.eq5}
proves that
\begin{equation} \label{ConvergenceTheorem3.pf.eq6}
{\bf A}_1{\bf x}_{1,2}^*= {\bf A}_1{\bf x}_{1,1}^* \ \  {\rm and } \ \  {\bf z}_2^*={\bf z}_1^*.
\end{equation}

 By \eqref{ConvergenceTheorem3.pf.eq6} and  the first two equations in \eqref{ConvergenceTheorem3.pf.eq3}, we have
 that
$
\big({\bf H}_j+\beta{\bf A}_j^{T}{\bf A}_j\big)({\bf x}_{j,1}^*-{\bf x}_{j,2}^*)
={\bf 0}$ for all $1\le j\le l$.
This together with \eqref{PD-Condition1}
 implies that
\begin{equation}  \label{ConvergenceTheorem3.pf.eq7}
{\bf x}_{j,2}^*={\bf x}_{j,1}^*, \ 1\le j\le l.
\end{equation}
Combining \eqref{ConvergenceTheorem3.pf.eq6} and \eqref{ConvergenceTheorem3.pf.eq7} proves that
${\bf w}^*_2={\bf w}_1^*$. This completes the proof on uniqueness  of the limit points of the sequence
 ${\bf w}^k, k\ge 0$.
\end{proof}

\subsection{Proof of Proposition \ref{Prop-Guarantee-Assumption1}}\label{supplement}

Our proof is inspired by \cite[Theorem 2]{CMY2015}.  The first inequality in \eqref{Prop-Guarantee-Assumption1.eq2}
follows from \eqref{Prop-Guarantee-Assumption1.eq1}.
Therefore it suffices to prove
\begin{equation}\label{Prop-Guarantee-Assumption1.pf.eq1}
 \sum_{k=1}^{\infty}\|{\bf w}^k-{\bf w}^{k-1}\|_{{\bf G}}^2
\leq \frac{\|{\bf w}^0-{\bf w}^{*}\|_{{\bf G}}^2}{(1-3\alpha)(1-\alpha)}.
\end{equation}
Take ${\bf w}^*\in\mathcal{W}^*$.
 Recall that
$$\bar {\bf w}^k= {\bf w}^k+ \alpha_k ({\bf w}^k-{\bf w}^{k-1}),\  k\ge 0.$$
This together with  \eqref{e3.7},  \eqref{e3.8} and \eqref{e3.9} implies that
\begin{eqnarray}\label{e3.16}
\qquad &   &     \|{\bf w}^{k+1}-{\bf w}^*\|_{\bf G}^2- (1+\alpha_k) \|{\bf w}^{k}-{\bf w}^*\|_{\bf G}^2+
\alpha_k\|{\bf w}^{k-1}-{\bf w}^*\|_{\bf G}^2\\
& \hskip-0.08in  = &  \hskip-0.08in  \langle {\bf w}^{k+1}-{\bf w}^*, {\bf w}^{k+1}-\bar{{\bf w}}^k \rangle_{\bf G}-\|{\bf w}^{k+1}-{\bf w}^k\|_{\bf G}^2+ \alpha_k\|{\bf w}^{k}-{\bf w}^{k-1}\|_{\bf G}^2\nonumber\\
& \hskip-0.08in  & \hskip-0.08in
+2\alpha_k\langle {\bf w}^{k+1}-{\bf w}^k, {\bf w}^{k}-{\bf w}^{k-1} \rangle_{\bf G}\nonumber\\
& \hskip-0.08in  \leq & \hskip-0.08in    -\|{\bf w}^{k+1}-{\bf w}^k\|_{\bf G}^2+ \alpha_k\|{\bf w}^{k}-{\bf w}^{k-1}\|_{\bf G}^2
+2\alpha_k\langle {\bf w}^{k+1}-{\bf w}^k, {\bf w}^{k}-{\bf w}^{k-1} \rangle_{\bf G}\nonumber\\
\quad & \hskip-0.08in  \leq & \hskip-0.08in
-(1-\alpha_k)\|{\bf w}^{k+1}-{\bf w}^k\|_{{\bf G}}^2+2\alpha_k\|{\bf w}^{k}-{\bf w}^{k-1}\|_{{\bf G}}^2. \nonumber
\end{eqnarray}
Set
$$
\mu_k:=\|{\bf w}^{k}-{\bf w}^*\|_{\bf G}^2-\alpha_k\|{\bf w}^{k-1}-{\bf w}^*\|_{\bf G}^2+2\alpha_k\|{\bf w}^k-{\bf w}^{k-1}\|_{{\bf G}}^2, \ k\ge 0.
$$
Then it follows from \eqref{e3.16} and the assumption $0\le \alpha_k\le \alpha_{k+1}\le \alpha<1/3$ that
\begin{eqnarray}\label{decreaingSequnce.mu}
\mu_{k+1}-\mu_k & \hskip-0.08in \le  & \hskip-0.08in
 -(1-\alpha_k-2 \alpha_{k+1})\|{\bf w}^{k+1}-{\bf w}^k\|_{{\bf G}}^2  +(\alpha_k-\alpha_{k+1})
 \|{\bf w}^{k}-{\bf w}^*\|_{{\bf G}}^2 \nonumber\\
& \hskip-0.08in \leq &  \hskip-0.08in -(1-3\alpha)\|{\bf w}^{k+1}-{\bf w}^{k}\|_{{\bf G}}^2\le 0,\  k\ge 0,
\end{eqnarray}
 which  implies that $\mu_k, k\ge 0$ is an  nonincreasing sequence bounded  above by
$$\mu_0\le (1-\alpha_0)\|{\bf w}^0-{\bf w}^*\|_{\bf G}^2\leq \|{\bf w}^0-{\bf w}^*\|_{\bf G}^2.$$
Therefore
\begin{equation*} \label{decreaingSequnce.mu+1}
\|{\bf w}^{k}-{\bf w}^*\|_{\bf G}^2-\alpha\|{\bf w}^{k-1}-{\bf w}^*\|_{\bf G}^2\le
\mu_k\le \|{\bf w}^0-{\bf w}^*\|_{\bf G}^2,\  k\ge 0.
\end{equation*}
Applying the above upper estimate repeatedly gives
\begin{eqnarray} \label{decreaingSequnce.mu+2}
\|{\bf w}^{k}-{\bf w}^*\|_{\bf G}^2 & \hskip-0.08in \leq & \hskip-0.08in  \|{\bf w}^0-{\bf w}^*\|_{\bf G}^2+\alpha\|{\bf w}^{k-1}-{\bf w}^*\|_{\bf G}^2\le \ldots \nonumber\\
& \hskip-0.08in\leq & \hskip-0.08in   \sum_{j=0}^{k-1}\alpha^j \|{\bf w}^{0}-{\bf w}^*\|_{\bf G}^2 +\alpha^k\|{\bf w}^{0}-{\bf w}^*\|_{\bf G}^2\le \frac{\|{\bf w}^0-{\bf w}^*\|_{\bf G}^2}{1-\alpha}.
\end{eqnarray}

By  \eqref{decreaingSequnce.mu}, we have
$$
(1-3\alpha)\|{\bf w}^{k+1}-{\bf w}^{k}\|_{{\bf G}}^2\leq \mu_{k}-\mu_{k+1},\ k\ge 0.
$$
Taking sum over $k$ on above inequality  and applying \eqref{decreaingSequnce.mu+2}, we obtain
\begin{eqnarray*}
 & & (1-3\alpha)\sum_{j=0}^k\|{\bf w}^{j+1}-{\bf w}^{j}\|_{{\bf G}}^2
  \leq  \mu_0-\mu_{k+1} \nonumber\\
& \hskip-0.08in  \leq  &  \hskip-0.08in \|{\bf w}^{0}-{\bf w}^*\|_{\bf G}^2 + \alpha_{k+1} \|{\bf w}^{k}-{\bf w}^*\|_{\bf G}^2%\\
\le \frac{1}{1-\alpha}\|{\bf w}^{0}-{\bf w}^*\|_{\bf G}^2.
\end{eqnarray*}
This proves \eqref{Prop-Guarantee-Assumption1.pf.eq1} and completes the proof.

%%%%%%%%%%%%%%%%%%%%%%%%%%%%%%%%%%%%%%%%%%%%%%%%%%%%%%%%%%%%%%%%%%%%%%%%%%%%%%%%%%%%%%%%%%%%%%%%
\section{Inertial Proximal ADMM and Compressive Affine Phase Retrieval}
\label{s4.section}

The problem to reconstruct of a (sparse) real signal  ${\bf x}$ from its affine quadratic measurements
\eqref{APRModel1} is highly nonlinear.   Based on the Prox-IADMM  for separable multi-block convex optimizations,
we propose
 a compressive affine phase retrieval via lifting (CAPReaL)  approach \eqref{CompressedAffinePhaseLift-Biconvex-Equivalent} for the affine phase  retrieval problem in  Section  \ref{s6.1}.
The  affine quadratic measurements
\eqref{APRModel1} could be corrupted in practice. In Section \ref{s6.2},  we propose
compressive affine phase retrieval via lifting with $\ell^p$-constraint ($p$-CAPReaL)
to  reconstruct a real signal approximately from its corrupted affine quadratic measurements.
The demonstration of our proposed algorithms
to recover sparse signals  stably from their
(un)corrupted affine quadratic measurements
will be  presented in Section  \ref{Numerical.Section}.

\subsection{Compressive affine phase retrieval via lifting}\label{s6.1}
Define  the soft thresholding operator
$S({\bf x}, r), r\ge 0$, for ${\bf x}=(x_1, \ldots, x_n)^T$ by
\begin{align}\label{SoftThresholding}
S({\bf x}, r)=(\text{sgn}(x_1)(|x_1|-r)_+, \cdots, \text{sgn}(x_n)(|x_n|-r)_+)^T,
\end{align}
 and denote  the projection onto the positive semi-definite cone  ${\bf S}_+^n$  by   $\mathbb{P}_{\succeq}: S^n \rightarrow S^n_+$.
 For the case that ${\bf X}$ has the eigenvalue decomposition ${\bf X}={\bf U}\Lambda{\bf U}^{T}$,
 then   $\mathbb{P}_{\succeq}({\bf X})={\bf U}\Lambda_+{\bf U}^{T}$, where
 $U$ is an orthogonal matrix, $\Lambda={\rm diag}(\lambda_1, \ldots, \lambda_n)$ is a diagonal matrix and
 $\Lambda_+={\rm diag}((\lambda_1)_=, \ldots, (\lambda_n)_+)$.
Observe that the  CAPReaL model  \eqref{CompressedAffinePhaseLift-Equivalent} is
a linearly constrained  separable $3$-block convex optimization problem %with linear constraint,
\eqref{Multiseparatedoperator} with    ${\bf x}_i$ and ${\bf A}_i, i=1, 2, 3$ given by $\textbf{x}_1=\textbf{x},\ \textbf{x}_2=\textbf{X},\ \textbf{x}_3=\textbf{Y}$,
and
$$
{\bf A}_1=
\begin{bmatrix}
{\bf B}\\
{\bf O}
\end{bmatrix}, \
{\bf A}_2=
\begin{bmatrix}
\mathcal{A}/2\\
\mathcal{I}_n
\end{bmatrix}, \
{\bf A}_3=
\begin{bmatrix}
\mathcal{A}/2\\
-\mathcal{I}_n
\end{bmatrix}.
$$
Therefore taking
\begin{equation}\label{carpl.heq}
{\bf H}_1=\frac{\beta}{\eta_1}{\bf I}_n-\beta{\bf B}^{T}{\bf B} \ {\rm and}\
{\bf H}_i=\frac{\beta}{\eta_{i}}\mathcal{I}_n
-\frac{\beta}{4}(\mathcal{A}^{*}\mathcal{A}+4\mathcal{I}_n)
\ {\rm for}\ i=2, 3,
% \ \ {\rm and}\  \
%{\bf H}_3=\frac{\beta}{\eta_{3}}\mathcal{I}_n
%-\frac{\beta}{4}(\mathcal{A}^{*}\mathcal{A}+4\mathcal{I}_n)
\end{equation}
with
\begin{equation}\label{eta123.requirement}
0<\eta_1<(\|{\bf B}^{T}{\bf B}\|_{2\rightarrow 2})^{-1}
\ \ {\rm and}\  \
0<\eta_2,\eta_3<2(\|\mathcal{A}^{*}\mathcal{A}+4\mathcal{I}_n\|_{F\rightarrow F})^{-1},
\end{equation}
we obtain  the following concrete form of  the corresponding Prox-IADMM algorithm,
 %to deal with the CAPReaL model \eqref{CompressedAffinePhaseLift-Equivalent},
 where
$\mathcal{A}^*:\mathbb{R}^m\ni {\bf c}=(c_1, \ldots, c_m)^T\mapsto \sum_{j=1}^m c_j{\bf a}_j^{T}{\bf a}_j \in \mathbb{R}^{n\times n}$ is the adjoint operator of $\mathcal{A}$,  and $\mathcal{I}_n^*:\mathbb{R}^{n\times n} \rightarrow\mathbb{R}^{n\times n}$ is the adjoint operator of $\mathcal{I}_n$.

\medskip
\noindent\rule[0.25\baselineskip]{\textwidth}{1pt}
%\begin{algorithm}
\label{algorithm4}
\centerline {\bf  CAPReaL Algorithm}\\
{\bf Input}:\ Given $({\bf x}^0,{\bf X}^0,{\bf Y}^0;{\bf z}^0,{\bf Z}^0)$, $\tau$, $\lambda$, $\beta>0$,
 parameters $\eta_i, 1\le i\le 3$, satisfying \eqref{eta123.requirement}, and step sizes $\alpha_k, k\ge 0$. \\
{\bf Initials}:\  Let $({\bf x}^{-1},{\bf X}^{-1},{\bf Y}^{-1};{\bf z}^{-1},{\bf Z}^{-1})
=({\bf x}^0,{\bf X}^0,{\bf Y}^0;{\bf z}^0,{\bf Z}^0)$ and $k=0$.\\
{\bf Circulate} Step 1--Step 6 until ``some stopping criterion is satisfied":  %\\

 ~{\bf Step 1} Iterate as
\begin{eqnarray*}
(\bar{{\bf x}}^k,\bar{{\bf X}}^k,\bar{{\bf Y}}^k;\bar{{\bf z}}^k,\bar{{\bf Z}}^k)
&  \hskip-0.08in  = &   \hskip-0.08in ({\bf X}^k,{\bf Y}^k,{\bf x}^k;{\bf z}^k, {\bf Z}^k) +\alpha_k({\bf x}^k-{\bf x}^{k-1},{\bf X}^k-{\bf X}^{k-1}, \nonumber\\
& &  \quad {\bf Y}^k-{\bf Y}^{k-1};{\bf z}^k-{\bf z}^{k-1},
{\bf Z}^k-{\bf Z}^{k-1}).
\end{eqnarray*}

{\bf Step 2}
Compute ${\bf x}^{k+1}$ by
\begin{align*}
{\bf x}^{k+1}=S\Bigg(\bar{{\bf x}}^k
-\eta_1{\bf B}^{T}\bigg(\frac{1}{2}\mathcal{A}(\bar{{\bf X}}^{k})+\frac{1}{2}\mathcal{A}(\bar{{\bf Y}}^{k})
+{\bf B}\bar{{\bf x}}^k-{\bf c}-\frac{\bar{{\bf z}}^{k}}{\beta}\bigg),
\frac{\lambda\eta_1}{\beta}\Bigg).
\end{align*}

{\bf Step 3} Update multiplier ${\bf z}^{k+1},{\bf Z}^{k+1}$ via
\begin{align*}
{\bf z}^{k+1}&=\bar{{\bf z}}^k-\beta\bigg(\frac{1}{2}\mathcal{A}(\bar{{\bf X}}^{k})+\frac{1}{2}\mathcal{A}(\bar{{\bf Y}}^{k})
+{\bf B}{\bf x}^{k+1}-{\bf c}\bigg),\\
{\bf Z}^{k+1}&=\bar{{\bf Z}}^{k}-\beta(\bar{{\bf X}}^{k}-\bar{{\bf Y}}^k).
\end{align*}

{\bf Step 4} Compute ${\bf X}^{k+1}$ by
\begin{eqnarray*}
{\bf X}^{k+1} &   \hskip-0.08in   = &   \hskip-0.08in  \mathcal{P}_{\succeq}\bigg(\bar{{\bf X}}^{k}-\frac{\eta_2}{\beta}{\bf I}_n
-\frac{\eta_2}{2}\mathcal{A}^*\bigg(\frac{1}{2}\mathcal{A}(\bar{{\bf X}}^k)+\frac{1}{2}\mathcal{A}(\bar{{\bf Y}}^k)
+{\bf B}{\bf x}^{k+1}-{\bf c}-\frac{\bar{{\bf z}}^{k+1} }{\beta}\bigg)\nonumber\\
& & \qquad \quad-\eta_2\bigg(\bar{{\bf X}}^k-\bar{{\bf Y}}^k-\frac{{\bf Z}^{k+1} }{\beta}\bigg)
\bigg).
\end{eqnarray*}

{\bf Step 5} Compute ${\bf Y}^{k+1}$ by
\begin{eqnarray*}
{\bf Y}^{k+1}&  \hskip-0.08in = &  \hskip-0.08in  S\Bigg(\bar{{\bf Y}}^{k}
-\frac{\eta_3}{2}\mathcal{A}^*\bigg(\frac{1}{2}\mathcal{A}(\bar{{\bf X}}^{k})+\frac{1}{2}\mathcal{A}(\bar{{\bf Y}}^{k})
+{\bf B}{\bf x}^{k+1}-{\bf c}-\frac{1}{\beta}{\bf z}^{k+1}\bigg)\nonumber\\
& & \qquad +\eta_3\bigg(\bar{{\bf X}}^{k}-\bar{{\bf Y}}^k-\frac{{\bf Z}^{k+1}}{\beta}\bigg),\frac{\tau\eta_2}{\beta}\Bigg).
\end{eqnarray*}

{\bf Step 6} Update $k$ to $k+1$.\\
{\bf Output}:  $(\hat{{\bf x}},\hat{{\bf X}},\hat{{\bf Y}})$.\\
\noindent\rule[0.25\baselineskip]{\textwidth}{1pt}

By \eqref{carpl.heq}, \eqref{eta123.requirement} and Theorem \ref{IADMMConvergence3.mainthm}, the above CAPReaL algorithm converges.
 However, the solution $(\hat{{\bf X}},\hat{{\bf Y}},\hat{{\bf x}})$ of the above algorithm may not satisfy
the  constrained condition $\hat{{\bf Y}}=\hat{{\bf x}} \hat{{\bf x}}^{H}$. In order to compensate for these relaxations, we take  two additional steps:
\begin{itemize}
\item[{(a)}]  In addition to  the stopping criterion \eqref{StoppingCreterion.e2} to the Prox-IADMM,
%we select the following
% additional stopping criteria,
%\begin{equation}\label{e5.5+0}
%\frac{\|{\bf Y}^k-{\bf x}^k({\bf x}^k)^{T}\|_{F}}{\|{\bf x}^k({\bf x}^k)^{T}\|_{F}}
%\leq\varepsilon
%\end{equation}
we select the additional stopping criteria,
\begin{equation}\label{stopping.compensation}
\frac{\|{\bf Y}^k-{\bf x}^k({\bf x}^k)^{T}\|_{F}}{\|{\bf x}^k({\bf x}^k)^{T}\|_{F}}\leq\tilde{\varepsilon}\end{equation}
in the implementation of the CAPReaL algorithm.

\item [{(b)}] Add the following steps after the implementation of the CAPReaL algorithm.
% the  getting $(\hat{{\bf X}},\hat{{\bf Y}},\hat{{\bf x}})$ from Algorithm 4,5 and 6.

\begin{itemize}
\item[{(b1)}] Find the best rank-one approximation $\hat{{\bf X}}_{\text{rank}(1)}^r= \sigma_1 {\bf u}_1{\bf u}_1^{H}$ of $\hat{{\bf X}}$, and take $\tilde{{\bf x}}=\alpha \sqrt{\sigma_1}{\bf u}_1$,
%    \begin{equation}\label{e5.5+1}
%    \tilde{{\bf x}}=\alpha \sqrt{\sigma_1}{\bf u}_1,
%    \end{equation}
     where  $\sigma_1$
     is the maximal singular value of the matrix $\hat{{\bf X}}$ and  the sign $\alpha=\pm 1$ is chosen so that
     $\langle \tilde{{\bf x}}, \hat{{\bf x}}\rangle\ge 0$. %  $\sqrt{\sigma_1}$ is real square root of  .
\item[{(b2)}]
Find the best $s^2$-sparse approximation $\hat{{\bf Y}}_{\max(s^2)}$ of $\hat{{\bf Y}}$ in the norm $\|\cdot\|_1$,
and
 compute the full rank decomposition of $\hat{{\bf Y}}_{\max(s^2)}={\bf U}{\bf V}^{T}$, and then take $\tilde{{\bf y}}=\tilde \alpha {\tilde {\bf u}}_1$,
% \begin{equation}\label{e5.5+2}
% \tilde{{\bf y}}=\tilde \alpha {\tilde {\bf u}}_1,
% \end{equation}
 where ${\bf u}_1$ is the first column of ${\bf U}$ and
the sign $\tilde \alpha=\pm 1$ is chosen so that
     $\langle \tilde{{\bf y}}, \hat{{\bf x}}\rangle\ge 0$.

\item[{(b3)}]Compute ${\bf x}^*=(\hat{{\bf x}}+\tilde{{\bf x}}+\tilde{{{\bf y}}})/3$.
%\begin{equation} \label{e5.5+3}
%{\bf x}^*=\frac{\hat{{\bf x}}+\tilde{{\bf x}}+\tilde{{{\bf y}}}}{3}.
%\end{equation}
\end{itemize}
\end{itemize}

\subsection{Compressive affine phase retrieval via lifting with penalty}\label{s6.2}
In this section, we consider compressive affine phase retrieval problem with corrupted measurements,
\begin{eqnarray}\label{APR-Noise-Model1}
\bar{{\bf b}}
%\mathrm{{\bf M}_{{\bf A},{\bf b}}^2}
&  \hskip-0.08in = & \hskip-0.08in  \big(|\langle {\bf a}_1,{\bf x} \rangle+b_1|^2,\ldots,|\langle {\bf a}_m,{\bf x} \rangle+b_m|^2\big)^T+{\bf e}\nonumber\\
&  \hskip-0.08in = & \hskip-0.08in  \mathcal{A}({\bf x}{\bf x}^{H})
+{\bf B}{\bf x}+|{\bf b}|^2 +{\bf e},
\end{eqnarray}
where ${\bf e}=({\bf e}_1, \ldots, {\bf e}_m)^T\in\mathbb{R}^m$ is the  noise.
Similar to the bi-convex relaxation model \eqref{CompressedAffinePhaseLift-Biconvex-Equivalent}, we propose the following approach:
\begin{subequations} \label{CompressedAffinePhaseLift-Nosie-Biconvex}
\begin{equation} \label{CompressedAffinePhaseLift-Nosie-Biconvex.a}
\min_{{\bf X}\succeq {\bf O},{\bf Y}\in\mathbb{R}^{n\times n},{\bf x}\in\mathbb{R}^n,{\bf y}\in\mathbb{R}^{m}}~\text{tr}({\bf X})+\tau\|{\bf Y}\|_1+\lambda\|{\bf x}\|_1
+\rho  %h_p({\bf y}) %
\|{\bf y}\|_p^p
\end{equation}
\vskip-0.18in
\begin{equation} \label{CompressedAffinePhaseLift-Nosie-Biconvex.b}
\text{subject \ to}\ \ ~\frac{1}{2}\mathcal{A}({\bf X})+\frac{1}{2}\mathcal{A}({\bf Y})
+{\bf B}{\bf x}-{\bf y}={\bf c},\
%\end{equation}
%\begin{equation} \label{CompressedAffinePhaseLift-Nosie-Biconvex.d}
{\bf X}-{\bf Y}={\bf O}, \ \ {\rm and} \
\end{equation}
\vskip-0.18in
\begin{equation} \label{CompressedAffinePhaseLift-Nosie-Biconvex.c}
{\bf Y}={\bf x}{\bf x}^{H},
\end{equation}
\end{subequations}
where ${\bf c}=\bar{{\bf b}}-|{\bf b}|^2$, $\tau, \lambda,  \rho>0$ are balance parameters, and
%$1\le p<\infty$.
$$
h_p({\bf y})=
\left\{\begin{array}{ll} \|{\bf y}\|_p^p  &  {\rm if}\   0<p<\infty\\
\|{\bf y}\|_{\infty} & {\rm if} \ p=\infty.
\end{array}\right.
$$
We call the above approach as the  Compressive Affine Phase Retrieval via Lifting with $p$-Constraint, and use the abbreviation  $p$-CAPReaL.
Holding the constraint %${\bf X}={\bf x}{\bf x}^{T}$ and ${\bf Y}={\bf x}{\bf x}^{T}$
in \eqref{CompressedAffinePhaseLift-Nosie-Biconvex.c} about  ${\bf Y}$, the approach in \eqref{CompressedAffinePhaseLift-Nosie-Biconvex}
 becomes
a   separable {\bf $4$-block} convex optimization problem %with linear constraint,
\eqref{Multiseparatedoperator} with linearly constraint, where ${\bf x}_1={\bf y},\ {\bf x}_2={\bf X},\ {\bf x}_3={\bf Y},\ {\bf x}_4={\bf x},$
% $$
% {\bf x}_1={\bf y},\ {\bf x}_2={\bf X},\ {\bf x}_3={\bf Y},\ {\bf x}_4={\bf x},
% $$
 and
$$
{\bf A}_1=
\begin{bmatrix}
-{\bf I}_m\\
{\bf O}
\end{bmatrix}, \
{\bf A}_2=
\begin{bmatrix}
\mathcal{A}/2\\
\mathcal{I}_n
\end{bmatrix}, \
{\bf A}_3=
\begin{bmatrix}
\mathcal{A}/2\\
-\mathcal{I}_n
\end{bmatrix}
 \ \ {\rm and}\  \
{\bf A}_4=
\begin{bmatrix}
{\bf B}\\
{\bf O}
\end{bmatrix}.
$$
Thus we can  use the  Prox-IADMM to solve the  above separable $4$-block convex optimization problem %with linear
 with the regularization matrices
\begin{equation}
 {\bf H}_1=
\frac{\beta}{\eta_1}{\bf I}_m-\beta {\bf I}_m,~
 {\bf H}_4=\frac{\beta}{\eta_4}{\bf I}_n-\beta{\bf B}^{T}{\bf B},
 %\ {\rm and} \ %{\bf H}_2=\frac{\beta}{\eta_2}\mathcal{I}_n-\frac{\beta}{4}\mathcal{A}^{*}\mathcal{A},~
\ {\rm and}\ {\bf H}_i=\frac{\beta}{\eta_i}\mathcal{I}_n-\frac{\beta}{4}\mathcal{A}^{*}\mathcal{A}
\ {\rm for}\ i=2, 3,
\end{equation}
where  $\eta_i>0, 1\le i\le 4$, satisfy
\begin{equation}\label{eta123.eq0}
0<\eta_1< 1,\
0<\eta_2,\eta_3<\frac{4}{3\|\mathcal{A}^{*}\mathcal{A}+4\mathcal{I}_n\|_{F\rightarrow F}}\
\ \ {\rm and}\ \
0<\eta_4<\frac{1}{3\|{\bf B}^{T}{\bf B}\|_{2\rightarrow 2}}.
\end{equation}
For the above selection of regularization matrices, the Prox-IADMM for the special cases that $p=1, 2, \infty$
 has the following concise formulation,
where  $S^*$ is  the proximal operator of  $\ell_{\infty}$ norm \cite{C2016, PB2013,ZN2019},
$$S^*({\bf b}, \lambda)=\text{Prox}_{\lambda \|\cdot\|_{\infty}}({\bf b}):=\arg\min_{{\bf x}}\frac{1}{2}\|{\bf x}-{\bf b}\|_2^2+\lambda\|{\bf x}\|_{\infty}.$$
%Here, we only list the algorithm for $2$-CAPReaL. And for other values of $p$, readers can find them in the supplementary material \cite{LCS2019}.
For $p=1,2,\infty $, the $p$-CAPReaL scheme can be formulated as follows.

\medskip
\noindent\rule[0.25\baselineskip]{\textwidth}{1pt}
\label{carpl-ladc.def}
\centerline{{\bf $p$-CAPReaL Algorithm}}\\

\noindent {\bf Input}:\  $({\bf y}^0,{\bf X}^0,{\bf Y}^0,{\bf x}^0;{\bf z}^0,{\bf Z}^0)$, $\tau>0$, $\lambda>0$, $\rho>0$, $\beta>0$,
nonnegative step sizes $\alpha_k, k\ge 0$, and parameters $\eta_1, \eta_2, \eta_3, \eta_4$ in \eqref{eta123.eq0}.
\\
 \noindent
 {\bf Initials}: \ Set $({\bf y}^{-1},{\bf X}^{-1},{\bf Y}^{-1},{\bf x}^{-1};{\bf z}^{-1})
=({\bf y}^0,{\bf X}^0,{\bf Y}^0,{\bf x}^0;{\bf z}^0,{\bf Z}^0)$ and $k=0$.\\
{\bf Circulate} Step 1-Step 8 until ``some stopping criterion is satisfied":

{\bf Step 1} Iterate as
\begin{eqnarray*}
 \hskip-0.12in&  \hskip-0.08in  & \hskip-0.08in (\bar{{\bf y}}^k,\bar{{\bf X}}^k,\bar{{\bf Y}}^k,\bar{{\bf x}}^k;\bar{{\bf z}}^k,\bar{{\bf Z}}^k)
 = ({\bf y}^k,{\bf X}^k,{\bf Y}^k,{\bf x}^k;{\bf z}^k,{\bf Z}^k)+\alpha_k({\bf y}^k-{\bf y}^{k-1}, \\
 \hskip-0.12in &  \hskip-0.08in &  \qquad {\bf X}^k-{\bf X}^{k-1}, {\bf Y}^k-{\bf Y}^{k-1},
{\bf x}^k-{\bf x}^{k-1};{\bf z}^k-{\bf z}^{k-1},{\bf Z}^k-{\bf Z}^{k-1}).
\end{eqnarray*}

{\bf Step 2}
%Compute ${\bf w}^{k+1}$ by
Compute ${\bf y}^{k+1}$ by
\begin{align}\label{e5.5}
{\bf y}^{k+1}
&=S\Big(\bar{{\bf y}}^k
+\eta_1\bigg(\frac{1}{2}\mathcal{A}(\bar{{\bf x}}^{k})+\frac{1}{2}\mathcal{A}(\bar{{\bf Y}}^{k})
+{\bf B}\bar{{\bf x}}^k-\bar{{\bf y}}^k-{\bf c}-\frac{\bar{{\bf z}}^{k}}{\beta}\bigg)
,\frac{\rho\eta_1}{\beta}\Big)
\end{align}
if $p=1$, and
\begin{align}\label{e5.6}
{\bf y}^{k+1}
&=\frac{\beta}{\beta+2\rho\eta_1}\bar{{\bf y}}^k+\frac{\beta\eta_1}{\beta+2\rho\eta_1}
\bigg(\frac{1}{2}\mathcal{A}(\bar{{\bf X}}^{k})+\frac{1}{2}\mathcal{A}(\bar{{\bf Y}}^{k})
+{\bf B}\bar{{\bf x}}^k-\bar{{\bf y}}^k-{\bf c}-\frac{\bar{{\bf z}}^{k}}{\beta}\bigg)
\end{align}
if $p=2$,
 and
\begin{align}\label{e5.7}
{\bf y}^{k+1}
&=S^*\Bigg(\bar{{\bf y}}^k
+\eta_1\bigg(\frac{1}{2}\mathcal{A}(\bar{{\bf X}}^{k})+\frac{1}{2}\mathcal{A}(\bar{{\bf Y}}^{k})
+{\bf B}\bar{{\bf x}}^k-\bar{{\bf y}}^k-{\bf c}-\frac{\bar{{\bf z}}^{k}}{\beta}\bigg)
,\frac{\rho\eta_1}{\beta}\Bigg)
\end{align}
if $p=\infty$.

 ~{\bf Step 3} Update multipliers ${\bf z}^{k+1}$ and ${\bf Z}^{k+1}$ via
\begin{align}\label{e5.2}
{\bf z}^{k+1}&=\bar{{\bf z}}^k-\beta\bigg(\frac{1}{2}\mathcal{A}(\bar{{\bf X}}^{k})
+\frac{1}{2}\mathcal{A}(\bar{{\bf Y}}^{k})
+{\bf B}\bar{{\bf x}}^k-{\bf y}^{k+1}-{\bf c}\bigg),\nonumber\\
{\bf Z}^{k+1}&=\bar{{\bf Z}}^k-\beta\big(\bar{{\bf X}}^{k}-\bar{{\bf Y}}^k\big).
\end{align}
 ~{\bf Step 4} Compute ${\bf X}^{k+1}$ by
\begin{eqnarray}\label{e5.1}
{\bf X}^{k+1}
&\hskip-0.08in = & \hskip-0.08in\mathcal{P}_{\succeq}\bigg(\bar{{\bf X}}^{k}-\frac{\eta_2}{\beta}{\bf I}_n
-\frac{\eta_2}{2}\mathcal{A}^*\bigg(\frac{1}{2}\mathcal{A}(\bar{{\bf X}}^k)+\frac{1}{2}\mathcal{A}(\bar{{\bf Y}}^{k})
\nonumber\\
&  &  \qquad
+{\bf B}\bar{{\bf x}}^k-{\bf y}^{k+1}-{\bf c}-\frac{{\bf z}^{k+1}}{\beta}\bigg) -\eta_2\bigg(\bar{{\bf X}}^k-\bar{{\bf Y}}^k-\frac{{\bf Z}^{k+1}}{\beta}\bigg)\bigg).
\end{eqnarray}\\
 ~{\bf Step 5} Compute ${\bf Y}^{k+1}$ by
\begin{align}\label{e5.3}
{\bf Y}^{k+1}&=S\Big(\big(\bar{{\bf Y}}^{k}
-\frac{\eta_3}{2}\mathcal{A}^*\big(\frac{1}{2}\mathcal{A}(\bar{{\bf X}}^{k})+\frac{1}{2}\mathcal{A}(\bar{{\bf Y}}^{k})
+{\bf B}\bar{{\bf x}}^k-{\bf y}^{k+1}-{\bf c}-\frac{{\bf z}^{k+1}}{\beta}\big)\nonumber\\
&\hspace*{24pt}+\eta_3\big(\bar{{\bf X}}^{k}-\bar{{\bf Y}}^k-\frac{{\bf Z}^{k+1}}{\beta}\big),
\frac{\tau\eta_3}{\beta}\Big).
\end{align}
 ~{\bf Step 6} Compute ${\bf x}^{k+1}$ by
\begin{align}\label{e5.4}
{\bf x}^{k+1}
&=S\Big(\bar{{\bf x}}^k
-\eta_4{\bf B}^{T}\big(\frac{1}{2}\mathcal{A}(\bar{{\bf X}}^{k})+\frac{1}{2}\mathcal{A}(\bar{{\bf Y}}^{k})
+{\bf B}\bar{{\bf x}}^k-{\bf y}^{k+1}-{\bf c}-\frac{{\bf z}^{k+1}}{\beta}\big),
\frac{\eta_4\lambda}{\beta}\Big).
\end{align}
 ~{\bf Step 7} Update $k$ to $k+1$.\\
{\bf Output}:  $(\hat{{\bf y}},\hat{{\bf X}},\hat{{\bf Y}},\hat{{\bf x}})$.\\
\noindent\rule[0.25\baselineskip]{\textwidth}{1pt}

By \eqref{carpl.heq}, \eqref{eta123.requirement} and Theorem \ref{IADMMConvergence3.mainthm}, the $p$-CAPReaL algorithm converges for $1\le p\le \infty$.
However, the solution $\hat{{\bf Y}},\hat{{\bf x}},\hat{{\bf y}}$ of the above algorithm may not satisfy
the  constrained condition $\hat{{\bf Y}}=\hat{{\bf x}} \hat{{\bf x}}^{T}$. In order to compensate for that relaxation,
we take % two
same additional steps as those in Subsection \ref{s6.1}.

\section{Numerical Simulations}\label{Numerical.Section}

In this section, we demonstrate performance of the proposed  ($p$-)CAPReaL  algorithm
to recover $s$-sparse real vectors ${\bf x}_o\in \mathbb{R}^n$ from
either the noiseless
 quadratic measurement ${\bf c}=| {\bf A} {\bf x}_o+{\bf b}|^2$
%\begin{equation}\label{simulation.data1} {\bf c}=| {\bf A} {\bf x}_o+{\bf b}|^2\end{equation}
or the noisy
 quadratic measurement ${\bf c}=| {\bf A} {\bf x}_o+{\bf b}|^2+{\bf e}$,
and compare it with  the conventional
 phase retrieval algorithms \cite{CLM2016, JH2017, LV2013, OYDS2012}.
%      to reconstruct sparse
%real signals from their (noisy) affine quadratic measurements.
 In our simulations, the measurement matrix
${\bf A}\in \mathbb{R}^{m\times n}$ is the real standard Gaussian matrix of size $m\times n$,
the true  $s$-sparse signal ${\bf x}_o\in\mathbb{R}^n$  has each nonzero components randomly i.i.d. drawn according to
the continuous uniform distribution ${\mathcal U}(-1, 1)$ on $[-1, 1]$, and
 the reference vector ${\bf b}\in\mathbb{R}^{m}$ has its components $b_j=\xi_j y_j, j=1,\ldots,m,$
with $\xi_j$ and $y_j$ randomly i.i.d. drawn according to
the continuous uniform distribution ${\mathcal U}(-1, 1)$ and standard normal distribution $\mathcal{N}(0, 1)$ respectively  \cite{BLT2018,ZN2019}.
%In our simulations, the measurement matrix ${\bf A}=[{\bf a}_1,\ldots,{\bf a}_m]^{T}\in\mathbb{R}^{m\times n}$ is a
% real standard Gaussian matrix,  the reference vector ${\bf b}$ has its components generated according to the standard Gaussian distribution,
In our simulations,
%\begin{equation}\label{simulation.data2} {\bf c}=| {\bf A} {\bf x}_o+{\bf b}|^2+{\bf e},\end{equation}
  we consider  a Gaussian white noise  ${\bf e}$ with variance $\sigma^2>0$, i.e.,
${\bf e}\sim \sigma \mathcal{N}(0, {\bf I}_m)$, or
a Cauchy  noise
%\begin{equation}\label{Cauchynoise}
${\bf e}\sim  {\mathcal C}(0, \gamma)$
%\end{equation}
with its probability density function %of the   Cauchy distribution ${\mathcal C}(0, \gamma)$
 given by  $(\pi \gamma (1+ |x/\gamma|)^2)^{-1}$, where $\gamma$  is the scale parameter to specify the  noise half-width at half-maximum  \cite{SPD2013, user2014, WPYYL2017}, and we also test for a uniformly distributed noise
${\bf e}\sim  {\mathcal U}(-\delta, \delta)$ (i.e., the uniform distribution on the interval $(-\delta, \delta)$), where $\delta>0$ is a noise range parameter
\cite{BLT2018,ZN2019}.
All experiments
were performed under Windows Vista Premium and MATLAB v7.8 (R2016b)
running on a Huawei laptop with an Intel(R) Core(TM)i5-8250U CPU at 1.8 GHz and 8195MB RAM of memory.

%where  we consider the following three types of noises:
%\begin{itemize}
%\item[{(i)}] ${\bf e}$ is a Gaussian white noise with variance $\sigma^2>0$,
%\begin{equation}\label{gaussiannoise}
%{\bf e}\sim \sigma \mathcal{N}(0, {\bf I}_m).
%\end{equation}
%
%\item[{(ii)}]
% ${\bf e}$ is a Cauchy  noise,
%\begin{equation}\label{Cauchynoise}
%{\bf e}\sim  {\mathcal C}(0, \gamma)
%\end{equation}
%where the probability density function of the   Cauchy distribution ${\mathcal C}(0, \gamma)$
%is given by  $(\pi \gamma (1+ |x/\gamma|)^2)^{-1}$ and $\gamma$  is the scale parameter to specify the  noise half-width at half-maximum  \cite{SPD2013, user2014, WPYYL2017}.
%
%\item[{(iii)}] ${\bf e}$ is a uniformly distributed noise,
%\begin{equation}\label{Uniformlynoise}
%{\bf e}\sim  {\mathcal U}(-\delta, \delta),
%\end{equation}
%where $\delta>0$ is a noise range parameter
%\cite{BLT2018,ZN2019}.
%\end{itemize}

\subsection{CAPReaL algorithm  with different selection of step sizes}\label{caprealalgorithm.simulation}

In this subsection, we  demonstrate the performance
of the CAPReaL algorithm with different selection of step sizes
to recover sparse signals  from their affine quadratic affine measurements.
Shown in Table \ref{tab:CAPRLsuccnumber-alpha} are  average success percentages of the CAPReaL algorithm for different selection of step sizes over 100 independent realizations
to recover sparse signals from their noiseless quadratic measurements of size $m$,
 where the original  sparsity signal ${\bf x}_o$  has sparsity $s=4$ and length $n=64$, and
 step sizes $\alpha_k=1/8,1/4,1/3,1/2$
are independent of $k$ for the first three simulations and $\alpha_k=1/3-(1/3)^{\lfloor k/5 \rfloor}$ and $(1/3)^{\lfloor k/5 \rfloor}, k\ge 0$ for the last two simulations. Here ${\lfloor a \rfloor}$ denotes the nearest integer less than or equal to $a$.
In the  simulation,
the recovery is regarded as successful if
%the reconstructed signal ${\bf x}^*$ in \eqref{e5.5+3} satisfies
%$$
%\frac{\|{\bf x}^*-{\bf x}_o\|_2}{\|{\bf x}_o\|_2}\leq  0.01, %\epsilon,
%$$
 $\|{\bf x}^*-{\bf x}_o\|_2/\|{\bf x}_o\|_2\leq  0.01$,
where %$\epsilon=0.01$ is the tolerance parameter,  and
${\bf x}^*$ is the
reconstructed signal via the CAPReaL algorithm.
This indicates that step sizes in the CAPReaL algorithm should be chosen appropriately and
the CAPReaL algorithm with  step size $\alpha_k=1/4$ for all $k\ge 0$ has highest success percentage
to recover sparse signals from their phaseless affine measurements.
Due to the above observation,  in the following simulations, we  {\bf always} choose  $\alpha_k=1/4, k\ge 0$, as step sizes in the CAPReaL algorithm
and also
in the $p$-CAPReaL algorithm.

\begin{table}[t] %[!htbp]
\setlength{\tabcolsep}{7pt}
	{\caption{ \rm  Success percentage  of the CAPReaL algorithm  to recover  sparse real signals  over 100 repeated trials
for different ratios $m/n$ between the number $m$ of measurements and the length  of original signals,
 and for different selections of  step sizes $\alpha_k, k\ge 0$.
}\label{tab:CAPRLsuccnumber-alpha}}
\begin{center}
		\begin{tabular}{c  |c  c c c c c c   c}\hline
\backslashbox%\diagbox
{ %Step sizes
 $\alpha_k$}{$m/n$} %$\alpha_k$ \ $\backslash$ \ $m/n$
  &0.5
&0.75  &0.875 &1     &1.25  &1.5  &1.75  &2\\ \hline
		   $1/8$&1     &11   &30 &36     &82   &98     &100   &100      \\
		   $1/4$ &2     &19    &40     &61     &95          &98  &100     &100 \\
         $1/3$  &0      &6    &38    &56     &88       &96  & 100    &100   \\
         $1/2$  &0       &0     &0     &0     &0        &0   &0     &0     \\
$1/3-3^{-{\lfloor k/5 \rfloor}}$ &1     &16    &31     &53   &86     &97   &100     &100 \\
 $3^{-{\lfloor k/5 \rfloor}}$     &1    &15    &27   &45   &90    &96  &100     &100 \\ \hline 			\specialrule{0.0em}{2.0pt}{2.0pt}
   % \hline
		\end{tabular}
	\end{center}
\end{table}

\subsection{Comparison between CAPReaL  and  Jacobian/twisted ADMM-based algorithms}
\label{capreal.simulation0}
The proposed CAPReal algorithm to  recover sparse real vectors from their affine quadratic measurements
 is based on the Prox-IADMM. % discussed in Section \ref{s3.section}.
 In our simulations, we always select step sizes $\alpha_k=1/4, k\ge 0$, in the CAPReaL algorithm, see Subsection \ref{caprealalgorithm.simulation}.
As the  Prox-IADMM  \eqref{MultiproximalADMM} with  zero step sizes
  becomes the classical ADMM  \eqref{DirectiveExtendADMM}, we may use  the corresponding CAPReal algorithm based on the classical ADMM,
 CAPReaL-Zero for abbreviation,
  to solve \eqref{CompressedAffinePhaseLift-Equivalent}.  Based on the Jacobi-Proximal ADMM
 \cite{DLPY2017}, we propose the following iterative algorithm, CAPReaL-Jacobi for abbreviation,
 to solve \eqref{CompressedAffinePhaseLift-Equivalent}, where $\eta_1,\eta_2,\eta_3$ are proximal parameters, % in  \eqref{eta123.requirement},
  each iteration  is modified from the proximal Jacobian ADMM \cite{DLPY2017},
\begin{eqnarray*}\label{CAPReaL-Jocabi-1}
{\bf x}^{k+1} & \hskip-0.08in = & \hskip-0.08in S\Big({\bf x}^k
-\eta_1{\bf B}^{T}\Big(\frac{1}{2}\mathcal{A}({\bf X}^{k})+\frac{1}{2}\mathcal{A}({\bf Y}^{k})
+{\bf B}{\bf x}^k-{\bf c}-\frac{{\bf z}^{k}}{\beta}\Big),
\frac{\lambda\eta_1}{\beta}\Big),\\
{\bf X}^{k+1} &  \hskip-0.08in = & \hskip-0.08in  \mathcal{P}_{\succeq}\Big({\bf X}^{k}-\frac{\eta_2}{\beta}{\bf I}_n
-\frac{\eta_2}{2}\mathcal{A}^*\bigg(\frac{1}{2}\mathcal{A}({\bf X}^k)+\frac{1}{2}\mathcal{A}({\bf Y}^k)
+{\bf B}{\bf x}^{k+1}-{\bf c}-\frac{{\bf z}^{k} }{\beta}\Big)\nonumber\\
& & \qquad \quad-\eta_2\Big({\bf X}^k-{\bf Y}^k-\frac{{\bf Z}^{k}}{\beta}\Big)\Big),\\
%\end{eqnarray*}
%\vskip-0.18in
%\begin{eqnarray*}\label{CAPReaL-Jocabi-3}
{\bf Y}^{k+1}&  \hskip-0.08in = &  \hskip-0.08in  S\Big({\bf Y}^{k}
-\frac{\eta_3}{2}\mathcal{A}^*\Big(\frac{1}{2}\mathcal{A}({\bf X}^{k})+\frac{1}{2}\mathcal{A}({\bf Y}^{k})
+{\bf B}{\bf x}^{k+1}-{\bf c}-\frac{1}{\beta}{\bf z}^{k+1}\Big)\nonumber\\
& & \qquad +\eta_3\big({\bf X}^{k}-{\bf Y}^k-\frac{{\bf Z}^{k+1}}{\beta}\big),\frac{\tau\eta_2}{\beta}\Big),\\
{\bf z}^{k+1}&  \hskip-0.08in = &  \hskip-0.08in{\bf z}^k-\beta\Big(\frac{1}{2}\mathcal{A}({\bf X}^{k+1})
+\frac{1}{2}\mathcal{A}({\bf Y}^{k+1})
+{\bf B}{\bf x}^{k+1}-{\bf y}^{k+1}-{\bf c}\Big),\nonumber\\
{\bf Z}^{k+1} &  \hskip-0.08in = &  \hskip-0.08in {\bf Z}^k-\beta\big({\bf X}^{k+1}-{\bf Y}^k\big),
\end{eqnarray*}
%\end{subequations}
 and  the  compensation step is the same as the one in
 Subsection \ref{s6.1} being used to design the CAPReal algorithm.
Similarly,
based on  twisted version of the proximal ADMM
\cite{WS2017}
 and following the same compensation step as  the one in the CAPReal algorithm, we propose the following iterative algorithm, CAPReaL-Twisted for abbreviation, to solve
\eqref{CompressedAffinePhaseLift-Equivalent}, where $\alpha\in(0,2)$, %${\bf u}=({\bf Y};{\bf x};{\bf z};{\bf Z})$,
$0<\eta_2<(2\|{\bf B}^{T}{\bf B}\|_{2\rightarrow 2})^{-1},
0<\eta_3<2(\|\mathcal{A}^{*}\mathcal{A}+4\mathcal{I}_n\|_{F\rightarrow F})^{-1}$
are proximal parameters, and
each iteration is essentially the proximal  twisted ADMM  \cite{WS2017},
\begin{eqnarray*}\label{CAPReaL-Twisted}
 & \hskip-0.08in   & \hskip-0.08in  \tilde{\bf X}^{k}   =
\mathcal{P}_{\succeq}\Big(\big(\beta\mathcal{A}^*\mathcal{A}/4+\beta\mathcal{I}_n\big)^{-1}
\Big(\beta\big({\bf Y}^k+\frac{{\bf Z}^{k}}{\beta}\big)-{\bf I}_n\nonumber\\
& & \qquad\qquad \quad -\frac{\beta}{2}\mathcal{A}^*\big(\frac{1}{2}\mathcal{A}({\bf Y}^k)
+{\bf B}{\bf x}^{k}-{\bf c}-\frac{{\bf z}^{k} }{\beta}\big)\Big)\Big),
\\
 & \hskip-0.08in   & \hskip-0.08in  \tilde{{\bf z}}^{k}  =  {\bf z}^k-\beta\Big(\frac{1}{2}\mathcal{A}(\tilde{{\bf X}}^{k})+\frac{1}{2}\mathcal{A}({\bf Y}^{k})+{\bf B}{\bf x}^{k}-{\bf c}\Big),\nonumber\\
& \hskip-0.08in   & \hskip-0.08in \tilde{{\bf Z}}^{k} = {\bf Z}^{k}-\beta(\tilde{{\bf X}}^{k}-{\bf Y}^k),\\
 & \hskip-0.08in   & \hskip-0.08in  \tilde{{\bf x}}^{k}  =  S\Big({\bf x}^k
-\eta_1{\bf B}^{T}\Big(\frac{1}{2}\mathcal{A}(\tilde{{\bf X}}^{k})+\frac{1}{2}\mathcal{A}({\bf Y}^{k})
+{\bf B}{\bf x}^k-{\bf c}-\frac{\tilde{{\bf z}}^{k}}{\beta}\Big),
\frac{\lambda\eta_1}{\beta}\Big),\\
 & \hskip-0.08in   & \hskip-0.08in  \tilde{{\bf Y}}^{k}= S\Big({\bf Y}^{k}
-\frac{\eta_3}{2}\mathcal{A}^*\Big(\frac{1}{2}\mathcal{A}(\tilde{{\bf X}}^{k})+\frac{1}{2}\mathcal{A}({\bf Y}^{k})
+{\bf B}{\bf x}^{k}-{\bf c}-\frac{1}{\beta}\tilde{{\bf z}}^{k}\Big)\nonumber\\
& &\qquad
+\eta_3\big(\tilde{{\bf X}}^{k}-{\bf Y}^k-\frac{\tilde{{\bf Z}}^{k}}{\beta}\big),\frac{\tau\eta_2}{\beta}\Big),\\
& \hskip-0.08in  & \hskip-0.08in  {\bf X}^{k+1}   =  \tilde{{\bf X}}^{k}, \\
& \hskip-0.08in  & \hskip-0.08in %{\rm and} \
({\bf Y}^{k+1};{\bf x}^{k+1};{\bf z}^{k+1};{\bf Z}^{k+1})= (1-\alpha) ({\bf Y}^{k};{\bf x}^{k};{\bf z}^{k};{\bf Z}^{k})
+ \alpha (\widetilde {\bf Y}^{k}; \widetilde{\bf x}^{k}; \widetilde{\bf z}^{k};\widetilde {\bf Z}^{k}).
%{\bf u}^{k+1} =  {\bf u}^{k}-\alpha({\bf u}^{k}-\tilde{{\bf u}}^{k}).
%{\bf Y}^{k+1}={\bf Y}^{k}-\alpha({\bf Y}^{k}-\tilde{{\bf Y}}^{k}),
%{\bf x}^{k+1}={\bf x}^{k}-\alpha({\bf x}^{k}-\tilde{{\bf x}}^{k}),\nonumber\\
%{\bf z}^{k+1}&&={\bf z}^{k}-\alpha({\bf z}^{k}-\tilde{{\bf z}}^{k}),
%{\bf Z}^{k+1}={\bf Z}^{k}-\alpha({\bf Z}^{k}-\tilde{{\bf Z}}^{k}),
\end{eqnarray*}
In this subsection, we  present some numerical results to compare the performance
of CAPReaL, CAPReaL-Zero, CAPReaL-Jacobi and CAPReaL-Twisted  algorithms to
recover $s$-sparse real vectors ${\bf x}_o\in \mathbb{R}^n$ from
their
 quadratic measurement ${\bf c}=| {\bf A} {\bf x}_o+{\bf b}|^2$.

%We display the following items.

\begin{table}[t] %[!htbp]
\setlength{\tabcolsep}{9pt}
	\caption{\rm Average iteration number and time consumption over 100 trials to
implement the proposed algorithms for  different ratio $m/n$ between the number $m$ of measurements and
the length $n$ of the original signal.}
\label%{tab:CAPRLsuccnumber-Comparion3}
{tab:CAPRLsuccnumber-Comparison}
	\begin{center}
		\begin{tabular}{c |c | c c c }\hline
%\backslashbox%\diagbox
$m/n$ &Algorithm
			       &Iter  &Time  &$\frac{\|{\bf x}^*-{\bf x}_0\|_2}{\|{\bf x}_0\|_2}$  \\  \hline
1     &CAPReaL-Jacobi      &994.6    &2.7496    &2.70e-1    \\
      &CAPReaL-Twisted    &1000   &3.3177     &1.12e-1       \\
      &CAPReaL-Zero        &964.4    &2.6862     &1.05e-2       \\
&CAPReaL     &960   &2.6824    &6.12e-3     \\   \hline
		%	\specialrule{0.0em}{2.0pt}{2.0pt}
1.5     &CAPReaL-Jacobi    &866.6    &2.7674  &6.60e-3      \\
      &CAPReaL-Twisted     &998.2    &3.5418  &3.90e-5       \\
       &CAPReaL-Zero       &905.5    &2.8921  &5.01e-5       \\
&CAPReaL     &863.1    &2.5233    &3.51e-5   \\   \hline
			%\specialrule{0.0em}{2.0pt}{2.0pt}
2     &CAPReaL-Jacobi       &834.3   &2.3517    &3.99e-4      \\
      &CAPReaL-Twisted     &986.5   &3.5778  &9.56e-6       \\
      &CAPReaL-Zero       &846.4    &2.8398 &2.91e-5     \\
&CAPReaL      &823.8     &2.2863  &2.95e-6   \\   \hline
			\specialrule{0.0em}{2.0pt}{2.0pt}
		\end{tabular}
	\end{center}
\end{table}

 Shown in  Table \ref{tab:CAPRLsuccnumber-Comparison} %{tab:CAPRLsuccnumber-Comparion}
are the average of the iteration number $Iter$ and
the time consumption  $Time$ in seconds to reach the stopping criterion, and the relative  reconstruction error
$\|{\bf x}^*-{\bf x}_o\|_2/\|{\bf x}_o\|_2$ between the recovered sparse signal ${\bf x}^*$ and the original sparse signal
${\bf x}_o$ over 100 trials for different ratio $m/n$ between the number $n$ of measurements and the length $n$ of the original vector, where
the original  sparsity signal ${\bf x}_o$  has sparsity $s=4$ and length $n=64$,
and   the stopping criteria in the compensation step
are the same for all algorithms,
\begin{equation*}\label{Biconvexconstraint.Compensate}
{\|{\bf Y}^k-{\bf x}^k({\bf x}^k)^{T}\|_{F}}/{\|{\bf x}^k({\bf x}^k)^{T}\|_{F}}
\leq\tilde{\varepsilon}:=10^{-5},
\end{equation*}
 cf.  \eqref{stopping.compensation},
 and  the stopping criteria for the ADMM step are
\begin{eqnarray*}\label{CAPReaL.StoppingCreterion}
&&\|{\bf x}^{k+1}-\bar{{\bf x}}^{k}\|_{\beta/\eta_1{\bf I}-\beta{\bf B}^{T}{\bf B}}^2
+\frac{2\beta\|{\bf X}_j^{k+1}-\bar{{\bf X}}_j^{k}\|_{2}^2}{\eta_2}
+\frac{2\beta\|{\bf Y}_j^{k+1}-\bar{{\bf Y}}_j^{k}\|_{2}^2}{\eta_3}\nonumber\\
&&\quad +\frac{3}{\beta}\|{\bf z}^{k+1}-\bar{{\bf z}}^{k}\|_{2}^2
+\frac{3}{\beta}\|{\bf X}^{k+1}-\bar{{\bf Z}}^{k}\|_{2}^2
\le \epsilon:=10^{-2}
\end{eqnarray*}
for the CAPReaL and CAPReaL algorithms (cf. \eqref{StoppingCreterion.e2}),
\begin{equation*}\label{CAPReaL-Twisted.StoppingCreterion}
\max\bigg\{
\frac{\|{\bf x}^{k}-\tilde{\bf x}^{k}\|_2}{1+\|{\bf x}^{k}\|_2}
\frac{\|{\bf Y}^{k}-\tilde{\bf Y}^{k}\|_2}{1+\|{\bf Y}^{k}\|_2},
\frac{\|{\bf z}^{k}-\tilde{\bf z}^{k}\|_2}{1+\|{\bf z}^{k}\|_2},
\frac{\|{\bf Z}^{k}-\tilde{\bf Z}^{k}\|_2}{1+\|{\bf Z}^{k}\|_2}
\bigg\}\le \epsilon:=10^{-2}
\end{equation*}
for the CAPReaL-Twisted algorithm (cf. \cite[Eqn. 51]{WS2017}), and
\begin{eqnarray}\label{CAPReaL-Jocabi.StoppingCreterion}
&&\frac{2\beta\|{\bf x}_j^{k+1}-{\bf x}_j^{k}\|_{2}^2}{\eta_1}
+\frac{2\beta\|{\bf X}_j^{k+1}-{\bf X}_j^{k}\|_{2}^2}{\eta_2}
+\frac{2\beta\|{\bf Y}_j^{k+1}-{\bf Y}_j^{k}\|_{2}^2}{\eta_3}\nonumber\\
&&\qquad+\frac{3-\gamma}{\beta\gamma^2}\|{\bf z}^{k+1}-{\bf z}^{k}\|_{2}^2
+\frac{3-\gamma}{\beta\gamma^2}\|{\bf Z}^{k+1}-{\bf Z}^{k}\|_{2}^2
\le \varsigma:=10^{-2}
\end{eqnarray}
foe the CAPReaL-Jacobi algorithm (cf. \cite[Lemma 2.1, Eqn 2.2]{DLPY2017}).
 Plotted in Figure \ref{figure.CAPRLsuccnumber-Comparion} is the average of
  the relative error
$\|{\bf x}^k-{\bf x}_o\|_2/\|{\bf x}_o\|_2, 1\le k\le 1000$, between the reconstructed signal ${\bf x}^k$  in the $k$-th iteration
and the original sparse signal
${\bf x}_o$ over 100 trials.
From  Table \ref{tab:CAPRLsuccnumber-Comparison} and Figure \ref{figure.CAPRLsuccnumber-Comparion}, we observe that
the proposed CAPReaL algorithm has {\bf more favorable} performance on the recovery of sparse real vectors  from
their
 quadratic measurements than  the CAPReaL-Twisted, CAPReaL-Jacobi, and CAPReaL-Zero algorithms do.

\begin{figure}[t] %[htbp]
\begin{tabular}{ccc}
\includegraphics[width=4.1cm]{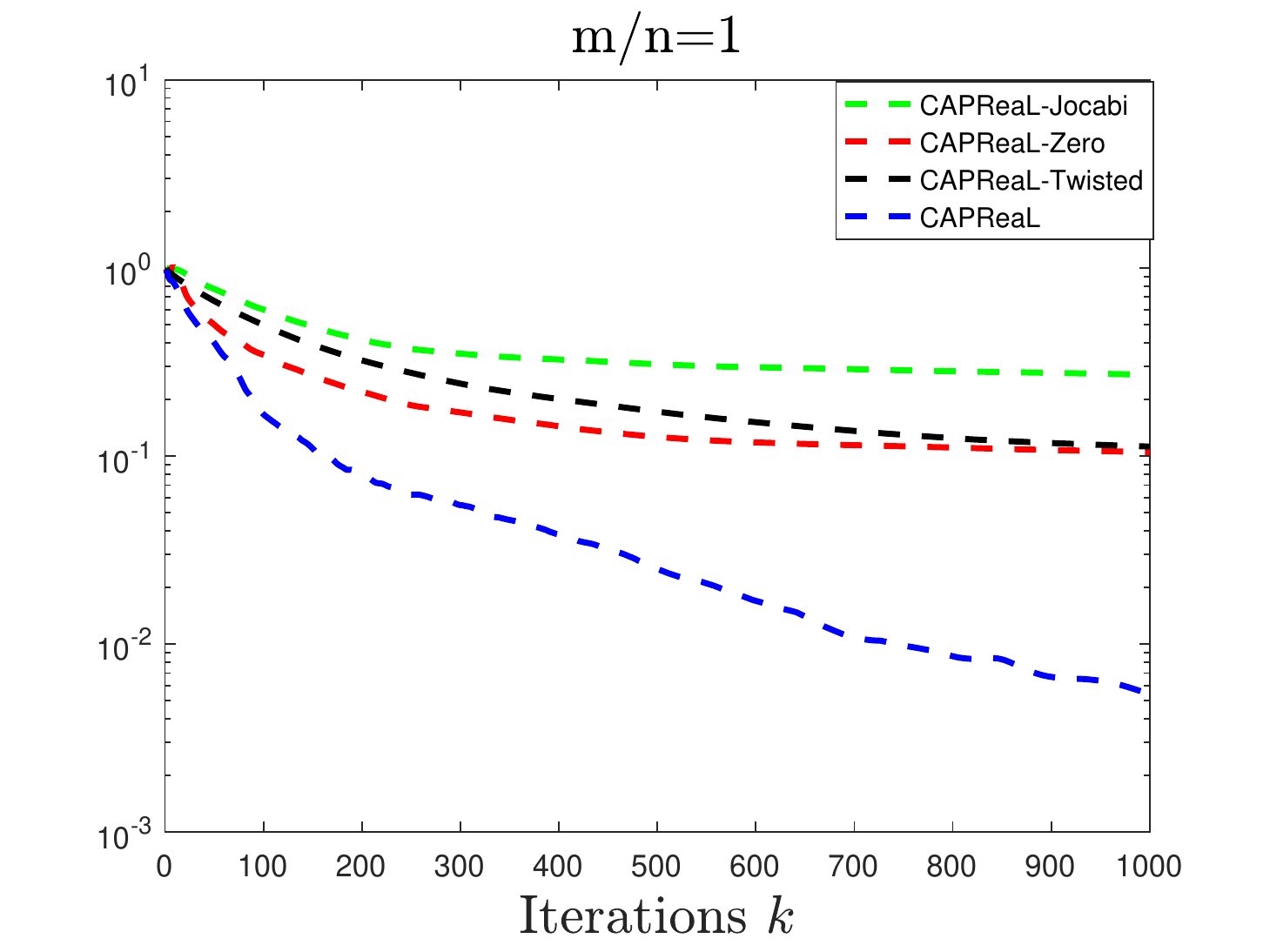}&
\includegraphics[width=4.1cm]{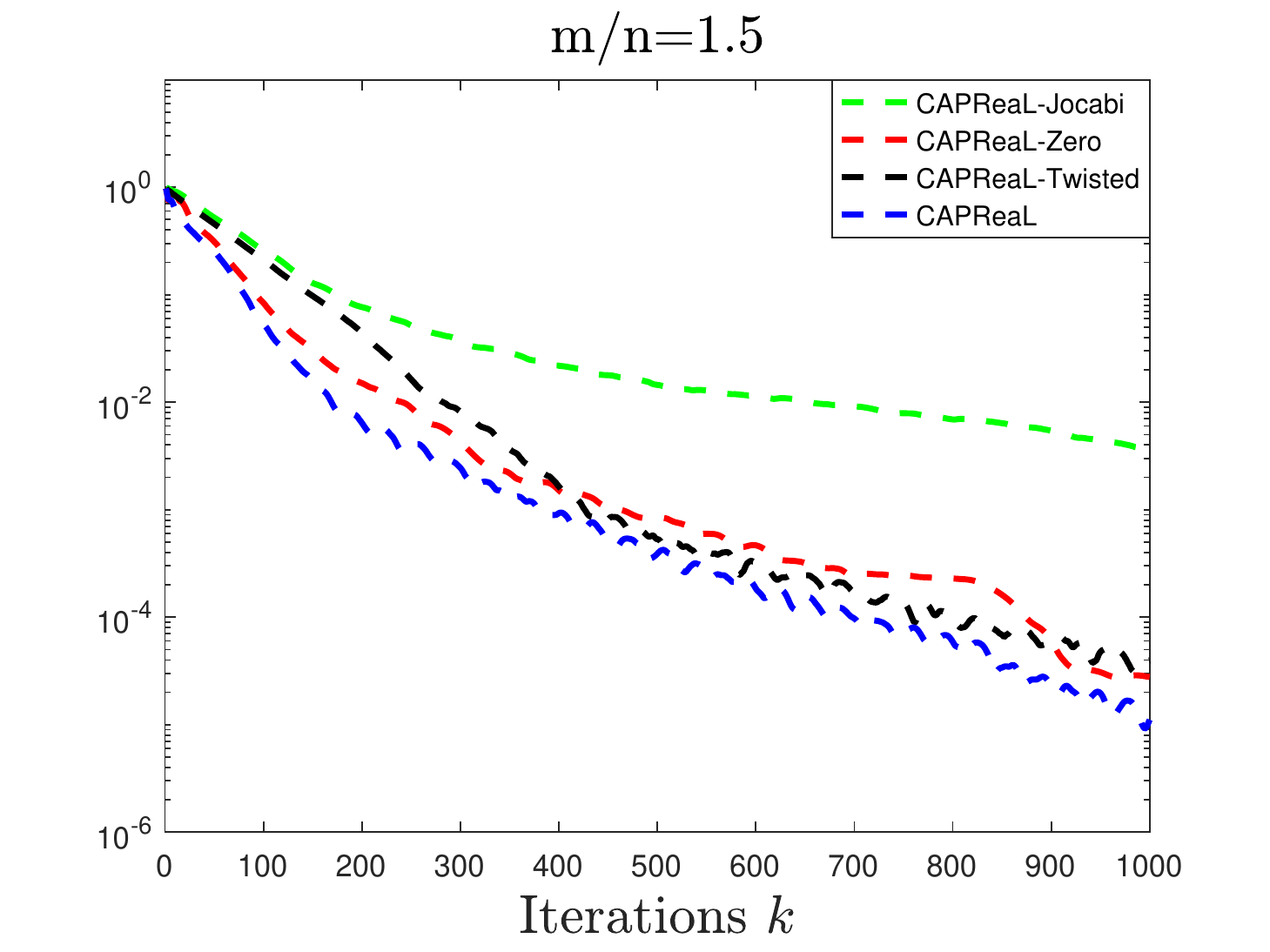}&
\includegraphics[width=4.1cm]{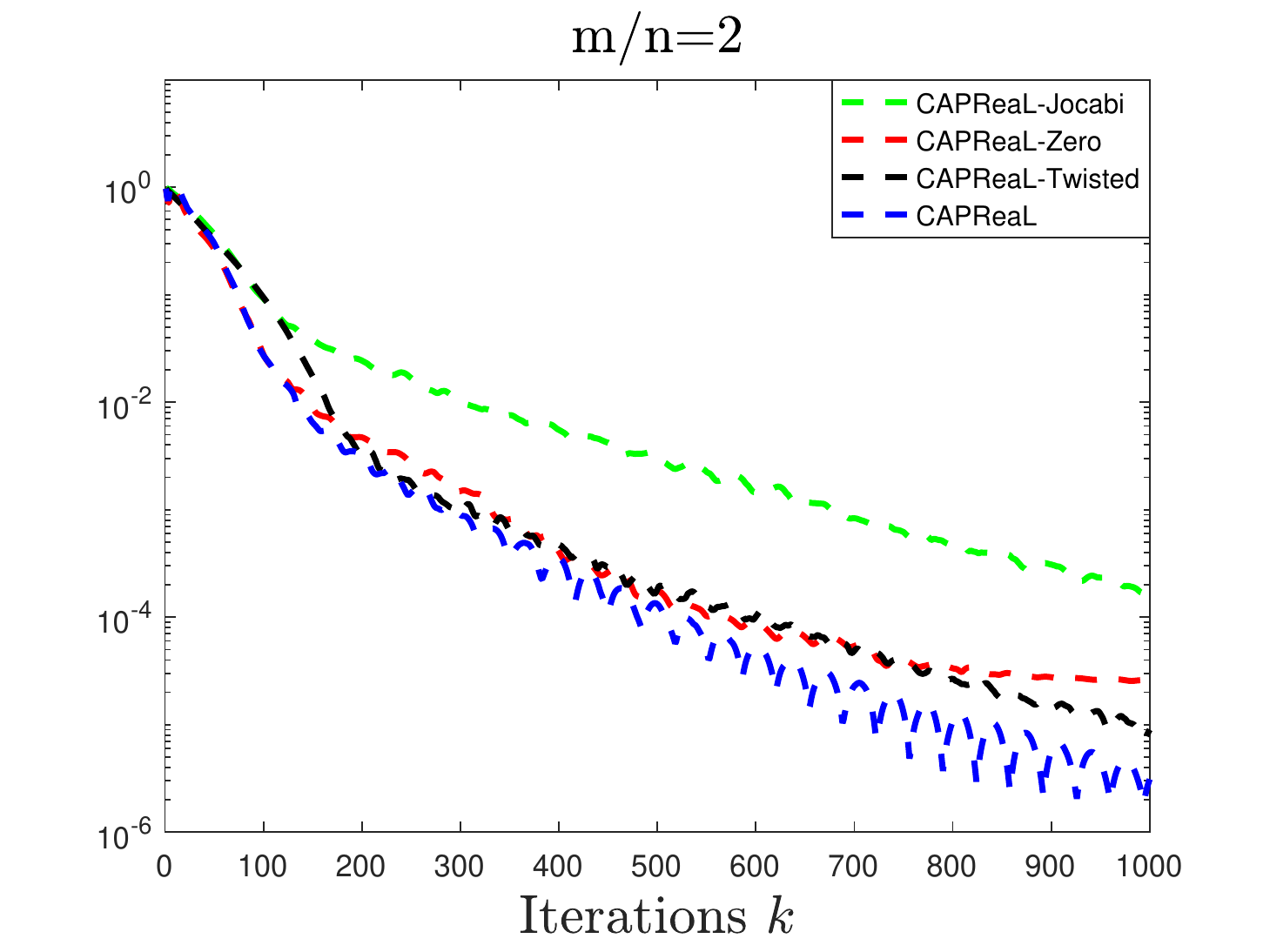}\\
\end{tabular}
\centering
\caption{\rm  The average relative error
of ${\|{\bf x}^{k}-{\bf x}_0\|_2}/{\|{\bf x}_0\|_2}$, $1\le k\le 1000$, in $k$-th iteration over 100 trials in the implementation of the proposed algorithms to
reconstruct sparse signals from their  quadratic measurements.}
\label{figure.CAPRLsuccnumber-Comparion}
\end{figure}

\subsection{Noiseless quadratic measurements}\label{noiseless.simulation}
(Sparse) phase retrieval plays an
influential role  in signal/image/speech processing  and it has received considerable attention in recent years,
%see \cite{CLS2015, CSV2013, HLVX2018, YX2015} and references therein.
see \cite{CLS2015, CSV2013, HLVX2018} and references therein.
A fundamental problem is whether and how a (sparse) vector ${\bf x}\in {\mathbb R^n}$ (or  ${\mathbb C^n}$) can be reconstructed from its quadratic measurements
 %$|{\bf A}{\bf x}|^2=(|{\bf a}_1^T {\bf x}|^2, \ldots,  |{\bf a}_m^T {\bf x}|^2)^T$,
%\begin{equation}
%\label{pd.def}
${\bf c}=|{\bf A}{\bf x}|^2=[|{\bf a}_1^T {\bf x}|^2, \ldots,  |{\bf a}_m^T {\bf x}|^2]^T $,
%\end{equation}
where ${\bf A}=[{\bf a}_1,\ldots,{\bf a}_m]^T$ is the measurement matrix.
Various algorithms
have been proposed to recover an (sparse) original signal, up to a trivial ambiguity,  from its
quadratic measurements, see  the survey paper
\cite{JEH2015} and references therein.
By \eqref{affinePRvsPR},
the recovery of  a signal ${\bf x}\in \mathbb{R}^n$ with sparsity $s$ from its affine quadratic measurement $|{\bf A}{\bf x}+{\bf b}|^2$
reduces to finding a signal $\tilde{\bf x}$ with sparsity $s+1$ and last component $1$ from its quadratic measurement
$|\tilde {\bf A}\tilde {\bf x}|^2=|{\bf A}{\bf x}+{\bf b}|^2$.
Therefore
we may adjust the   CPRL algorithm \cite{OYDS2012,LV2013}, Thresholded Wirtinger flow method (TWF) \cite{CLM2016}, CoPRAM approach \cite{JH2017}
by normalizing the last component to 1 in each iteration through  dividing the last component and  we  denote the adjusted algorithms as $\rm{CPRL}_r$, $\rm{TWF}_r$ and $\rm{CoPRAM}_r$ respectively.
Shown in Table \ref{tab:CAPRLsuccnumber-NewSPR} is  the success percentage of the proposed CAPReaL algorithm and the adjusted algorithms $\rm{CPRL}_r$, $\rm{TWF}_r$ and $\rm{CoPRAM}_r$
to recover  $s$-sparse  vectors in $\mathbb{R}^n$ from their quadratic affine measurements of size $m$,   over 100 trials, where $s=4$,   $n=64$ and $1/2\le m/n\le 2$. This indicates that
 the proposal CAPReaL method has the {\bf best} performance to recover sparse signals from
 their noiseless affine quadratic measurements, followed close behind by ${\rm CPRL_r}$ and then
 by $\rm{TWF}_r$ and $\rm{CoPRAM}_r$.
On the other hand, our simulations indicates that TWF${}_r$ and CoPRAM${}_r$  consume much less time in the implementation than
CAPReaL and ${\rm CPRL_r}$ do.

\begin{table}[t] %[!htbp]
\setlength{\tabcolsep}{9pt}
	\caption{ \rm  Success  percentage of the  $\rm{CPRL}_r$, $\rm{TWF}_r$, $\rm{CoPRAM}_r$ and CAPReaL algorithms  to recover sparse signals with 100 repeated trials for
 different ratios $m/n$ between the number $m$ of measurements and   length $n$  of the original sparse signal.
}\label{tab:CAPRLsuccnumber-NewSPR}
	\begin{center}
		\begin{tabular}{c |c c c c c c c c c c c c c c }\hline
\backslashbox%\diagbox
{Alg.} %orithm}
{$m/n$} %$\alpha_k$ \ $\backslash$ \ $m/n$
			       &0.5  &0.75  &1     &1.25  &1.5  &1.75  &2\\ \hline
         $\rm{CPRL_r}$&1      &10   &62       &87  &93   &96     &100       \\
         $\rm{TWF_r}$ &0      &3    &14     &33     &51    &62      &70\\
     $\rm{CoPRAM_r}$ &0     &1    &2      &2      &3       &4      &5 \\
%         ${\rm CAPReaL}$&0      &5.5    &46    &85.5   &99     &100    &100\\ \hline
         ${\rm CAPReaL}$&2     &19       &61     &95          &98  &100     &100 \\\hline
			\specialrule{0.0em}{2.0pt}{2.0pt}
		\end{tabular}
	\end{center}
\end{table}

\subsection{Quadratic measurements corrupted by  Gaussian  noises}

In this subsection,  we demonstrate the performance of $2$-CAPReaL algorithm
 to  recover sparse signals ${\bf x}_o\in\mathbb{R}^n$ from their
 quadratic measurements
%\eqref{simulation.data2}
%corrupted by Gaussian white noises \eqref{gaussiannoise}.
corrupted by Gaussian white noises.
For the  comparison, we compare  the proposed $2$-CAPReaL algorithm with  adjusted  CPRL-QC${}_r$, TWF${}_r$ and CoPRAM${}_r$. Here
CPRL-QC${}_r$ is  adjusted from   the  CPRL-QC  algorithm \cite{OYDS2012},
\begin{subequations}\label{CPRL-QC}
\begin{equation}\label{CPRL-QC.a}
\min_{{\bf X}\succeq {\bf O}}~\text{tr}({\bf X})+\tau\|{\bf X}\|_1
\ \
\text{subject \ to}\ ~\|\mathcal{A}({\bf X})-{\bf c}\|_2\leq \varepsilon,
\end{equation}
\end{subequations}
by normalizing  the last entries of the matrix ${\bf X}$ to 1 in each iteration by  dividing  $X_{n+1, n+1}$,
where $\tau>0$ is balancing parameter and $\varepsilon= \|{\bf e}\|_2$ is the noise bound.
We use the average of  the signal-to-noise ratio (SNR) in dB,
\begin{equation}
\text{SNR}({\bf x}^*,{\bf x}_o)=20\log_{10}\frac{\|{\bf x}_o\|_2}{\|{\bf x}^*-{\bf x}_o\|_2},
\end{equation}
over 100 independent trials
as our performance measure,  where   ${\bf x}^*$ is the reconstructed signal.
Shown in Table \ref{tab:CAPRL-QC-sparse4}
is the result of our proposed $2$-CAPReaL algorithm to recover sparse signals from their quadratic affine measurements and
the performance comparison with the CPRL-QC${}_r$, TWF${}_r$ and CoPRAM${}_r$, where the Gaussian white noise level
$\sigma=10^{-3}, 10^{-1}$.  This  shows that  for $m/n\ge 1$, the proposal $2$-CAPReaL  is  more robust against  Gaussian white noises
than the   CPRL-QC${}_r$, TWF${}_r$ and CoPRAM${}_r$  do  especially when the noise level is low,
 while for $m/n<1$  the CPRL-QC${}_r$ has best performance followed by proposal $2$-CAPReaL.

\begin{table}[t] %[!htbp]
\setlength{\tabcolsep}{5pt} %\small
	{\caption{\rm The average SNR of the  CPRL-QC${}_r$, TWF${}_r$, CoPRAM${}_r$ and $2$-CAPReaL algorithms  to recover sparse solutions
over 100 trials for
 different ratios $m/n$ between the number $m$ of measurements and the length $n$ of original signals
 and for two different Gaussian noise levels  $\sigma$. }\label{tab:CAPRL-QC-sparse4}}
	\begin{center}
		\begin{tabular}{c |c| c c c c c c c}\hline
%			\multirow{2}{*}{m/n} &\multirow{2}{*}{\quad}
%&\multicolumn{3}{c}{c=1/3} &\multicolumn{3}{c}{c=1/2}\\ \cmidrule(lr) {3-5}  \cmidrule(lr){6-8}
%		Noiselevel
$\sigma$	& \backslashbox%\diagbox
{Alg.} %orithm}
{$m/n$}        &0.5  &0.75   &1     &1.25  &1.5     &1.75  &2\\ \hline

$10^{-3}$&CPRL-QC${}_r$&5.68  &16.59  &31.03  &38.92  &40.24 &41.28 &41.93 \\
            &TWF${}_r$ &-8.72  &6.25  &23.42 &35.25 &45.65  &55.82 &57.63\\
\quad& CoPRAM${}_r$    &2.31  &4.37  &6.33  &7.13   &8.73  &8.77  &9.81 \\
     \quad&2-CAPReaL&4.58 &14.01 &40.44  &56.38  &63.27 &66.66  &67.97 \\ \hline\hline
		%  \specialrule{0.0em}{2.0pt}{2.0pt}
$10^{-1}$&CPRL-QC${}_r$&4.32    &8.19  &14.98 &24.54  &26.87   &30.43 &31.67\\
                 \quad&TWF${}_r$&-9.27  &6.89  &16.59 &25.29  &29.96 &31.05   &32.70\\
          \quad&CoPRAM${}_r$    &1.98  &3.71  &5.87  &7.36 &7.40   &8.71 &9.17 \\
     \quad& 2-CAPReaL&2.21 &5.92 &13.59  &24.66 &29.97  &32.96  &34.09 \\ \hline
		  \specialrule{0.0em}{2.0pt}{2.0pt}
		\end{tabular}
	\end{center}
\end{table}

\subsection{Quadratic measurements corrupted by impulsive noises}
For the case that quadratic  measurements are corrupted by the impulsive Cauchy noise, we will use  the $p$-CAPReaL algorithm with $p=1$ to recover sparse signals from their corrupted quadratic measurements.
 Presented in Table \ref{tab:CAPRL-LADC-Cauchy-sparsity4} are
 performances of CPRL-LADC${}_r$,
TWF${}_r$, CoPRAM${}_r$ and 1-CAPReaL algorithms to recover sparse solutions
 for different ratios $m/n$ between the number $m$ of measurements and the length $n$ of original signals, and for two different Cauchy  noise
levels $\gamma$,
where CPRL-LADC${}_r$ is modified from the
CPRL-LADC  algorithm, %  \cite{CCG2015},
\begin{subequations}\label{CPRL-LADC}
\begin{equation}\label{CPRL-LADC.a}
\min_{{\bf X}\succeq {\bf O}}~\text{tr}({\bf X})+\tau\|{\bf X}\|_1 \ \
%\end{equation}
%\begin{equation}\label{CPRL-LADC.b}
\text{subject \ to}\ ~\|\mathcal{A}({\bf X})-{\bf c}\|_1\leq \varepsilon,
\end{equation}
\end{subequations}
by adjusting the last entries of the matrix ${\bf X}$ to one in each iteration by  dividing  $X_{n+1, n+1}$, where
$\tau>0$ is balancing parameter and $\varepsilon =\|{\bf e}\|_1$ is noise bound.
%{\color{red} what is $\epsilon$ in the simulation? how to related to  the parameter $\gamma$ in the Cauchy noise?}{\color{blue}Here $\varepsilon =\|{\bf e}\|_1$  depend on the parameter $\gamma$ in the Cauchy noise. In fact, the larger the parameter $\gamma$ is, the larger the noise bound $\varepsilon$ is.}
Therefore for  the recovery of  sparse signals from their  affine quadratic
measurements corrupted by the impulsive noise of Cauchy type,
%it shows that
the CPRL-LADC${}_r$ and the  proposed 1-CAPReaL have much better performance than  TWF${}_r$ and CoPRAM${}_r$  do,
the CPRL-LADC${}_r$ achieves higher SNR  than the 1-CAPReaL does when we have less measurements   and
 the 1-CAPReaL does better job than CPRL-LADC${}_r$  does when we have more measurements.

\begin{table}[t] %[!htbp]
\setlength{\tabcolsep}{5pt}%\small
	{\caption{\rm
The average SNR of the  CPRL-LADC${}_r$, TWF${}_r$, CoPRAM${}_r$ and $1$-CAPReaL algorithms  to recover sparse signals from  their quadratic affine measurements corrupted by Cauchy noises
over  100 trials for
 different ratios $m/n$
 and for two different  Cauchy noise levels  $\gamma$.
 }\label{tab:CAPRL-LADC-Cauchy-sparsity4}}
	\begin{center}
		\begin{tabular}{c| c| c c c  c c c c  c c c c  c c c c}\hline
%			\multirow{2}{*}{m/n} &\multirow{2}{*}{\quad}
%&\multicolumn{3}{c}{c=1/3} &\multicolumn{3}{c}{c=1/2}\\ \cmidrule(lr) {3-5}  \cmidrule(lr){6-8}
			$\gamma $ %Noise level	
& \backslashbox%\diagbox
{Alg.} %orithm}
{$m/n$}            &0.5     &0.75    &1      &1.25   &1.5     &1.75   &2 \\
\hline %&2.25  &2.5  &2.75 &3 &3.25 &3.5 &3.75 &4\\ \hline
$10^{-4}$ & CPRL-LADC${}_r$ &3.61   &12.17  &34.61  &47.51 &53.25   &55.03  &55.56\\
       \quad& TWF${}_r$&-30.65  &-21.83  &-8.91  &8.64 &23.28   &32.91  &48.53 \\
       \quad& CoPRAM${}_r$&1.82 &3.68  &6.61  &7.15   &8.05    &9.26   &10.07 \\
          \quad& 1-CAPReaL &5.15 &21.60 &55.27 &70.84 &75.66 &77.63 &77.80 \\
\hline  \hline
		  %\specialrule{0.0em}{2.0pt}{2.0pt}
$10^{-2}$ &CPRL-LADC${}_r$&3.44   &10.90  &21.18 &29.75 &30.74   &32.14 &32.37\\
           \quad&TWF${}_r$     &-30.64  &-22.85 &-11.65  &1.91  &12.27  &16.83  &22.96\\
           \quad&CoPRAM${}_r$ &1.47   &3.65   &5.12  &5.43   &6.23   &7.01  &8.25 \\
          \quad& 1-CAPReaL &1.42 &4.10  &10.59  &23.11  &32.34 &40.29 &42.64 \\
          \hline
		  \specialrule{0.0em}{2.0pt}{2.0pt}
		\end{tabular}
	\end{center}
\end{table}

\subsection{Quadratic measurements corrupted by  bounded noises}
\label{boundednoise.simulation}

In this subsection, we  approximate the true
sparse signal ${\bf x}_o$ when its quadratic  measurements are corrupted by
a uniformly distributed noise with different noise bound $\delta$.  Shown in Table \ref{tab:CAPRL-IC-Uniformlydistribution-sparsity4-new} is the performances of  CPRL-IC${}_r$, TWF${}_r$, CoPRAM${}_r$ and $\infty$-CAPReaL, where CPRL-IC${}_r$ is modified from the
CPRL-IC  algorithm,
\begin{subequations}\label{CPRL-IC}
\begin{equation}\label{CPRL-IC.a}
\min_{{\bf X}\succeq {\bf O}}~\text{tr}({\bf X})+\tau\|{\bf X}\|_1
\end{equation}
\begin{equation}\label{CPRL-IC.b}
\text{subject \ to}~\|\mathcal{A}({\bf X})-{\bf c}\|_{\infty}\leq \delta,
\end{equation}
\end{subequations}
by adjusting the last component of the matrix ${\bf X}$ to one in each iteration by  dividing  $X_{n+1,n+1}$, where
$\delta =\|{\bf e}\|_{\infty}$ is the noise bound.
These results indicate that  the  proposed $\infty$-CAPReaL has much better performance than CPRL-IC${}_r$, TWF${}_r$ and CoPRAM${}_r$ do
 when the noise level is low,  %i.e., $\delta=10^{-3},10^{-2}$,
while the TWF${}_r$ achieves higher SNR  than the $\infty$-CAPReaL does when the noise level is high. %, i.e., $\delta=10^{-1},1$.}

\begin{table}[t] %[!htbp]
\setlength{\tabcolsep}{5pt} %\small
	{\caption{\rm The average SNR of the  CPRL-IC${}_r$, TWF${}_r$, CoPRAM${}_r$ and $\infty$-CAPReaL algorithms  to recover sparse signals from  their quadratic affine measurements corrupted by bounded noises over  100 trials for different ratios $m/n$ and for four different bounded noise levels  $\delta$, where the sparsity is $s=4$ and vector length is  $n=64$,  and the quadratic affine measurements ${\bf b}=|{\bf A}{\bf x}_0|^2+{\bf e}$ with ${\bf e}$ is the uniformly distribution noise with noise bound $\delta$. }\label{tab:CAPRL-IC-Uniformlydistribution-sparsity4-new} }
	\begin{center}
		\begin{tabular}{c| c| c c c c c c c }\hline
%			\multirow{2}{*}{m/n} &\multirow{2}{*}{\quad}
%&\multicolumn{3}{c}{c=1/3} &\multicolumn{3}{c}{c=1/2}\\ \cmidrule(lr) {3-5}  \cmidrule(lr){6-8}
$\delta$ %	Noise bound 	
& \backslashbox%\diagbox
{Alg. %orithm
}{$m/n$}              &0.5  &0.75    &1    &1.25   &1.5    &1.75 &2      \\ \hline
$10^{-3}$ &{CPRL-IC}${}_r$&3.43  &18.92   &35.54   &47.67   &52.69   &54.81
       &55.22\\
      \quad&$\rm{TWF_r}$ &-15.53 &1.51  &23.88   &38.12   &47.26   &54.16  &57.52 \\
&$\rm{CoPRAM_r}$ &1.86  &4.37   &5.37   &7.39   &7.45    &9.54 &10.02  \\
    \quad& $\infty$-CAPReaL&5.19  &15.10  &36.55  &55.98   &63.89 &65.94  &69.03\\
    \hline\hline
   $10^{-2}$  \quad &{CPRL-IC}${}_r$&3.76  &12.91   &32.04   &40.97   &47.15   &48.82
    &48.83\\
    \quad&$\rm{TWF_r}$ &-17.75 &1.42  &22.35   &33.28   &40.45  &45.35 &50.03 \\
   &$\rm{CoPRAM_r}$ &1.75  &4.13   &6.38   &7.12   &7.69   &9.21 &9.68  \\
    \quad&$\infty$-CAPReaL&4.66  &13.43  &28.45  &41.38   &45.79 &48.39 &51.95\\
    \hline\hline
$10^{-1}$ \quad &{CPRL-IC}${}_r$&2.72 &9.55   &16.70   &24.55   &27.37  &28.40 &30.46\\
          \quad&$\rm{TWF_r}$ &-17.41 &-1.12  &18.12   &28.42   &32.65    &33.32
           &35.65  \\
   &$\rm{CoPRAM_r}$ &2.33  &3.69  &5.03   &6.48  &7.46    &8.54 &9.02   \\
    \quad&$\infty$-CAPReaL&3.93 &13.33  &19.31  &28.68 &30.80  &33.35 &34.97 \\
    \hline\hline
    %\specialrule{0.0em}{2.0pt}{2.0pt}
$1$ \quad &{CPRL-IC}${}_r$&1.82 &3.04   &4.11   &6.07   &7.80   &8.60 &10.43\\
    \quad&$\rm{TWF_r}$ &-17.19 &-3.47  &9.29   &15.07   &17.25  &17.99
                    &20.32  \\
 &$\rm{CoPRAM_r}$ &0.02  &1.46   &2.44   &3.25   &4.57  &4.38 &5.87   \\
     \quad&$\infty$-CAPReaL &2.05 &4.14 &6.19 &8.87   &10.68   &11.88 &14.87 \\
         \hline
    \specialrule{0.0em}{2.0pt}{2.0pt}
		\end{tabular}
	\end{center}
\end{table}

\section*{Acknowledgments}
%{\color{blue}Prof Sun and Chen, Please add new project number if you have a new one. }
The authors would like to  thank  Professors  Zaiwen Wen, Anthony Man-Cho So, and Drs. Bin Gao, Huanmin Ge and Di Yang  for their help in the preparation of this paper.

%%%%%%%%%%%%%%%%%%%%%%%%%%%%%%%%%%%   Bibliography  %%%%%%%%%%%%%%%%%%%%%%%%%%%%%%%%%%%%%%%%

\end{document}